\documentclass[11pt]{article}

\usepackage{subfig}
\usepackage{amsmath, amsthm}
\usepackage{bm}
\usepackage{enumerate}
\usepackage{tikz}
\usepackage{appendix}
\usepackage{amsfonts}

\numberwithin{equation}{section}

\newtheoremstyle{plainNoItalics}{}{}{\normalfont}{}{\bfseries}{.}{ }{}

\theoremstyle{plain}
\theoremstyle{plainNoItalics}
\newtheorem{rem}{Remark}[section]
\newtheorem{exa}{Example}[section]
\newtheorem{lemma}{Lemma}[section]
\newtheorem{Theorem}{Theorem}[section]
\newtheorem{Algorithm}{Algorithm}[section]
\newtheorem{Prop}{Proposition}[section]
\newtheorem{Def}{Definition}[section]

\begin{document}

\baselineskip=1.38pc
\linespread{1.0}

%=============  title  ==============
\begin{center}
\renewcommand{\thefootnote}{\fnsymbol{footnote}}%footnote设为特殊字符 special 
{\bf
Provably convergent Newton--Raphson methods for recovering primitive variables with applications to  physical-constraint-preserving Hermite WENO schemes for relativistic hydrodynamics\footnotemark[1]
}
\footnotetext[1]{The work of the first and second authors is partially supported by National Key R\&D Program of China (Grant No.~2022Y-FA1004500). The work of the third author is partially supported by Shenzhen Science and Technology Program (Grant No.~RCJC20221008092757098) and National Natural Science Foundation of China (Grant No.~12171227).}
\renewcommand{\thefootnote}{\arabic{footnote}}%footnote设为数字
%% finite volume
\vspace{6mm}

Chaoyi Cai\footnotemark[1],~~ Jianxian Qiu\footnotemark[2],~~ Kailiang Wu\footnotemark[3]

\footnotetext[1]{School of Mathematical Sciences, Xiamen University, Xiamen, Fujian 361005, P.R.~China. E-mail: caichaoyi@stu.xmu.edu.cn.}
\footnotetext[2]{School of Mathematical Sciences and Fujian Provincial Key Laboratory of Mathematical Modeling and High-Performance Scientific Computing, Xiamen University, Xiamen, Fujian 361005, P.R.~China. E-mail: jxqiu@xmu.edu.cn.}
\footnotetext[3]{Department of Mathematics and SUSTech International Center for Mathematics, Southern University of Science and Technology, and National Center for Applied Mathematics Shenzhen (NCAMS), Shenzhen 518055, P.R.~China. E-mail: wukl@sustech.edu.cn.}
\end{center}

\vspace{6mm}

\noindent
{\bf Abstract.} The relativistic hydrodynamics (RHD) equations have three crucial intrinsic physical constraints on the primitive variables: positivity of pressure and density, and subluminal fluid velocity. However, numerical simulations can violate these constraints, leading to nonphysical results or even simulation failure. 
Designing genuinely physical-constraint-preserving (PCP) schemes is very difficult, as the primitive variables cannot be explicitly reformulated using conservative variables due to relativistic effects. In this paper, we propose three efficient Newton--Raphson (NR) methods for robustly recovering primitive variables from conservative variables. Importantly, we rigorously prove that these NR methods are always convergent and PCP, meaning they preserve the physical constraints throughout the NR iterations. 
The discovery of these robust NR methods 
and their PCP convergence analyses are highly nontrivial and technical. 
As an application, we apply the proposed NR methods to 
 design PCP finite volume Hermite weighted essentially non-oscillatory (HWENO) schemes for solving the RHD equations. 
Our PCP HWENO schemes incorporate high-order HWENO reconstruction, a PCP limiter, and strong-stability-preserving time discretization. We rigorously prove the PCP property of the fully discrete schemes using convex decomposition techniques. Moreover, we suggest the characteristic decomposition with rescaled eigenvectors and scale-invariant nonlinear weights to enhance the performance of the HWENO schemes in simulating large-scale RHD problems. 
Several demanding numerical tests are conducted to demonstrate the robustness, accuracy, and high resolution of the proposed PCP HWENO schemes and to validate the efficiency of our NR methods.

\vspace{6mm}


\newpage

\section{Introduction}
	Relativistic fluid flows widely appear in many astrophysical phenomena, such as gamma-ray bursts, supernova explosions, extragalactic jets, and accretion onto black holes. 
	When the velocity of fluid is close to the speed of light, 
	the classic compressible Euler equations are no longer valid due to 
	the relativistic effect. 
	The governing equations of the $d$-dimensional special relativistic hydrodynamics (RHD) can be written into a system of conservation laws as follows 
\begin{equation}\label{eq:intro_rhd}
\frac{\partial \bm{U}}{\partial t}+\sum_{i=1}^{d}\frac{\partial \bm{F}_i\left(\bm{U}\right)}{\partial x_i}={\bf 0},
\end{equation}
where the conservative vector and flux vectors are given by 
\begin{align}\label{eq:intro_conservvecter}
\bm{U}&=\left(D,m_1,\dots,m_d,E\right)^\top,
\\
\label{eq:intro_fluxvecter}
	\bm{F}_i&=\left(Dv_i,m_1v_i+p\delta_{1,i},\dots,m_dv_i+p\delta_{d,i},m_i\right)^\top.
\end{align}
The conservative variables $D,\bm{m}=(m_1,\dots,m_d)$, and $E$ represent the mass density, the momentum vector, and the total energy, respectively. Let $\bm{Q}=\left(\rho,\bm{v},p\right)^\top$ denote the primitive variable vector, where $\bm{v}=(v_1,\dots,v_d)$ is the fluid velocity vector, and $\rho$, $p$ represent the rest-mass density and the kinetic pressure, respectively. Then $\bm{U}$ can be calculated from $\bm{Q}$ by 
\begin{eqnarray}\label{eq:intro_conver2conserv}
	\left\{
	\begin{aligned}[c]
		&D=\rho W, \\
		&\bm{m}=DhW\bm{v},\\
		&E=DhW-p,\\
		&h=1+e+\frac{p}{\rho},\\
		&W=1/\sqrt{1-|{\bm v}|^2},
	\end{aligned}
	\right.
\end{eqnarray}
along with the ideal equation of state (EOS) considered in this paper:  
\begin{equation*}\label{cond:gamma law}
	e=\frac{p}{\left(\gamma-1\right)\rho},
\end{equation*}
 where $\gamma \in \left(1,2\right]$ is the adiabatic index. In (\ref{eq:intro_conver2conserv}), the velocity is normalized such that the speed of light is $1$, $W$ denotes the Lorentz factor, $h$ is the specific enthalphy, and $e$ represents the specific internal energy. 
 As seen from \eqref{eq:intro_conver2conserv} and \eqref{eq:intro_fluxvecter}, both the conservative vector ${\bm U}$ and the flux ${\bm F}_i$ are explicit functions of $\bm Q$. 
 However, neither ${\bm F}_i$ nor $\bm Q$ can be explicitly expressed by ${\bm U}$. 
 In the computations, in order to evaluate ${\bm F}_i({\bf U})$ and the eigenvalues/eigenvectors of its Jacobian matrix $\frac{\partial {\bm F}_i({\bf U})}{\partial {\bm U}}$, we have to first recover the corresponding primitive variables $\bm Q$ from ${\bf U}$.  
 This recovery procedure is complicated and typically requires one to numerically solve nonlinear algebraic equations. 
 %; see section \ref{sec:Pressure recovering algorithms} for some details
 Several recovery algorithms were proposed in the past decades. 
 Most algorithms calculate one intermediate variable first and then compute other primitive variables using this intermediate variable. An algorithm based on the baryon number density as the intermediate variable was constructed in \cite{dolezal1995relativistic}. 
 Several Newton--Raphson (NR) and analytical algorithms were proposed in \cite{riccardi2008primitive} with 
 a variety of intermediate variables such as the pressure, the velocity, and the Lorentz factor. In \cite{ryu2006equation}, three analytical algorithms were designed for recovering the primitive variables for three different EOS. 
 However, the convergence of existing NR algorithms is not guaranteed in theory, while the analytical algorithms often suffer from low accuracy and high computational cost. Recently, three robust linearly convergent iterative recovery algorithms were studied in \cite{chen2022physical}.

% , for example, solve the following nonlinear equation suggested in \cite{2015High}: 
%\begin{equation*}
%	\Phi(p) :=\frac{p}{\gamma-1}-E+\frac{ | \bm{m} |^2}{E+p}+D\sqrt{1-\frac{ | \bm{m} |^2}{(E+p)^2}}=0
%\end{equation*} 
%to obtain the value of $p$, then $\bm{v}$ and $\rho$ can be sequentially calculated by $\bm{v}=\bm{m}/(E+p)$ and 
%$\rho = D \sqrt{ 1- |\bm v |^2 }$. 

It is often extremely difficult to obtain the analytical solutions of the RHD system \eqref{eq:intro_rhd} due to the high nonlinearity. As a result, numerical simulation has become an effective and practical approach to study RHD. Various numerical methods have been developed for solving the RHD equations over the past few decades, including but not limited to finite difference methods \cite{wilson2007relativistic,dolezal1995relativistic,zhang2006ram,radice2012thc,he2012adaptive,2015High}, finite volume methods \cite{mignone2005hllc,tchekhovskoy2007wham,balsara2016subluminal}, discontinuous Galerkin (DG) methods \cite{radice2011discontinuous,zhao2013runge,teukolsky2016formulation,kidder2017spectre}, and so on. The interested readers are also referred to the review papers %,chen2021second
 \cite{font2008numerical,marti2003numerical,marti2015grid} for more related developments in this direction.

The RHD equations (\ref{eq:intro_rhd}) are a nonlinear hyperbolic system of conservation laws, which can result in discontinuities in the entropy solution even with smooth initial conditions. As well-known, this type of equations are difficult to solve due to the possibility of generating numerical oscillations near the discontinuities. Series of high-order numerical schemes based on essentially non-oscillatory (ENO) and weighted ENO (WENO) methods have been developed for solving such hyperbolic conservation laws. The history of ENO and WENO schemes can be traced back to 1985, when Harten introduced the total variation diminishing (TVD) concept \cite{harten1983high}, which formed the basis of the ENO schemes  \cite{harten1987preliminary,harten1987uniformly}. In 1994, Liu, Osher, and Chan proposed the first WENO scheme \cite{liu1994weighted}. Jiang and Shu then improved upon it by giving the framework for the design of the smooth indicators and nonlinear weights in 1996 \cite{jiang1996efficient}. This kind of nonlinear weights enable the attainment of uniformly higher-order accuracy in smooth solutions, while simultaneously avoiding the emergence of numerical oscillations in discontinuous solutions. Since then, a lot of researches have sprung up on WENO schemes, including but not limited to  \cite{1999Weighted,levy1999central,shi2002technique,2011High}. Recently, Zhu and Qiu \cite{zhu2016new} proposed a simpler WENO construction, which is a convex combination of a fourth-degree polynomial and two linear polynomials with any three positive linear weights that sum to one. To address the issue of wide stencils in WENO schemes, Qiu and Shu developed Hermite WENO (HWENO) schemes  \cite{qiu2004hermite,qiu2005hermite}, which can achieve higher-order accuracy with the same reconstruction stencils as WENO schemes. 
To reduce computational cost in WENO reconstruction, several hybrid WENO schemes \cite{costa2007multi,li2010hybrid,zhao2019new} have been proposed. These schemes use linear schemes directly in smooth regions while still utilizing WENO schemes in the discontinuous regions. Recently, Zhao, Chen, and Qiu  proposed several new finite volume HWENO schemes \cite{zhao2020hybrid,zhao2020hermite}. Among them, the hybrid HWENO schemes \cite{zhao2020hermite} incorporate the thoughts behind hybrid schemes and the limiters in DG  methods, making this kind of schemes more efficient and easier to implement.

Although the WENO and HWENO schemes are stable in many numerical simulations, they are generally not physical-constraint-preserving (PCP), namely, they do not always preserve the intrinsic physical constraints: for the RHD equations \eqref{eq:intro_rhd}, such constraints include the positivity of pressure and rest-mass density, as well as the subluminal constraint on the fluid velocity. 
For the RHD equations \eqref{eq:intro_rhd}, all the admissible states ${\bm U}$ satisfying these constraints form the following set 
\begin{equation}\label{eq:G0}
	\mathcal{G}_0=\left\{\bm{U}=\left(D,\bm{m},E\right)^\top:~\rho(\bm{U})>0,~ p(\bm{U})>0,~ |{\bm v}(\bm{U})|< 1 \right\}.
\end{equation}
In fact, preserving the numerical solutions in this set $\mathcal{G}_0$ is essential, because if any of these constraints 
are violated in the numerical computations, the corresponding discrete equations would become ill-posed and the simulation would break down. 
As we mentioned, in the RHD case, the primitive variables cannot be explicitly expressed by $\bm U$, so that the three functions $\rho(\bm{U})$, $p(\bm{U})$, and ${\bm v}(\bm{U})$ in \eqref{eq:G0} are implicit. This makes the study of PCP schemes for RHD nontrivial and more difficult than the non-relativistic hydrodynamics.   
In recent years, there are lots of efforts on developing high-order PCP or bound-preserving schemes 
for hyperbolic conservation laws via two types of limiters. The first type of limiter can be used in the finite volume and DG frameworks and was first proposed by Zhang and Shu to keep the maximum-principle-preserving property for scalar conservation laws \cite{zhang2010maximum} and the positivity-preserving property for the non-relativistic Euler equations \cite{zhang2010positivity}. The second type of limiter is a flux-correction limiter that modifies the high-order flux with a first-order PCP flux to obtain a new flux with high accuracy and PCP property; cf.~\cite{xu2014parametrized,xiong2016parametrized,hu2013positivity}. The interested readers are referred to the reviews \cite{xu2017bound,shu2016bound} for more related works. 

The first PCP work on RHD was made in \cite{2015High}, which provided a rigorous proof for the PCP property of the local Lax--Friedrichs flux and proposed the PCP finite difference WENO schemes for RHD. 
The following explicit equivalent form of the admissible state set \eqref{eq:G0} was also proved in \cite{2015High} 
\begin{equation}\label{eq:G1}
	\mathcal{G}=\left\{\bm{U}=\left(D,\bm{m},E\right)^\top:~ D>0,~ g(\bm{U}):=E-\sqrt{D^2+| \bm{m} |^2}>0\right\},
\end{equation}
where $g(\bm{U})$ is a concave function, and moreover, $\mathcal{G} = \mathcal{G}_0$ is a convex set. 
Qin, Shu, and Yang developed a bound-preserving DG method for RHD in \cite{qin2016bound}. The PCP Lagrangian finite volume schemes were designed in \cite{ling2019physical}. A PCP central DG method was proposed for RHD  with a general EOS in \cite{wu2016physical}. The framework for designing provably high-order PCP methods for general RHD was established in \cite{wu2017design}. 
More recently, a minimum principle on specific entropy and high-order accurate invariant region preserving numerical methods were studied for RHD in \cite{wu2021minimum}. 
A PCP finite volume WENO method was developed on unstructured meshes in \cite{chen2022physical}, where three robust algorithms were also introduced for recovering the primitive variables. 
Besides, the PCP schemes were also studied for the relativistic magnetohydrodynamics (MHD) in \cite{wu2017admissible,wu2021provably}. Most notably, the theoretical analyses in  \cite{wu2017admissible,wu2018positivity} based on the geometric quasilinearization technique \cite{WuShu2021GQL} revealed that the PCP property of MHD schemes is strongly connected with a discrete divergence-free condition on the magnetic field. 
In addition,  a flux limiter was proposed in \cite{radice2014high} to preserve the positivity of rest-mass density, and a subluminal reconstruction technique was designed for relativistic MHD in \cite{balsara2016subluminal}.  
As we have mentioned,  
neither the flux ${\bm F}_i$ nor $\bm Q$ can be explicitly expressed by ${\bm U}$ in the RHD case. 
In order to evaluate the flux and the eigenvalues/eigenvectors of its Jacobian matrix, we have to recover the primitive variables $\bm Q$ from ${\bf U}$. 
This recovery procedure requires solving a nonlinear algebraic equation by some root-finding algorithms. 
Although the PCP property $\bm{U} \in \mathcal{G}$ guarantees the existence and uniqueness of the corresponding physical primitive variables in theory \cite{2015High}, it however does not ensure  
the convergence of the root-finding algorithms, nor the physical constraints of the computed primitive variables obtained by the root-finding algorithms.

This paper aims to design genuinely PCP schemes. We first propose three efficient PCP NR methods for robustly recovering primitive variables from conservative variables. 
Importantly, we rigorously prove that these NR methods are always convergent and PCP, meaning they preserve the physical constraints throughout the NR iterations. 
The discovery of these robust NR methods 
and their PCP convergence analyses are highly nontrivial 
and become the most significant contribution of this work. 
In particular, our analyses involve careful and detailed investigations of the convexity/concavity structures of the iterative functions.  
As applications, we apply the proposed NR methods to 
develop robust and efficient high-order PCP finite volume HWENO schemes for solving the RHD equations \eqref{eq:intro_rhd}.  Our approach builds upon the NR methods, the hybrid high-order HWENO reconstruction proposed in \cite{zhao2020hermite}, a PCP limiter, and strong-stability-preserving Runge--Kutta method for time discretization. 
We rigorously prove the PCP property of our HWENO schemes under a CFL condition, by using the Lax--Friedrichs splitting property and convex decomposition techniques. 
Moreover, we suggest 
the rescaled eigenvectors for characteristic decomposition and the scale-invariant nonlinear weights to 
address the issue of numerical oscillations arising from the wide range of variable spans in the RHD equations and enhance the performance of the HWENO schemes in simulating large-scale RHD problems. 
We implement the proposed one-dimensional (1D) and two-dimensional (2D) PCP HWENO schemes, and 
 provide extensive challenging numerical tests to demonstrate the robustness, accuracy, and high resolution of our PCP HWENO schemes and to validate the efficiency of our NR methods.

This paper is organized as follows. 
Section 2 proposes three efficient PCP convergent NR methods for recovering primitive variables and provides the theoretical analysis on their convergence and PCP property. 
Section 3 introduces the 1D PCP finite volume HWENO method for RHD and provides the theoretical analysis of the PCP property. Section 4 extends the method and analysis to the RHD systems in two dimensions and the cylindrical coordinates. Section 5 conducts several numerical experiments to  demonstrate the PCP property, accuracy, and effectiveness of the PCP HWENO schemes and NR methods. Section 6 gives the conclusion of this paper.

\section{Efficient PCP convergent Newton--Raphson methods for recovering primitive variables}\label{sec:Pressure recovering algorithms}

In sections \ref{sec:1D HWENO scheme} and \ref{sec:2D HWENO scheme}, we will develop a PCP HWENO method that ensures the numerical solutions of the conservative variables ${\bf U}$ in the admissible state set $\mathcal{G}$. 
According to the analysis in \cite{2015High}, when $\bm{U} \in \mathcal{G}$, the corresponding primitive variables $\bm{Q}=(\rho,\bm{v},p)^\top$ are uniquely determined and satisfy the physical constraints  
\begin{equation}\label{key433}
	\rho >0, \quad p>0, \quad |\bm{v}|< 1. 
\end{equation}

When computing the fluxes $\bm{F}_i(\bm{U})$ with the conservative vectors $\bm{U} \in \mathcal{G}$, it is necessary to recover the primitive variables $\bm{Q}=(\rho,\bm{v},p)^\top$ from $\bm{U}$. 
This recovery procedure requires solving a nonlinear algebraic equation by some root-finding algorithms, because there is no explicit expressions for $\bm{Q}$ in terms of $\bm{U}$. 
In the past decades, many algorithms recovering the primitive variables were proposed; see  
\cite{riccardi2008primitive} for a review. 

The PCP property $\bm{U} \in \mathcal{G}$ guarantees the uniqueness of the corresponding physical primitive variables in theory.  
However, it does not ensure  
the convergence of the root-finding algorithms, nor the physical constraints \eqref{key433} for the primitive variables computed by the root-finding algorithms. 
We would like to seek iterative root-finding algorithms that are convergent and PCP, namely, preserve the 
physical constraints \eqref{key433} during the iteration process. In particular, we 
are interested in recovering the pressure $p({\bm U})$ first; 
once $p({\bm U})$ is recovered, the velocity vector and the density can be calculated sequentially as follows:
\begin{equation}\label{key321}
	\bm{v} ({\bm U})=\bm{m}/(E+p({\bm U})), \qquad \rho ({\bm U})=D\sqrt{1-| \bm{v}({\bm U}) |^2}.
\end{equation}
We observe that, for a given ${\bm U} \in {\mathcal G}$ satisfying $D>0$ and $E>\sqrt{D^2+|{\bm m}|^2}$, if the recovered $p({\bm U})>0$, then the calculation \eqref{key321} leads to 
$|\bm{v}({\bm U})|<1$ and  $\rho ({\bm U})>0$. Hence we propose the following definitions. 

\begin{Def}
	Given ${\bm U}\in {\mathcal G}$, 
	a pressure-recovering algorithm is called PCP, if all the approximate pressures in the iterative sequence $\left\{p_{n}\right\}_{n\ge 1}$ are always positive.
\end{Def}

\begin{Def}
	Given ${\bm U}\in {\mathcal G}$, 
	a pressure-recovering algorithm is called convergent, if  the iterative sequence $\left\{p_{n}\right\}_{n\ge 1}$ converges to the physical pressure $p({\bm U})$, namely, 
	$\lim\limits_{n\to +\infty} p_n = p({\bm U})$. 
\end{Def} 

In \cite{chen2022physical}, three PCP convergent pressure-recovering algorithms were developed. 
Those algorithms are based on bisection, fixed-point iteration, and a hybrid iteration combining them, 
respectively. Consequently, those algorithms  in \cite{chen2022physical} have only a first order of linear convergence.
In this section, we develop three faster pressure-recovering algorithms, 
which are Newton--Raphson (NR) method and have a convergence order of 2. 
%, while the other has a convergence order of $\frac{\sqrt{5}+1}{2}$. 
Furthermore, we will rigorously prove that the three proposed NR methods are both PCP and convergent.

\subsection{NR-I method: monotonically convergent PCP NR iteration}\label{algor:newnewton}

As shown in \cite{2015High}, the true physical pressure $p({\bm U})$ corresponding to a conservative vector 
${\bm U}=(D, {\bm m}, E)^\top$ satisfies the following nonlinear algebraic equation: 
\begin{equation}\label{eq:primitive}
	\Phi(p) :=\frac{p}{\gamma-1}-E+\frac{| \bm{m} |^2}{E+p}+D\sqrt{1-\frac{| \bm{m} |^2}{(E+p)^2}}=0. 
\end{equation}
When ${\bm U} \in {\mathcal G}$, we can show that the function $\Phi(p)$ is strictly monotonically increasing with respect to $p \in [0, +\infty )$. Moreover, $\Phi(0)<0$ and 
$\lim\limits_{p\to +\infty}\Phi(p)=+\infty$. This yields the equation \eqref{eq:primitive} admits a unique positive solution, which is the true physical pressure $p({\bm U})$. 
In addition, the equation \eqref{eq:primitive} implies that 
the true physical pressure $p({\bm U})$ satisfies 
\begin{align}\label{key112}
	\frac{p}{\gamma-1}-E+\frac{| \bm{m} |^2}{E+p} < 0. 
	%\\
	%p \le (\gamma - 1) \left( E-D\sqrt{1-\frac{| \bm{m} |^2}{E^2}} \right).
\end{align}
A natural idea is to directly solve equation \eqref{eq:primitive} by using the NR method. Unfortunately, the numerical experiments in \cite{chen2022physical} indicate that the convergence of such a NR method requires a good initial guess, and the approximate pressure may become negative during the NR iterations.  

In order to introduce our monotonically convergent NR method, we 
define 
\begin{align}
	&h_1(p):=\left(| \bm{m} |^2+(E+p)\left(\frac{p}{\gamma-1}-E\right)\right)^2,
	\\
	&h_2(p):=D^2\left((E+p)^2-| \bm{m} |^2\right).
\end{align}
After a simple transformation of (\ref{eq:primitive}), we obtain a quartic equation of $p$:
\begin{equation}\label{eq:primitive_poly}
	\phi(p) :=(\gamma-1)^2\left[h_1(p)-h_2(p)\right]=c_0+c_1p+c_2p^2+c_3p^3+p^4=0
\end{equation}
with
\begin{eqnarray}\label{eq:expression of a}
	\left\{
	\begin{aligned}
		c_0&=(| \bm{m} |^2-E^2)(| \bm{m} |^2-E^2+D^2)(\gamma-1)^2, \\
		c_1&=2E(2-\gamma)(| \bm{m} |^2-E^2)(\gamma-1)-2ED^2(\gamma-1)^2,\\
		c_2&=E^2(\gamma^2-6\gamma+6)+2| \bm{m} |^2(\gamma-1)-D^2(\gamma-1)^2,\\
		c_3&=2E(2-\gamma).
	\end{aligned}
	\right.
\end{eqnarray}
The graphs of $h_1(p)$, $h_2(p)$, and $\phi(p)$ are shown in Figure \ref{fig:phikf}.
\begin{figure}[!htb]
	\centering
	\includegraphics[width=0.618\textwidth]{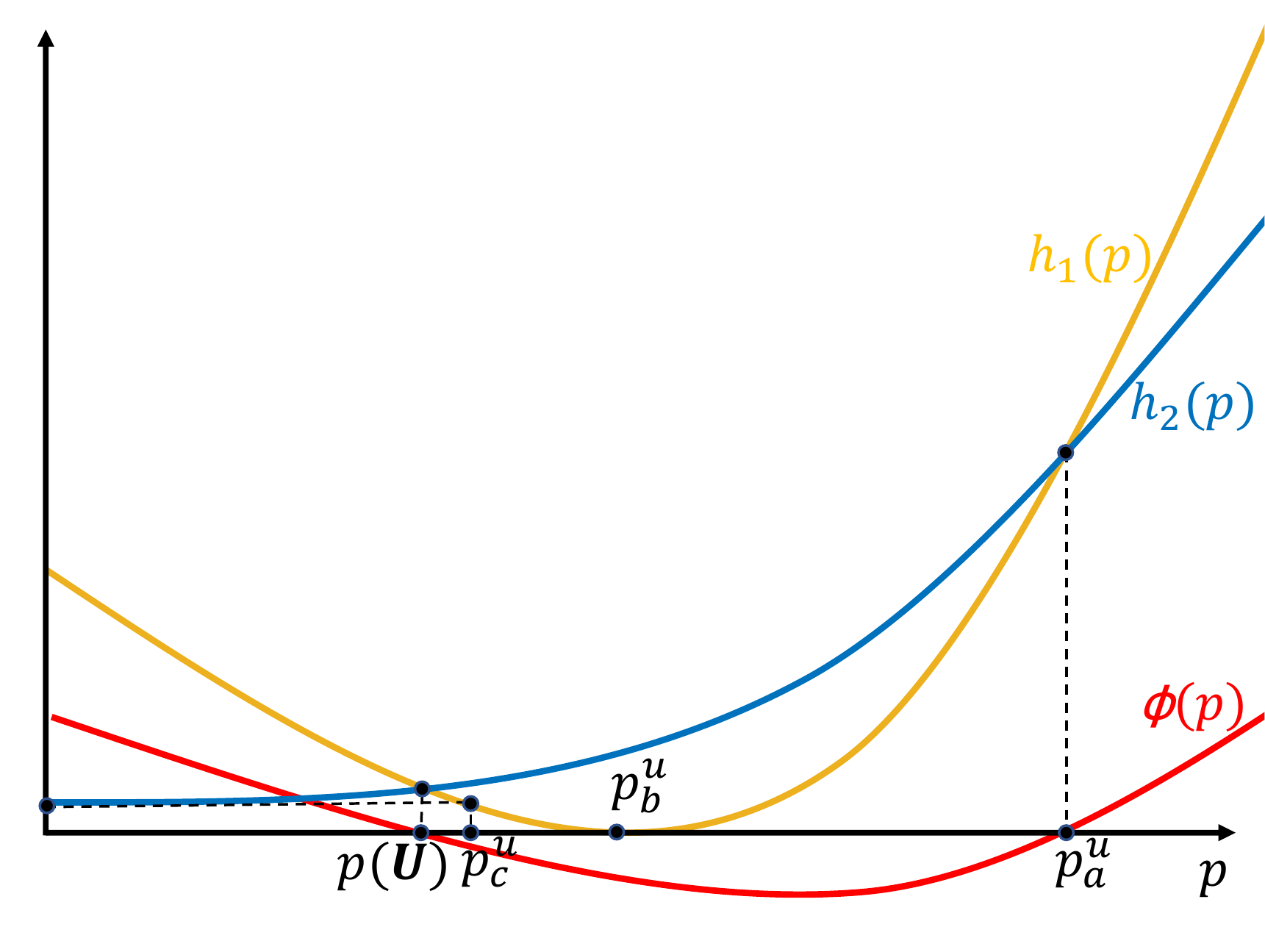}
	\caption{The graphs of $h_1(p)$, $h_2(p)$, and $\phi(p)$.}\label{fig:phikf}
\end{figure}

The transformation from \eqref{eq:primitive} to \eqref{eq:primitive_poly} produce some additional (nonphysical) roots, which fail to meet the constraints $p>0$ and \eqref{key112}. In the following, we will prove that the minimum positive root of equation \eqref{eq:primitive_poly} corresponds to the physical pressure $p({\bm U})$. Furthermore, we will propose a practical NR method that are highly efficient and can be proven to monotonically converge to the physical pressure $p({\bm U})$.

\begin{lemma} \label{lemma:signofa}
	Given ${\bm U}\in {\mathcal G}$, we have	$c_0>0$, $c_1<0$, and $c_3\geq0$. Furthermore, $c_3=0$ if and only if $\gamma=2$.
\end{lemma}

\begin{proof}
	Since $D>0$, $E>\sqrt{D^2 + |{\bm m}|^2}$, and $\gamma\in(1,2]$, we have 
	\begin{enumerate}[(i)]
		\item $c_0=(E^2-|\bm m|^2)(E^2-D^2-|\bm m|^2)(\gamma-1)^2>0$.
		
		\item $c_1=2E(2-\gamma)(|\bm m|^2-E^2)(\gamma-1)-2ED^2(\gamma-1)^2<-2ED^2(\gamma-1)^2<0$.
		
		\item $c_3=2E(2-\gamma)\geq 0$, and the equality holds if and only if $\gamma=2$.
	\end{enumerate}
\end{proof}

\begin{lemma} \label{lemma:2posroot}
	Given ${\bm U}\in {\mathcal G}$, $\phi(p)$ has at least two different positive roots, which are located in the intervals $(0,p_b^u)$ and $(p_b^u,+\infty)$, respectively, where 
	\begin{equation}\label{eq:pbu}
		p_b^u:=\frac{E(\gamma-2)+\sqrt{E^2(2-\gamma)^2-4(\gamma-1)(|\bm{m}|^2-E^2)}}{2}.
	\end{equation}
\end{lemma}

\begin{proof}
	Consider the quadratic function 
	$h_3(p)=| \bm{m} |^2+(E+p)\left(\frac{p}{\gamma-1}-E\right)$. 
	Note that $h_3(0)=| \bm{m} |^2-E^2<0$, and $h'_3(p) = \frac{2p}{\gamma-1}+\frac{2-\gamma}{\gamma-1}E > 0$ when $p>0$. Thus $h_3(p)$ is strictly increasing on $[0,+\infty)$ and has only one positive root, which is exactly $p_b^u$ defined in \eqref{eq:pbu}.  
	%	Therefor, $h(p)<0$ when $p\in\left[0,p_b^u\right)$ and $h(p)>0$ when $p\in\left(p_b^u,+\infty\right)$. Since $h(p)$ is monotonically increasing when $p\geq 0 > \frac{(\gamma-2)E}{2}$, we can obtain that $k(p)=h(p)^2$ is monotonically decreasing in $\left[0,p_b^u\right)$ and monotonically increasing in $\left(p_b^u,+\infty\right)$. Similarly, 
	Note that $h_2(p)$ is monotonically increasing when $p\geq 0$, which implies  
	$$h_2(p)\geq h_2(0)=D^2(E^2-|\bm{m}|^2)>D^4>0 \quad \forall p \ge 0.$$
	Hence we have
	\begin{align*}
		\phi(p_b^u) =(\gamma-1)^2(h_1(p_b^u)-h_2(p_b^u))
		=(\gamma-1)^2(h_3^2(p_b^u)-h_2(p_b^u))
		=-(\gamma-1)^2h_2(p_b^u)<0.
	\end{align*}
	Since $\phi(0)=c_0>0$ and $\lim\limits_{p\to +\infty}\phi(p)=+\infty$, 
	according to the zero point theorem, we know that $\phi(p)$ has at least two different positive roots, which are located in the intervals $(0,p_b^u)$ and $(p_b^u,+\infty)$, respectively. The proof is completed. 
\end{proof}

For convenience, we will count the number of roots by including the multiplicity of a repeated root, unless otherwise specified.

\begin{lemma} \label{lemma:2negroot}
	If $\phi(p)$ has 4 real roots, then it has 2 negative roots.
\end{lemma}

\begin{proof}
	Assume the 4 real roots of $\phi(p)$ are $\hat p_1\leq \hat p_2\leq \hat p_3\leq \hat p_4$. According to Lemma \ref{lemma:2posroot}, we have $\hat p_4 > \hat p_3> 0$. Then it suffices to prove $\hat p_1< \hat p_2<0$. 
	According to Vieta's formulas, $\hat p_1+\hat p_2+\hat p_3+ \hat p_4=-c_3\leq0$ and $\hat p_1 \hat p_2 \hat p_3 \hat p_4=c_0>0$, we can conclude that $\hat p_1\le \hat p_2<0$, which finishes the proof.
\end{proof}

With Lemmas \ref{lemma:2posroot} and \ref{lemma:2negroot}, we immediately obtain the following theorem.

\begin{Theorem} \label{thm:2negroot2posroot}
	The quartic polynomial	$\phi(p)$ has either 2 different positive roots and 2 negative roots, or 2 different positive roots and 2 complex roots.
\end{Theorem}

\begin{Theorem} \label{thm:smaller}
	The smallest positive root of $\phi(p)$ is the unique positive root of equation (\ref{eq:primitive}), which is the physical pressure $p({\bm U})$.
\end{Theorem}

\begin{proof}
	%	We know that pressure is positive, to proof this theorem, it is sufficient to proof only the bigger positive root of equation (\ref{eq:primitive_poly}) can violate condition (\ref{cond:implicit}). For any $p>0$, we have $(E+p)^2-\left \| \bm{m}\right \|^2\geq E^2-\left \| \bm{m}\right \|^2>D^2>0$, so the only possibility for a positive root p to violate (\ref{cond:implicit}) is $\frac{p}{\gamma-1}-E+\frac{\left \| \bm{m}\right \|^2}{E+p}>0$, denote $q(p)=\frac{p}{\gamma-1}-E+\frac{\left \| \bm{m}\right \|^2}{E+p}$ , then $q'(p)=\frac{2}{\gamma-1}-E+\frac{E}{\gamma-1}=\frac{(2-\gamma)E+2}{\gamma-1}>\frac{(2-\gamma)E}{\gamma-1}\geq 0$, so the positive root $p$ which satisfies $q(p)> 0$ can only be the bigger one, this finishes the proof.
	Denote the larger positive root of $\phi(p)$ as $p_a^u$. The proof of Lemma \ref{lemma:2posroot} implies that $p_a^u>p_b^u>0$. 
	It suffices to prove that $\Phi(p_a^u)\neq0$. Recall that $h_3(p)$ is strictly increasing in the interval $[0,+\infty)$. Therefore, 
	\begin{align*}
		\Phi(p_a^u) &=\frac{p_a^u}{\gamma-1}-E+\frac{| \bm{m} |^2}{E+p_a^u}+D\sqrt{1-\frac{| \bm{m} |^2}{(E+p_a^u)^2}}
		\\
		& =\frac{h_3(p_a^u)}{E+p_a^u}+D\sqrt{1-\frac{| \bm{m} |^2}{(E+p_a^u)^2}}
		\\
		&>\frac{h_3(p_b^u)}{E+p_a^u}+D\sqrt{1-\frac{| \bm{m} |^2}{(E+p_a^u)^2}}=D\sqrt{1-\frac{| \bm{m} |^2}{(E+p_a^u)^2}}>0,
	\end{align*}
	which finishes the proof.
\end{proof}

The relative positions of $p({\bm U})$, $p_a^u$, and $p_b^u$ are illustrated in Figure \ref{fig:phikf}.

%{\color{blue}
	%	\begin{rem}
		%		In \cite{riccardi2008primitive}, several quartic polynomials in recovering primitive variables are presented, including the quartic polynomial $P(\Pi)=0$, where $\Pi=p/[(\gamma-1)D]$ denotes the rescaled pressure. The equation $P(\Pi)=0$ is equivalent to the equation $\phi([(\gamma-1)D\Pi])=0$. \cite{riccardi2008primitive} mentioned that the quartic equation $P(\Pi)=0$ have two positive roots. Howerver, the method designed in \cite{riccardi2008primitive} to select the physical root need to solve a positive root of a cubic equation $P'(\Pi)=0$, all the two positive roots of the quartic equation $P(\Pi)=0$, and another two quadratic equations, which is costly. Moreover, the numerical method for computing the two positive roots in \cite{riccardi2008primitive} lacks convergence and requires preliminary dichotomic searches and, in some cases, even series expansion, which is too complicated to be practical. 
		%	\end{rem}
	%} 

Before giving our NR-I method, we present several lemmas, which are useful for establishing the convergence and PCP property of the NR-I method.

\begin{lemma} \label{lemma:newtonmono1}
	Let $\{p_n\}_{n\geq0}$ denote the iteration sequence obtained using the NR method to solve an  equation $f(p)=0$. We assume that $p_*$ is a root of $f(p)=0$. If $p_0<p_*$, $f\in C^2[p_0,p_*)$, and one of the following two conditions holds for all $p\in [p_0,p_*)$: 
	\begin{enumerate}[(i)]
		\setlength{\itemindent}{1em}  % item项目空1格！
		\item\label{cond:NR mono1} $f'(p)<0,f''(p)\geq 0$,
		
		\item\label{cond:NR mono2} $f'(p)>0,f''(p)\leq 0$,
	\end{enumerate}
	then the NR iteration sequence $\{p_n\}_{n\geq0}$ is  monotonically increasing and  converges to $p_*$.
\end{lemma}

\begin{proof}
	We only show the proof under the condition (\ref{cond:NR mono1}), while the proof under condition (\ref{cond:NR mono2}) is similar and thus omitted. 
	
	Firstly, we prove that $\{p_n\}_{n\geq0}$ is monotonically increasing and that $p_*$ is an upper bound of $\{p_n\}_{n\geq0}$. It suffices to prove that $p_*-p_{n+1}=p_*-(p_n-\frac{f(p_n)}{f'(p_n)})>0$ and $p_{n+1}-p_n=-\frac{f(p_n)}{f'(p_n)}>0$ if $p_0\leq p_n<p_*$. Under the condition (\ref{cond:NR mono1}), we have $f(p_n)>f(p_*)=0$ and $f'(p_n)<0$, so that $p_{n+1}-p_n=-\frac{f(p_n)}{f'(p_n)}>0$. Define $f_1(p):=p_*-(p-\frac{f(p)}{f'(p)})$, then 
	$$f_1'(p)=-1+\frac{{f'(p)}^2-f(p)f''(p)}{{f'(p)}^2}=-\frac{f(p)f''(p)}{f'(p)^2}\leq 0 \qquad \forall p \in [p_0,p_*).$$ 
	Since $f_1(p_*)=0$ and $p_n<p_*$, we have  $$0=f_1(p_*)<f_1(p_n)=p_*-\left(p_n-\frac{f(p_n)}{f'(p_n)}\right)=p_*-p_{n+1}.$$
	
	Secondly, we prove the convergence. Since $\{p_n\}_{n\geq0}$ is monotonically increasing and has an upper bound $p_*$, by the monotone convergence theorem, we know that the sequence  $\{p_n\}_{n\geq0}$ has a limit $p_{**}\leq p_{*}$. Therefore, 
	$$0=\lim_{n\to +\infty}(p_{n+1}-p_n)=-\lim_{n\to +\infty}\frac{f(p_n)}{f'(p_n)}=-\frac{f(p_{**})}{f'(p_{**})},$$ 
	which implies $f(p_{**})=0$. Since $p_0\leq p_{**}\leq p_{*}$ and $f(p)>0$ when $p\in [p_0,p_*)$, we obtain $p_{**}=p_{*}$. Hence $\lim\limits_{n\to +\infty}p_n=p_{**}=p_{*}$. The proof is completed.
	%	Firstly, the global convergence in these condition is known by properties of NR iteration, we only need to proof $\{p_n\}$ is monotonically increasing. We only prove the theorem under the condition (\ref{cond:NR mono1}) and the proof under condition (\ref{cond:NR mono2}) is similar.
	%	
	%	To prove the theorem, it is sufficent to prove that $p_*-p_{k+1}=p_*-(p_k-\frac{f(p_k)}{f'(p_k)})>0$ and $p_{k+1}-p_k=-\frac{f(p_k)}{f'(p_k)}>0$ if $p_0\leq p_k<p_*$. With the condition (\ref{cond:NR mono1}), we can know that $f(p_k)>f(p_*)=0$ and $f'(p_k)<0$, so $-\frac{f(p_k)}{f'(p_k)}>0$. Denote $f_1(p)=p_*-(p-\frac{f(p)}{f'(p)})$, then $f_1'(p)=-1+\frac{{f'(p)}^2-f(p)f''(p)}{{f'(p)}^2}=-\frac{f(p)f''(p)}{f'(p)^2}\leq 0$. With $f_1(p_*)=0$ and $p_k<p_*$, we obtain $0=f_1(p_*)<f_1(p_k)=p_*-(p_k-\frac{f(p_k)}{f'(p_k)})=p_*-p_{k+1}$, this finishes our proof.
\end{proof}

Similarly, we have the following lemma. 

\begin{lemma} \label{lemma:newtonmono2}
	Let $\{p_n\}_{n\geq0}$ denote the iteration sequence obtained using the NR method to solve an  equation $f(p)=0$. We assume that $p_*$ is a root of $f(p)=0$. If $p_0> p_*$, $f\in C^2(p_*,p_0]$, and one of the following two conditions holds for all $p\in (p_*,p_0]$:
	\begin{enumerate}[(i)]
		\setlength{\itemindent}{1em}  % item项目空1格！
		\item $f'(p)<0,f''(p)\leq 0$,
		
		\item $f'(p)>0,f''(p)\geq 0$,
	\end{enumerate}
	the NR iteration sequence $\{p_n\}_{n\geq0}$ is  monotonically decreasing and  converges to $p_*$.
\end{lemma}

%In the following, we denote $p_{not}$ as the bigger positive root of $\phi(p)=0$, and utilize the symbol $p_{phy}$ to represent the smaller positive root of $\phi(p)=0$, which is the physical root. 

\begin{Theorem} \label{lemma:poly1}
	If $\phi''(0)=2c_2>0$, then $\phi''(p)>0$ and $\phi'(p)<0$ for all $p \in [0, p({\bm U}))$.
\end{Theorem}

\begin{proof}
	Note that $\phi''(p)=12p^2+6c_3p+2c_2$ and $\phi'''(p)=24p+6c_3$. 
	Because $c_3\ge 0$, we have $\phi'''(p)>0$ when $p>0$, which yields 
	$\phi''(p)$ is strictly increasing in the interval $[0,+\infty)$. Since $\phi''(0)=2c_2>0$, 
	we have $\phi''(p)\ge \phi''(0) >0$ for all  $p \in  [0,p({\bm U}))\subset[0,+\infty)$.
	
	We then show $\phi'(p)<0$ 
	using proof by contradiction. Suppose that there exists $p_a\in [0,p({\bm U}))$ such that $\phi'(p_a)\geq 0$. 
	Because $\phi'(0)=c_1<0$, there must exist $p_b\in (0,p_a]$ such that $\phi'(p_b)= 0$ by the  intermediate value theorem. 
	Since $p_a^u > p({\bm U})$ and $\phi(p({\bm U}))=\phi(p_a^u)=0$, by Rolle's theorem, there exists $p_c\in [p({\bm U}),p_a^u]$, such that $\phi'(p_c)=0$. Therefore, $p_b\leq p_a<p({\bm U})\leq p_c$. Since $\phi'(p_b)=\phi'(p_c)=0$,  there exists $p_d\in[p_b,p_c]\subset[0,+\infty)$, such that $\phi''(p_d)=0$, which is contradictory to $\phi''(p)>0$ for all $p \in [0,+\infty)$. 
	Thus, the assumption is incorrect, and we have $\phi'(p)<0$ for all $p \in [0, p({\bm U}))$. The proof is completed. 
\end{proof}

\begin{Theorem} \label{lemma:poly2}
	If $\phi''(0)=2c_2\leq 0$, then the largest root of 
	the quadratic polynomial $\phi''(p)$, denoted by $p_e$, satisfies $p_e=\frac{-3c_3+\sqrt{9c_3^2-24c_2}}{12} \ge 0$. 
	Furthermore, we have 
	\begin{enumerate}[(i)]
		\setlength{\itemindent}{1em}  % item项目空1格！
		\item If $0\leq p_e<p({\bm U})$, then $\phi''(p)\geq 0$ and $\phi'(p)<0$ for all $p\in [p_e, p({\bm U}))$.
		\item If $p_e>p({\bm U})$, then $\phi''(p)\leq 0$ and $\phi'(p)<0$ for all $p\in (p({\bm U}),p_e]$.
	\end{enumerate}
\end{Theorem}

\begin{proof}
	Recall that $\phi''(p)=12p^2+6c_3p+2c_2$, and $\phi'''(p)=24p+6c_3>0$ when $p>0$, which yields 
	$\phi''(p)$ is strictly increasing in $[0,+\infty)$. 
	Note that  
	$\lim\limits_{p\rightarrow+\infty }\phi''(p)=+\infty$. 
	If 
	$\phi''(0)=2c_2 \le  0$, 
	then by the intermediate value theorem, we know  $p_e=\frac{-3c_3+\sqrt{9c_3^2-24c_2}}{12} \ge 0$.

	We then prove the conclusions (i) and (ii) separately.
	\begin{enumerate}[(i)]
		\setlength{\itemindent}{1em}  % item项目空1格！
		\item It is obvious that $\phi''(p)\geq 0$ when $p\geq p_e$. 
		Suppose there exists $p_a\in[p_e,p({\bm U}))$ such that $\phi'(p_a)\geq 0$, then $\phi'(p)\geq 0$ for all $p \in [p_a,p({\bm U}))$ because $\phi''(p)\geq 0$ for all $p \in [p_e,p({\bm U}))$. 
		This means $\phi'(p)$ is monotonically increasing in $[p_a,p({\bm U}))$, implying 
		$\phi(p_a)\leq \phi (p({\bm U})) = 0$. 
		Since $\phi(0)>0$, according to the intermediate value theorem, there exists $p_b\in (0,p_a]$ such that $\phi(p_b)=0$. This contradicts with the fact that $p({\bm U})$ is the smallest positive root of 
		$\phi(p)$. Thus, the assumption is incorrect, and we have $\phi'(p)<0$ for all $p\in [p_e, p({\bm U}))$.
		\item Notice that $\phi''(p)\leq 0$ for all $p\in(p({\bm U}),p_e]\subset[0,p_e]$. Since $\phi'(0)=a_1<0$ and $\phi''(p)\leq 0$ when $p \in [0,p_e]$, we have $\phi'(p)\leq \phi'(0)<0$ when $p\in [0, p_e]\supset(p({\bm U}), p_e]$. 
	\end{enumerate}
	The proof is completed. 
\end{proof}

Inspired by Theorems \ref{lemma:poly1} and \ref{lemma:poly2} as well as Lemmas \ref{lemma:newtonmono1} and \ref{lemma:newtonmono2}, we design the following NR method for recovering the pressure $p({\bf U})$ from the conservative vector $\bm{U}$. 

\begin{Algorithm}[NR-I method] \label{expression:newnewton}
	The NR iteration reads
	$$p_{n+1}=p_n-\frac{\phi(p_n)}{\phi'(p_n)},$$ 
	with $p_0$ given by 
	\begin{equation}
		p_0=
		\begin{cases}
			0, \quad  &\text{if }~c_2>0, \\
			\frac{-3c_3+\sqrt{9c_3^2-24c_2}}{12}, \quad  &\text{otherwise}.
		\end{cases}
	\end{equation}
	The specific expressions for $\phi(p)$ and $c_i,i=2,3$, are given in \eqref{eq:primitive_poly} and \eqref{eq:expression of a}.
\end{Algorithm}

Note that $p_0\geq 0$ and $p({\bm U})>0$. 
We immediately obtain the following conclusions from Theorems \ref{lemma:poly1} and \ref{lemma:poly2} as well as Lemmas \ref{lemma:newtonmono1} and \ref{lemma:newtonmono2}. 

\begin{Theorem} \label{thm:mono}
	The iteration sequence $\{p_n\}_{n\ge 0}$ generated by Algorithm \ref{expression:newnewton} converges monotonically to $p({\bm U})$. Furthermore, Algorithm \ref{expression:newnewton} is PCP.
\end{Theorem}

%Moreover, we have the following estimates on the convergence rate of Algorithm \ref{expression:newnewton}:

\begin{Theorem} \label{thm1} 
	Algorithm \ref{expression:newnewton} has a quadratic convergence. 
\end{Theorem}

\begin{proof}
	Thanks to Theorems \ref{thm:2negroot2posroot} and \ref{thm:smaller}, we known that the physical pressure 
	$p({\bf U})$, as a root of $\phi(p)$, is not a repeated root, namely, its multiplicity is one. 
	Therefore, as a NR iteration, Algorithm \ref{expression:newnewton} has quadratic convergence.  
\end{proof}

\begin{rem}
	In practical computations, the Horner's rule can be applied to efficiently evaluate the polynomials $\phi(p)$ and $\phi'(p)$. %To illustrate how the Horner's rule works, take a quadratic polynomial as an example: $d_1x^2+d_2x+d_3=(d_1x+d_2)x+d_3$.
\end{rem}

\begin{rem}
	When using the NR method to solve $\phi(p)=0$, oscillations may occur when the value of $\phi(p_n)$ is very close to zero; see \cite{flocke2015algorithm}. Detecting the oscillations  caused by round-off errors and then stopping the iteration are important to avoid unnecessary computational costs. Since Algorithm \ref{expression:newnewton} has monotonic convergence (Theorem \ref{thm:mono}), it is reasonable to expect that oscillations occur when the theoretical monotonicity of the iteration sequence is lost due to round-off errors. Therefore, in addition to the standard stopping criterion $|\phi(p_n)|<\epsilon_{target}$ ($\epsilon_{target}$ denotes the target accuracy), we also include such oscillations as a termination criterion in our computations. 
\end{rem}

\subsection{Analytical expression of $p({\bm U})$}\label{section:Analytical}
With Lemma \ref{lemma:signofa} and Theorems \ref{thm:2negroot2posroot}--\ref{thm:smaller}, we can obtain the analytical expression of pressure $p({\bm U})$ by using the Ferrari method. 

\begin{Algorithm}[Analytical] \label{Analytical}
	$p({\bm U})=\frac{M_5-c_3-M_6}{4}$ with 
	\begin{eqnarray*}
		\left\{
		\begin{aligned}
			M_1&=\frac{c_2^2+12c_0-3c_3c_1}{9}, \\
			M_2&=\frac{27c_1^2+2c_2^3+27c_3^2c_0-72c_2c_0-9c_3c_2c_1}{54},\\
			M_3&=\left(M_2+(M_2^2-M_1^3)^{\frac{1}{2}}\right)^{\frac{1}{3}},\\
			M_4&=M_3+\frac{M_1}{M_3}+\frac{c_2}{3},\\
			M_5&=(c_3^2+4(M_4-c_2))^{\frac{1}{2}},\\
			M_6&=\left((c_3-M_5)^2-8\Big(M_4-\frac{c_3M_4-2c_1}{M_5}\Big)\right)^{\frac{1}{2}},
		\end{aligned}
		\right.
	\end{eqnarray*}
	where the specific expressions of $c_k$ $(k=0,1,2,3)$ are given in \eqref{eq:expression of a}.  
	The function $(\cdot)^{\frac{1}{n}}$ $(n=2,3)$ is a single-valued complex function defined as follows 
	\begin{equation}
		(A_1+A_2i)^{\frac{1}{n}}=\left(A^{\frac{1}{n}}\cos\frac{\theta}{n}\right)+\left(A^{\frac{1}{n}}\sin\frac{\theta}{n}\right)i,
	\end{equation}
	where $i=\sqrt{-1}$ is the imaginary unit, $A_1,A_2\in\mathbb{R}$ and 
	\begin{align}
		A&=\sqrt{A_1^2+A_2^2},\\
		\theta&=\arctan\frac{A_2}{A_1}.
	\end{align}
	It can be seen that $p({\bm U})$ is essentially expressed explicitly in terms of $\bm U$. However, because the $(\cdot)^{\frac{1}{n}}$ operation $(n=2,3)$ inherently includes trigonometric, inverse trigonometric, and cube root or square root operations, using Algorithm \ref{Analytical} to calculate $p({\bm U})$ is very costly and sometimes of low accuracy. We will observe this fact and compare Algorithm \ref{Analytical} with our NR methods in the numerical experiments in Section \ref{sec:numerical_tests}. 
\end{Algorithm}

\subsection{NR-II method: robust PCP convergent NR iteration}\label{algor:robust newton}
To facilitate the algorithm design, we define 
\begin{equation}\label{eq:express of psi}
	\psi(p):=(E+p)\Phi(p) =| \bm{m} |^2+(E+p)\left(\frac{p}{\gamma-1}-E\right)+D\sqrt{(E+p)^2-| \bm{m} |^2}
\end{equation}
and transform the equation \eqref{eq:primitive} into
\begin{equation}\label{eq:primitive2}
	\psi(p)=0.
\end{equation}

In this subsection, we aim to study the NR iteration method solving the equation \eqref{eq:primitive2}. In particular, we would like to find an appropriate initial value $p_0$, 
such that the resulting NR method is PCP and provably convergent. 
Moreover, we hope that $p_0$ is sufficiently close to $p(\bm U)$ in order to reduce the number of iterations.

Define 
\begin{equation}\label{pcu}
	p_c^u:=\frac{(\gamma-2)E+\sqrt{(2-\gamma)^2E^2-4(\gamma-1)\left[(| \bm{m} |^2-E^2)+D\sqrt{E^2-| \bm{m} |^2}\right]}}{2},
\end{equation}
which is the root of the following equation 
\begin{equation}\label{key3123}
	h_1(p_c^u)=h_2(0). 
\end{equation}
We observe that $p_c^u$ is closer to $p(\bm U)$ than $p_b^u$, as shown in the Figure \ref{fig:phikf} and proven in Lemma \ref{lemma:upbound}.

\begin{lemma} \label{lemma:upbound}
	$p(\bm U)<p_c^u<p_b^u<p_a^u$.
\end{lemma}

\begin{proof}
	According Lemma \ref{lemma:2posroot} and Theorem \ref{thm:smaller}, we 
	obtain  $0<p({\bm U})<p_b^u<p_a^u$. 
	Recalling that $h_1(p)$ is strictly decreasing in the interval $(0,p_b^u)$ and $h_2(p)$ is strictly increasing in $(0,p_b^u)$, we have
	\begin{equation*}
		h_1(p_b^u)=0<D^2(E^2-| \bm{m} |^2)=h_2(0)<h_2(p({\bm U}))=h_1(p({\bm U}))<h_1(0). 
	\end{equation*}
	By the intermediate value theorem, there exists a unique point $p_c^u\in(0,p_b^u)$ such that $h_1(p_c^u)=h_2(0)$ and $p({\bm U})<p_c^u<p_b^u$. The expression \eqref{pcu} of $p_c^u$ can be easily obtained by solving the equation $h_1(p_c^u)=h_2(0)$. The proof is completed. 
\end{proof}

Although $p_c^u$ is close to $p(\bm U)$, the NR method for solving the equation \eqref{eq:primitive2} with $p_0=p_c^u$ is not always convergent and PCP. 
One can verify that 
\begin{align}\nonumber
	\psi'(p)&=\frac{2p+(2-\gamma)E}{\gamma-1}+\frac{D(E+p)}{\sqrt{(E+p)^2-| \bm{m} |^2}}>0 \quad \forall p \ge 0, 
	\\ \nonumber
	\psi''(p) &=\frac{2}{\gamma-1}-\frac{D| \bm{m} |^2}{\left[(E+p)^2-| \bm{m} |^2\right]^{\frac{3}{2}}},
	\\ \label{eq:232}
	\psi'''(p)&>0 \quad \forall p \ge 0,   
\end{align}
and $\lim\limits_{p\to +\infty}\psi''(p)=\frac{2}{\gamma-1}>0$. 
If $\psi''(0)\geq0$, then \eqref{eq:232} implies that $\psi''(p)>0$ when $p>0$. If $\psi''(0)<0$, then because $\lim\limits_{p\to +\infty}\psi''(p)>0$, there exists a inflection point $p_{in}>0$ satisfying $\psi''(p_{in})=0$. Since  $\psi'''(p)>0$ when $p>0$, we can deduce that $\psi''(p)>0$ on $(p_{in},+\infty)$ and $\psi''(p)<0$ on $(0,p_{in})$. Therefore, the concavity/convexity structure of the function $\psi(p)$ in $[0,+\infty)$ can only have three cases: 
\begin{enumerate}[(a)]
	\item $\psi''(p)>0$ when $p>0$.
	\item There exists a inflection point $p_{in}\in(0,p(\bm U))$ satisfying $\psi''(p_{in})=0$. $\psi''(p)>0$ on $(p_{in},+\infty)$, while $\psi''(p)<0$ on $(0,p_{in})$.
	\item There exists a inflection point $p_{in}\geq p(\bm U)$ satisfying $\psi''(p_{in})=0$. $\psi''(p)>0$ on $(p_{in},+\infty)$, while $\psi''(p)<0$ on $(0,p_{in})$. 
\end{enumerate}
The graphs of $\psi(p)$ are illustrated in Figure \ref{fig:convexity} for Cases (a),  (b), and  (c), respectively.

\begin{figure}[!htb]
	\centering
	\subfloat[]{
		\includegraphics[width=0.3\textwidth]{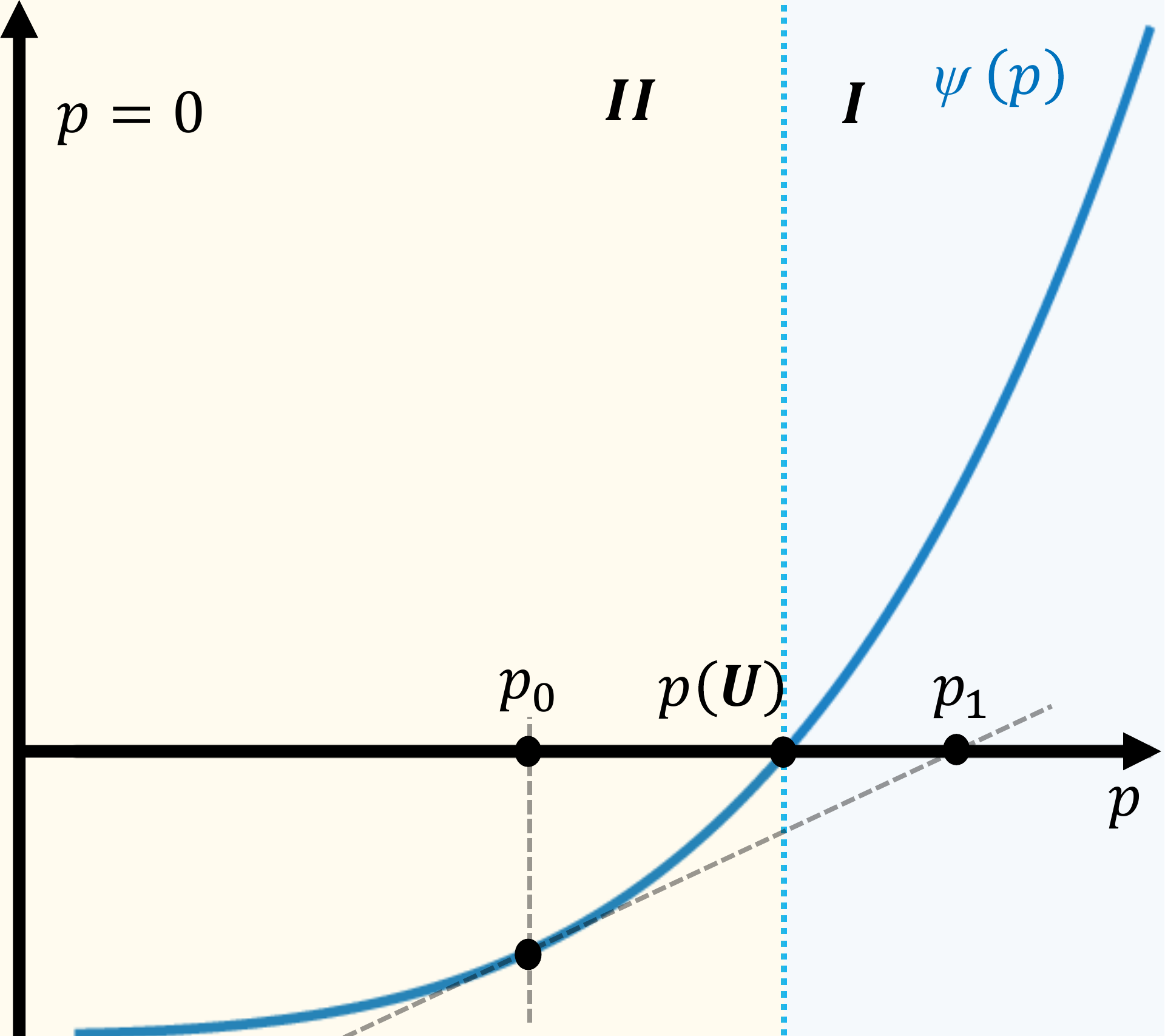}\label{fig:convexity1}
	}
	\subfloat[]{
		\includegraphics[width=0.3\textwidth]{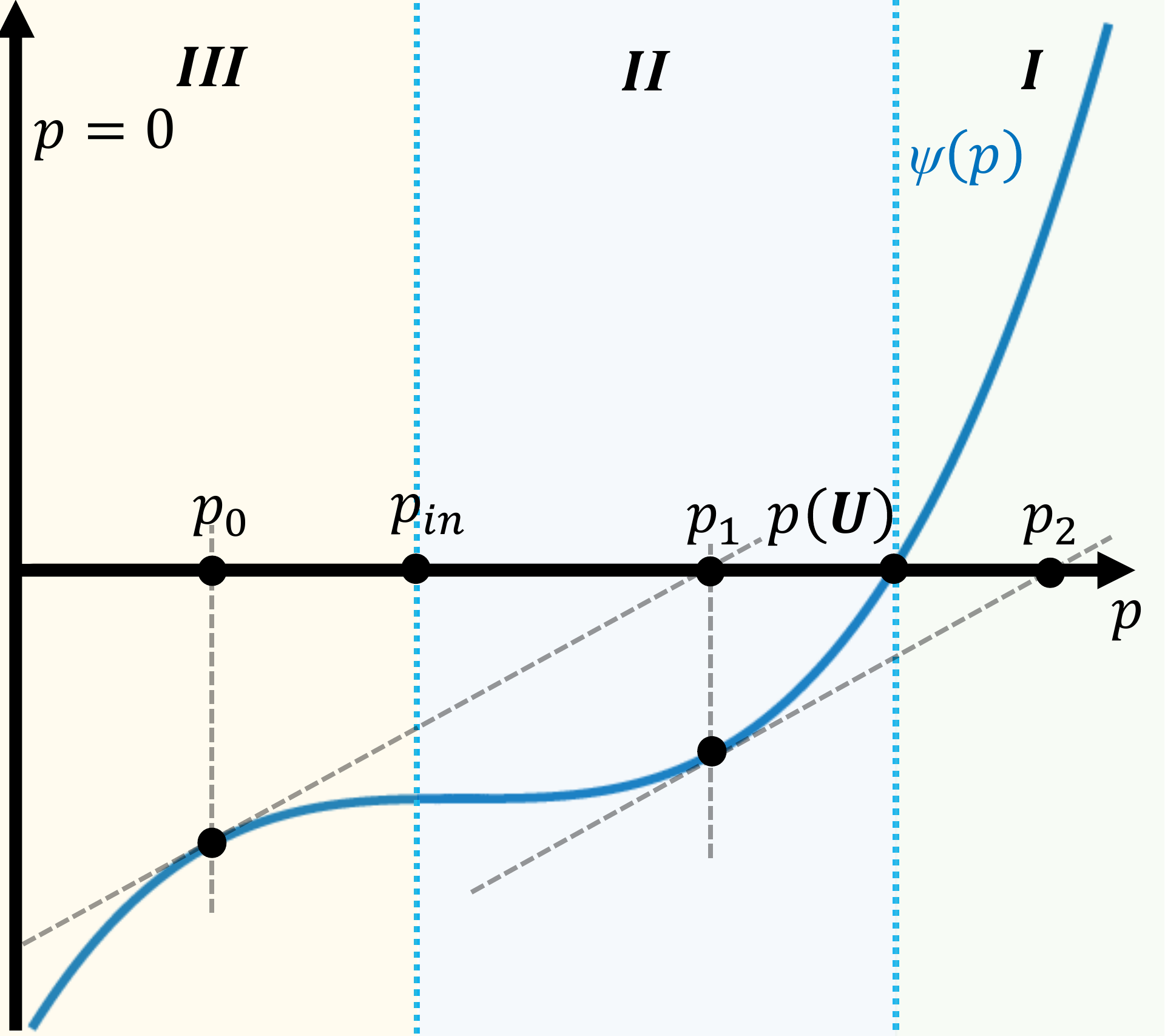}\label{fig:convexity2}
	}
	\subfloat[]{
		\includegraphics[width=0.3\textwidth]{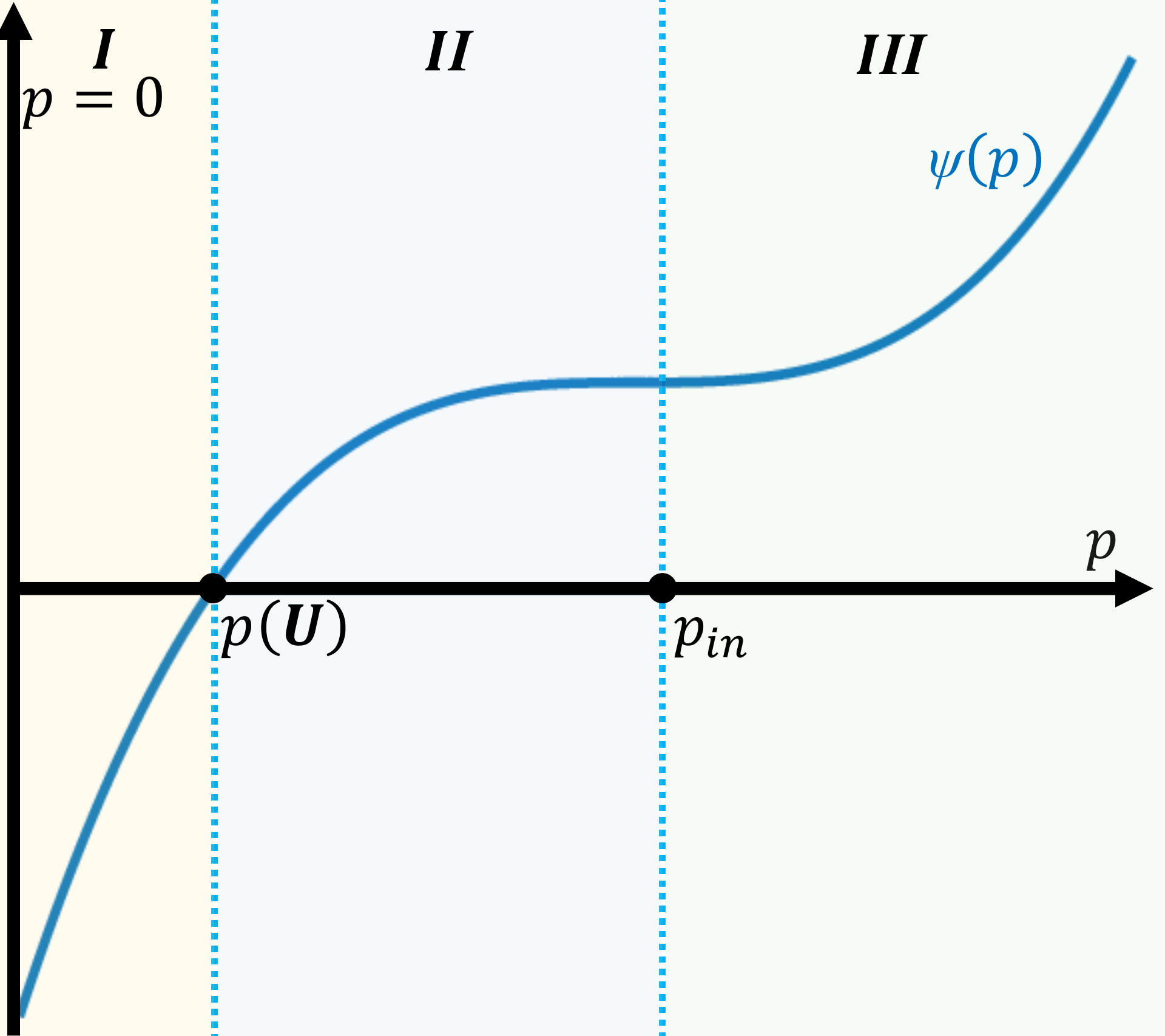}\label{fig:convexity3}
	}
	\caption{Three possible cases of the concavity/convexity structure of $\psi(p)$.}\label{fig:convexity}
\end{figure}

%In fact, if we take $p_0=0$, then the NR iteration for equation $\psi(p)=0$ is always convergent to $p_{phy}$, and the iteration sequence $\left\{p_i\right\}$ are all positive.

%\begin{Theorem}
%	The NR algorithm for equation $\psi(p)=0$ with $p_0=0$ is convergence-guaranteed and robust.
%\end{Theorem}

\begin{Theorem}\label{thm:old nr convergence}
	In Cases (a) and (b), 
	the NR iteration for solving $\psi(p)=0$ always converges to $p({\bm U})$ with any initial guess $p_0 \ge 0$. 
	In Case (c), the NR iteration for solving $\psi(p)=0$ always converges to $p({\bm U})$ with any initial guess $p_0\in[0,p({\bm U})]$, and if $p_0 > p({\bf U}) $, then 
	it	may fail to converge to $p({\bm U})$. Furthermore, if the NR iteration fails to converge, then  negative $p_{n}$ would appear in the iteration sequence $\left\{p_n\right\}_{n\geq0}$.
\end{Theorem}
\begin{proof}
	We discuss the three cases in Figure \ref{fig:convexity} separately.
	
	\begin{enumerate}
		\item[(a)] In this case, $\psi''(p)>0$ when $p>0$. 
		
		(I) If $p_0\in [p({\bm U}),+\infty)$, then the iteration sequence $\left\{p_n\right\}_{n\geq0}$ converges monotonically to $p({\bm U})$ according to Lemma \ref{lemma:newtonmono2}. 
		
		(II) If $p_0\in[0,p({\bm U}))$, then $p_1=p_0-\frac{\psi(p_0)}{\psi'(p_0)}$, $\psi(p_1)>\psi(p_0)+(p_1-p_0)\psi'(p_0)=0$. Thus $p_1\in [p({\bm U}),+\infty)$. Then, following the discussion of Case (I), the iteration sequence $\left\{p_n\right\}_{n\geq1}$ converges monotonically to $p({\bm U})$.
		
		\item[(b)] In this case, $p({\bm U})>p_{in}>0$, where $\psi''(p_{in})=0$. 
		
		(I) If $p_0\in [p({\bm U}),+\infty)$, then the iteration sequence $\left\{p_n\right\}_{n\geq0}$ converges monotonically to $p({\bm U})$ according to Lemma \ref{lemma:newtonmono2}. 
		
		(II) If $p_0\in [p_{in},p({\bm U}))$, then similarly to Case (a)(II), we have $p_1\in [p({\bm U}),+\infty)$, and the iteration sequence $\left\{p_n\right\}_{n\geq1}$ converges monotonically to $p({\bm U})$. 
		
		(III) If $p_0\in[0,p_{in})$, $p_1=p_0-\frac{\psi(p_0)}{\psi'(p_0)}>p_0$. If $p_1\ge p_{in}$, then the discussion returns to Cases (I) and (II). Thus we only need to discuss the case when $p_1\in[0,p_{in})$. By repeatedly following the aforementioned discussions, as long as $p_n \ge p_{in}$ appears in the iteration, we can return to the discussion of Case (I) or (II), and conclude that the NR method converges. It remains to discuss whether it is possible that $p_n < p_{in}$ for all $n \ge 0$. Assume that such a situation occurs, then since $p_n < p_{in}$ and $p_{n+1}>p_{n}$, according to the monotone bounded convergence theorem, $\{p_n\}_{n\geq0}$ has a limit $p^{*}\in[0,p_{in}] $. Therefore, 
		$$0=\displaystyle\lim_{n\to +\infty}(p_{n+1}-p_n)=-\displaystyle\lim_{n\to +\infty}\frac{\psi(p_n)}{\psi'(p_n)}=-\frac{\psi(p^{*})}{\psi'(p^{*})},$$ 
		which yields $\psi(p^{*})=0$. This is contradictory to $p^{*}\in[0,p_{in}]$ and $p_{in}<p({\bm U})$ (note that $p({\bm U})$ is the unique positive root of $\psi(p)$). Hence, the assumption is incorrect, and there always exists a $n$ such that $p_n \ge p_{in}$. In short, the NR method always converges. 
		
		\item[(c)] In this case, $p_{in}\ge p({\bm U})>0$ where $\psi''(p_{in})=0$. 
		
		(I) If $p_0\in[0,p({\bm U})]$, then the iteration sequence $\left\{p_n\right\}_{n\geq0}$ converges monotonically to $p({\bm U})$ according to Lemma \ref{lemma:newtonmono1}. 
		
		(II) If $p_0\in(p({\bm U}),p_{in}]$, then similar to Case (a)(II) and Case (b)(II), we have $p_1\le p(\bm U)$. If $p_1<0$, the convergence cannot be guaranteed. If $p_1\geq0$, then we return to the discussion of Case (I), and conclude that the iterative sequence converges to $p(\bm U)$.
		
		(III) If $p_0\in(p_{in},+\infty)$, similar to the discussion in Case (b)(III), we can use proof by contradiction to prove that there exists an iterative value $p_n$ belongs to the interval $(-\infty,0)$, $[0,p(\bm U)]$, or $(p(\bm U),p_{in})$. 
		If $p_n\in(-\infty,0)$, then the convergence cannot be guaranteed. If $p_n\in[0,p(\bm U)]$, then we return to the discussion of Case (c)(I) and conclude that the iterative sequence  converges to $p(\bm U)$. If $p_n\in(p(\bm U),p_{in}]$, then we return to the discussion of  Case (c)(II): the iteration sequence either converges to $p(\bm U)$, or a negative value appears in the iterative sequence so that the convergence cannot be guaranteed.
		
		In summary, in Case (c), the iteration sequence either remains non-negative and converges to $p(\bm U)$, or a negative number appears in the iterative sequence so that the convergence cannot be guaranteed. 
	\end{enumerate}
	The proof is completed. 
\end{proof}

As a direct consequence of Theorem \ref{thm:old nr convergence}, we have the following conclusion. 

\begin{Theorem}\label{prop:0ini robust}
	The NR method with $p_0=0$ for solving the equation $\psi(p)=0$ is always convergent and PCP.
\end{Theorem}

\begin{Theorem}\label{thm:robust NR convergent}
	If
	\begin{equation}\label{cond:oldnewton convergence}
		D<\frac{E^2-| \bm{m} |^2}{E},
	\end{equation}
	then the NR method for solving $\psi(p)=0$ with any initial guess $p_0\geq 0$ is always PCP and convergent.
\end{Theorem}

\begin{proof}
	First, we prove the PCP property, namely, show that the iteration sequence $\left\{p_n\right\}_{n\geq1}$ are always positive. 
	Assume that $p_{n}\geq0$, then it suffices to prove  $p_{n+1}=p_{n}-\frac{\psi(p_n)}{\psi'(p_n)} >0$. Recall that $\psi'(p)>0$ for all $p \in [0,+\infty)$. Note that $p_{n}-\frac{\psi(p_n)}{\psi'(p_n)} > 0$ is equivalent to
	\begin{align*}
		\psi(p_n)&=(E+p_n)\left(\frac{p_n}{\gamma-1}-E\right)+| \bm{m} |^2+D\sqrt{(E+p_n)^2-| \bm{m} |^2}\\
		&< \left(\frac{2p_n+(2-\gamma)E}{\gamma-1}+\frac{D(E+p_n)}{\sqrt{(E+p_n)^2-| \bm{m} |^2}}\right)p_n=p_n\psi'(p_n),
	\end{align*} 
	which is equivalent to
	\begin{equation}\label{eq:ddf}
		D(\gamma-1) < \frac{p_n^2-| \bm{m} |^2(\gamma-1)+E^2(\gamma-1)}{\frac{E^2-| \bm{m} |^2+Ep_n}{\sqrt{(E+p_n)^2-| \bm{m} |^2}}}.
	\end{equation}
	Noting that $p_n^2-| \bm{m} |^2(\gamma-1)+E^2(\gamma-1)\ge (\gamma-1)(E^2-| \bm{m} |^2)>0$, $\frac{E^2-| \bm{m} |^2+Ep_n}{\sqrt{(E+p_n)^2-| \bm{m} |^2}}>0$, and
	\begin{equation*}
		\frac{\partial}{\partial p}\left(\frac{E^2-| \bm{m} |^2+Ep}{\sqrt{(E+p)^2-| \bm{m} |^2}}\right)=\frac{| \bm{m} |^2p}{\left((E+p)^2-| \bm{m} |^2\right)^{\frac{3}{2}}}>0 \quad \forall p\in(0,+\infty), 
	\end{equation*}
	we obtain  
	\begin{equation}
		\frac{p_n^2-| \bm{m} |^2(\gamma-1)+E^2(\gamma-1)}{\frac{E^2-| \bm{m} |^2+Ep_n}{\sqrt{(E+p_n)^2-| \bm{m} |^2}}}
		>
		\frac{(\gamma-1)(E^2-| \bm{m} |^2)}{\lim \limits_{p \to +\infty}\left(\frac{E^2-| \bm{m} |^2+Ep}{\sqrt{(E+p)^2-| \bm{m} |^2}}\right)}
		=
		(\gamma-1)\frac{E^2-| \bm{m} |^2}{E}.\nonumber
	\end{equation}
	Therefore, if $D<\frac{E^2-| \bm{m} |^2}{E}$, then 
	$$D(\gamma-1)<(\gamma-1)\frac{E^2-| \bm{m} |^2}{E}<\frac{p_n^2-| \bm{m} |^2(\gamma-1)+E^2(\gamma-1)}{\frac{E^2-| \bm{m} |^2+Ep_n}{\sqrt{(E+p_n)^2-| \bm{m} |^2}}},$$
	which implies \eqref{eq:ddf}. 
	Hence, $p_{n}\geq0$ implies $p_{n+1}>0$. By induction, we know that the iteration sequence $\left\{p_n\right\}_{n\geq1}$ are always positive. 
	Thanks to Theorem \ref{thm:old nr convergence}, we known that the NR iteration converges. 
	The proof is completed.
	
\end{proof}

%Using Theorem \ref{thm:robust NR convergent} and Theorem \ref{thm:old nr convergence}, we can immediately obtain:
%
%\begin{lemma}\label{prop:old nr any ini}
%	The NR algorithm for $\psi(p)=0$ is PCP and convergent with any initial value $p_0\geq 0$ if
%	\begin{equation}
	%		D<\frac{E^2-| \bm{m} |^2}{E}.
	%	\end{equation}
%\end{lemma}

Inspired by Theorems \ref{prop:0ini robust} and \ref{thm:robust NR convergent}, we design the following algorithm. 

\begin{Algorithm}[NR-II]\label{expression:Robust NR} The NR iteration reads 
	$$p_{n+1}=p_n-\frac{\psi(p_n)}{\psi'(p_n)}$$ 
	with $p_0$ given by 
	\begin{equation}\label{NR-II:p0}
		p_0= 
		\begin{cases}
			0, \quad &\text{if }~D\geq\frac{E^2-| \bm{m} |^2}{E}, \\
			p_c^u, \quad & \text{otherwise}, 
		\end{cases}
	\end{equation}
	where the expression of $p_c^u$ is given in \eqref{pcu} and the expression of $\psi(p)$ is given in \eqref{eq:express of psi}.
\end{Algorithm}

%With Theorems \ref{prop:0ini robust} and \ref{thm:robust NR convergent}, it is not hard to observe that this algorithm is convergent and PCP. Since $\psi'(p)$ on the interval $[0,+\infty)$, $p(\bm U)$ is not a repeated root of $\psi(p)=0$. So the convergence rate of this NR algorithm is second order. To sum up, the ``NR-II" algorithm has the following properties:
%\begin{enumerate}[\bfseries (1)]
%	\setlength{\itemindent}{1em}  % item项目空1格！
%	\item 
%	It is convergent.
%	\item 
%	It is PCP.
%	\item 
%	The convergence rate of it is second order.
%\end{enumerate}

\begin{Theorem}\label{thm2} 
	Algorithm \ref{expression:Robust NR} is always PCP and convergent. Furthermore, it has a quadratic convergence. 
\end{Theorem}

\begin{proof}
	If $D\geq\frac{E^2-| \bm{m} |^2}{E}$, then $p_0=0$, so that Algorithm \ref{expression:Robust NR} is PCP and convergent, according to  Theorem \ref{prop:0ini robust}. 
	If $D<\frac{E^2-| \bm{m} |^2}{E}$, then Theorem \ref{thm:robust NR convergent} implies that Algorithm \ref{expression:Robust NR} is PCP and convergent. 
	Recalling that $\psi'(p)>0$ for all $p\in [0,+\infty)$, we know that $p(\bm U)$ is not a repeated root of $\psi(p)=0$. Therefore, the convergence rate of this NR method is quadratic.
\end{proof}

%\begin{rem}
%	The convergence criterion for the algorithm ``NR-II'' is the same as which introduced in \cite{flocke2015algorithm}, that is, when the NR iterations oscillate around the zero or the relative accuracy is less or equal to $\epsilon_{target}$, end the iteration. Here $\epsilon_{target}$ is the target accuracy of our algorithms.
%\end{rem} 

\subsection{Hybrid NR method: hybrid PCP convergent NR iteration}\label{algor:hybrid newton}

We have proposed two NR methods for recovering primitive variables. 
Algorithm \ref{expression:newnewton} is based on the NR iteration for the polynomial $\phi(p)$. When the polynomial $\phi(p)$ is ill-conditioned, small perturbations in the coefficients $c_i$ $(i=0,1,2,3)$ can lead to significant changes in the root of $\phi(p)=0$, which can result in a reduced accuracy of Algorithm \ref{expression:newnewton}. Nonetheless, since $\phi(p)$ is a polynomial, the computational cost of evaluating $\phi(p)$ and $\phi'(p)$ in each iteration of Algorithm \ref{expression:newnewton} is relatively low. Moreover, the monotonic convergence of Algorithm \ref{expression:newnewton} allows for a simpler and more effective stopping criterion, making it computationally faster (as we will show in Section \ref{sec:numerical_tests}). 
In contrast, Algorithm \ref{expression:Robust NR} requires taking a square root at each iteration to evaluate $\psi(p)$ and $\psi'(p)$, which leads to slower computation speed, but it does not suffer from the ill-conditioned issue and always provides higher accuracy. 

In this subsection, we will propose a hybrid approach that switches to Algorithm \ref{expression:Robust NR} when $\phi(p)$ is ill-conditioned and to Algorithm \ref{expression:newnewton} when $\phi(p)$ is not ill-conditioned, to obtain a NR method that achieves both fast convergence and high accuracy.

The first problem is, how to detect ill condition of polynomials \eqref{eq:primitive_poly} efficiently and conveniently. 
In fact, a polynomial might be ill-conditioned when cluster roots occur \cite{dunaway1972some}. In the following, we observe and prove that when $D$ or $(\gamma-1)$ is very small, $p_a^u$ will be very close to $p(\bm U)$, which results in the polynomial $\phi(p)$ becoming ill-conditioned.

\begin{lemma} \label{thm:detect1}
	For fixed $\gamma,\bm{m}$, and $E$, we have $\lim\limits_{D\to 0^+}|p_a^u-p({\bm U})|=0$.
\end{lemma}

\begin{proof}
	Note that $h_1(p)=h_3^2(p)$ is independent of $D$, and 
	$$\lim_{D\to 0^+}h_2(p)=\lim_{D\to 0^+}D^2\left((E+p)^2-| \bm{m} |^2\right)=0.$$ 
	Since  $\phi(p)=(\gamma-1)^2\left[h_3^2(p)-h_2(p)\right]$, when $D$ approaches zero, $p_a^u$ and $p({\bm U})$ approaches $p_b^u$, which is the unique positive root of $h_3(p)$ as shown in the proof of Lemma \ref{lemma:2posroot}.
\end{proof}

In practical computation, the ratio $\frac{h_2(0)}{h_1(0)}=\frac{D^2}{E^2-|\bm{m}|^2}$ can be used to estimate whether $D$ is small enough to result in clustered roots. Note that $E>\sqrt{D^2+|\bm{m}|^2}$, which implies $\frac{D^2}{E^2-|\bm{m}|^2}<1$.

\begin{lemma} \label{thm:detect2}
	For fixed conservative quantities $D,\bm{m}$, and $E$, we have $\lim\limits_{\gamma\to 1^+}p(\bm U)=\lim_{\gamma\to 1^+}p_a^u=0$, and $\lim\limits_{\gamma\to 1^+}|p_a^u-p({\bm U})|=0$.
\end{lemma}

\begin{proof}
	Consider $h_4(p):=c_1+c_2p+c_3p^2$, where $c_1<0$ and $c_3>0$ when $\gamma <2$ according to Lemma \ref{lemma:signofa}. 
	If $\gamma <2$, then $h_4(p)>0$ when $p>\frac{-c_2+\sqrt{c_2^2-4c_1c_3}}{2c_3}$. 
	When $\gamma<1.2<3-\sqrt{3}$, we have 
	$$
	c_2=E^2(\gamma^2-6\gamma+6)+2|\bm{m}|^2(\gamma-1)-D^2(\gamma-1)^2>0.24E^2-0.04D^2>0.2E^2>0,
	$$
	which yields 
	$$\frac{-c_2+\sqrt{c_2^2-4c_1c_3}}{2c_3}=\frac{-2c_1}{c_2+\sqrt{c_2^2-4c_1c_3}}<-\frac{c_1}{c_2}<-\frac{c_1}{0.2E^2}.$$
	Thus, if $\gamma<1.2$ and $p>-\frac{c_1}{0.2E^2}$, then $h_4(p)>0$ and 
	$\phi(p)=h_4(p)p+c_0+p^4>0$, where we have used $c_0>0$ proved in Lemma \ref{lemma:signofa}. 
	Therefore, 
	$$0<p(\bm U)<p_a^u<-\frac{c_1}{0.2E^2}=10(\gamma-1)\frac{E(2-\gamma)(E^2-| \bm{m} |^2)+ED^2(\gamma-1)}{E^2},$$ 
	which implies $\lim\limits_{\gamma\to 1^+}p(\bm U)=\lim\limits_{\gamma\to 1^+}p_a^u=0$. It follows that $\lim\limits_{\gamma\to 1^+}|p_a^u-p({\bm U})|=0$. The proof is completed.   
\end{proof}

By utilizing Lemmas \ref{thm:detect1} and \ref{thm:detect2}, we can effectively detect the ill-conditioned problem: the polynomial $\phi(p)$ may be ill-conditioned if $\gamma<1+\epsilon_1$ or $\frac{D^2}{E^2-| \bm{m} |^2}<\epsilon_2$, where $\epsilon_1\in\left(0,1\right)$ and $\epsilon_2\in\left(0,1\right)$ are two small positive numbers. 

\begin{rem} \label{rem00}
	When $p_a^u$ and $p({\bf U})$ are in close proximity, the polynomial $\phi(p)$ tends to become ill-conditioned, thereby compromising the accuracy of Algorithm \ref{expression:newnewton}. On the other hand, Algorithm \ref{expression:Robust NR} demonstrates high accuracy and rapid convergence when $p_a^u$ and $p({\bf U})$ are in close proximity. This is due to the fact that $p_c^u$ is positioned between $p(U)$ and $p_a^u$, resulting in the  initial value $p_0$ for Algorithm \ref{expression:Robust NR} that is often very close to the actual pressure $p(\bm U)$ in such cases. In other words, Algorithm \ref{expression:Robust NR} exactly compensates for the limitations of Algorithm \ref{expression:newnewton}.
\end{rem}

Based on the aforementioned discussions, we propose the following hybrid method. 

\begin{Algorithm} [Hybrid NR]\label{expression:Hybrid NR}
	\begin{equation}
		\begin{cases}
			\text{Switch to Algorithm }\ref{expression:newnewton}, \quad &\text{if } \gamma\geq1+\epsilon_1 \text{ and } \frac{D^2}{E^2-| \bm{m} |^2}\geq\epsilon_2.\\
			\text{Switch to Algorithm }\ref{expression:Robust NR},\quad & \text{otherwise},
		\end{cases}
	\end{equation}
	where $\epsilon_1\in(0,1)$ and $\epsilon_2\in(0,1)$ are two small positive numbers. In this paper, we set $\epsilon_1=0.01$ and $\epsilon_2=10^{-4}$.
\end{Algorithm}

\begin{rem}
	Since both the NR-I and NR-II methods are always PCP and convergent with quadratic convergence rate,  the above hybrid NR method is also PCP and convergent with quadratic convergence rate. 
\end{rem}

%As we have discussed, the hybrid NR algorithm achieves both fast convergence and high accuracy. 

If $\gamma\geq1+\epsilon_1$ and $\frac{D^2}{E^2-| \bm{m} |^2}\geq\epsilon_2$, Algorithm \ref{expression:Hybrid NR} utilizes Algorithm \ref{expression:newnewton}, which exhibits fast convergence and high accuracy with low computational cost since the polynomial $\phi(x)$ is not ill-conditioned in this scenario.
In contrast, when $\gamma < 1+\epsilon_1$ or $\frac{D^2}{E^2-| \bm{m} |^2} < \epsilon_2$, Algorithm \ref{expression:Hybrid NR} switches to Algorithm \ref{expression:Robust NR}, and in this case, the initial value $p_0$ of Algorithm \ref{expression:Robust NR} (defined in \eqref{NR-II:p0}) is typically in close proximity to the physical pressure $p(\bm U)$ (as discussed in Remark \ref{rem00}), thus ensuring both fast convergence and high accuracy.
Overall, the hybrid NR method achieves both high accuracy and fast convergence, as confirmed by numerical tests in Section \ref{sec:numerical_tests}.

\section{One-dimensional PCP HWENO scheme}\label{sec:1D HWENO scheme}

In this section, we present the PCP finite volume HWENO scheme for the 1D special RHD equations
\begin{equation}\label{eq:rhd_1D}
	\frac{\partial \bm{U}}{\partial t}+\frac{\partial \bm{F}\left(\bm{U}\right)}{\partial x}={\bf 0},
\end{equation}
where 
\begin{equation*}%\label{eq:1D flux and conserv variable}
	\bm{U}=\left(D,m_1,E\right)^\top,\quad \bm{F}=\left(Dv_i,m_1v_1+p,m_1\right)^\top.
\end{equation*}
%\subsection{1D PCP HWENO scheme} 
Divide the computational domain into $N$ uniform cells $I_i=\left [ x_{i-1/2},x_{i+1/2}\right]$, $1\leq i\leq N$, with the cell center $x_i=\frac{1}{2} (  x_{i-1/2}+x_{i+1/2} )$. Let $\Delta x$ denote the mesh size $x_{i+1/2}-x_{i-1/2}$.

\subsection{1D PCP  finite volume HWENO scheme}\label{section:1D PCP HWENO}

In this subsection, we give the 1D finite volume PCP HWENO scheme. For the notational convenience, we denote  $x_{i+a}=x_i+a\Delta x$, where $a \in [-\frac12,\frac12]$ is a real number.

The semi-discrete finite volume HWENO scheme for the RHD equations (\ref{eq:rhd_1D}) is given by 
\begin{align}\label{eq:1D discrete equation1}
	\frac{\mathrm{d} \bm{\overline{U}}_i(t)}{\mathrm{d} t} &=-\frac{\hat{\bm{F}}_{i+1/2}-\hat{\bm{F}}_{i-1/2}}{\Delta x}=:\mathcal{L}_U(\bm{U}_h(t),i),
	\\
	\label{eq:1D discrete equation2}
	\frac{\mathrm{d} \bm{\overline{V}}_i(t)}{\mathrm{d} t} &=-\frac{\hat{\bm{F}}_{i+1/2}+\hat{\bm{F}}_{i-1/2}}{2\Delta x}+\frac{\bm{H}_i}{\Delta x}=:\mathcal{L}_V(\bm{U}_h(t),i),
\end{align}
where 
$$
\bm{\overline{U}}_i\left(t\right) \!\approx\! \frac{1}{\Delta x}\int_{I_i}\bm{U}(x,t)\mathrm{d}x, 
\quad  
\bm{\overline{V}}_i\left(t\right) \!\approx\! \frac{1}{\Delta x}\int_{I_i}\bm{U}(x,t)\frac{x-x_i}{\Delta x}\mathrm{d}x
$$
denote the approximations to the zeroth and first order moments in $I_i$, respectively,
$$
\bm{U}_h(t)\!=\!\{\bm{\overline{U}}_i\left(t\right)\}_{i=1}^N\cup \{\bm{\overline{V}}_i\left(t\right)\}_{i=1}^N
$$
denotes the set that contains all cells’ zeroth and first order moments, and  $$\bm{H}_i=\sum\limits_{l=1}^{4}\hat{\omega}_\ell\bm{F}({\bm U}_{i\oplus a_\ell}) \approx \frac{1}{\Delta x}\int_{I_i}\bm{F}\left(\bm{U}\right)\mathrm{d}x$$ 
with $\bm{U}_{i\oplus a_\ell}$ denoting the approximation to $\bm{U}(x_{i+a_\ell},t)$ within the cell $I_i$. Here the four-point Gauss--Lobatto quadrature is used with the quadrature weights and nodes given by 
\begin{align*}
	&\hat{\omega}_1=\hat{\omega}_4=\frac{1}{12}, \qquad \hat{\omega}_2=\hat{\omega}_3=\frac{5}{12},
	\\
	&x_{i+a_\ell}=x_i+a_{\ell}\Delta x, \qquad \{ a_{\ell}\}_{\ell=1}^4=\left\{ -\frac{1}{2},~ -\frac{\sqrt{5}}{10},~ \frac{\sqrt{5}}{10},~ \frac{1}{2} \right\}.
\end{align*} 
In \eqref{eq:1D discrete equation1}--\eqref{eq:1D discrete equation2}, 
$\hat{\bm{F}}_{i+1/2}$ denotes the numerical flux at the cell interface point $x_{i+1/2}$. 
In this paper, we employ the Lax--Friedrichs numerical flux\footnote{Another option is  the Harten-Lax-van Leer (HLL) numerical flux, whose PCP property was proved in \cite{chen2022physical}.} 
\begin{equation}\label{1D LF}
	\hat{\bm{F}}_{i+1/2}=\frac{1}{2}\left(\bm{F}_1\left(\bm{ U}_{i\oplus\frac12}\right)+\bm{F}_1\left(\bm{ U}_{(i+1)\oplus\left(-\frac12\right)}\right)-\alpha\left(\bm{ U}_{(i+1)\oplus\left(-\frac12\right)}-\bm{U}_{i\oplus\frac12}\right)\right),
\end{equation}
which will be useful for achieving the PCP property of our HWENO scheme. Here $\alpha$ is defined by 
\begin{equation*}
	\alpha = \max_{i} \max\left\{\varrho_1(\bm{ U}_{i\oplus\frac12}),~\varrho_1(\bm{U}_{(i+1)\oplus\left(-\frac12\right)})\right\},
\end{equation*}
 where
$\varrho_1(\bm{U})$ denotes the spectral radius of the Jacobian matrix $\frac{\partial \bm F_1(\bm U)}{\partial \bm U}$. %denotes the maximum spectral radius of the Jacobian matrix $\frac{\partial \bm F_1(\bm U)}{\partial \bm U}$ in all the cells

\begin{rem}
It should be pointed out that the symbol ``$\oplus$" in the subscript of $\bm{U}_{i\oplus a}$ with $a \in [-\frac12,\frac12]$ is not a standard addition operation, but rather a symbol representing the position relative to the cell center $x_i$. For example, $\bm{U}_{i \oplus \frac12}$ represents the approximate value at the point $x_i+\frac{1}{2}\Delta x$ computed within the cell $I_i$, while $\bm{U}_{(i+1) \oplus \left(-\frac12\right)}$ stands for the approximate value at the same point $x_{i+1}-\frac{1}{2}\Delta x$ but computed within the cell $I_{i+1}$.
\end{rem}

To compute $\mathcal{L}_U(\bm{U}_h(t),i)$ and $\mathcal{L}_V(\bm{U}_h(t),i)$ in \eqref{eq:1D discrete equation1}--\eqref{eq:1D discrete equation2}, one needs to reconstruct the point values $\{\bm{U}_{i\oplus a_\ell}\}_{\ell=1}^4$ by using $\{\bm{\overline{U}}_i\}$ and $\{\bm{\overline{V}}_i\}$. 
%The initial discretization of zeroth and first order moment is $\bm{\overline{U}}\left(0\right)=\frac{1}{\Delta x}\int_{I_i}\bm{U}(x,0)\mathrm{d}x$ and $\bm{\overline{V}}\left(0\right)=\frac{1}{\Delta x}\int_{I_i}\bm{U}(x,0)\frac{x-x_i}{\Delta x}\mathrm{d}x$, we can use 3 points Gauss integral formula numerically.
To facilitate the subsequent description, we first introduce two reconstruction operators $\bm{M}_L(\cdot,\cdot,\cdot)$ and $\bm{M}_H(\cdot,\cdot,\cdot)$, before giving the detailed spatial reconstruction procedures of our 1D PCP finite volume HWENO scheme.

\subsubsection{Linear reconstruction operator $\bm{M}_L$}

Let us reconstruct a quintic polynomial $P_0(x)=\sum\limits_{l=0}^{5}a^0_l\left(\frac{x-x_i}{\Delta x}\right)^l$ satisfying
\begin{align} \label{eq:312}
	&\frac{1}{\Delta x}\int_{I_k}P_0\left(x\right)\mathrm{d}x=u_k,~k=i,i\pm1, 
	\\ \label{eq:313}
	&\frac{1}{\Delta x}\int_{I_k}P_0\left(x\right)\frac{x-x_i}{\Delta x}\mathrm{d}x=v_k,~k=i,i\pm1,
\end{align}
where $u_k$ and $v_k$ are given real numbers. The six equations \eqref{eq:312}--\eqref{eq:313} 
form a linear algebraic system for the unknowns $\{a^0_l\}_{l=0}^5$. Solving the system gives 
the expressions of $a_l^0$, which are the linear combinations of $u_i$, $v_i$, $u_{i\pm1}$, $v_{i\pm1}$ given by 
\begin{eqnarray*}\label{1Dp0coeff}
	\left\{
	\begin{aligned}[c]
		&a_0^0=-\frac{43}{384}u_{i-1} + \frac{235}{192}u_i -\frac{43}{384}u_{i+1} - \frac{27}{64}v_{i-1} + \frac{27}{64}v_{i+1}, \\
		&a_1^0=\frac{167}{576}u_{i-1}-\frac{167}{576}u_{i+1}+\frac{281}{288}v_{i-1}+\frac{2449}{144}v_i+\frac{281}{288}v_{i+1},\\
		&a_2^0=\frac{23}{16}u_{i-1}-\frac{23}{8}u_i+\frac{23}{16}u_{i+1}+\frac{45}{8}v_{i-1}-\frac{45}{8}v_{i+1},\\
		&a_3^0=-\frac{455}{216}u_{i-1}+\frac{455}{216}u_{i+1}-\frac{785}{108}v_{i-1}-\frac{1945}{54}v_i-\frac{785}{108}v_{i+1},\\
		&a_4^0=-\frac{5}{8}u_{i-1}+\frac{4}{5}u_i-\frac{5}{8}u_{i+1}-\frac{15}{4}v_{i-1}+\frac{15}{4}v_{i+1},\\
		&a_5^0=\frac{35}{36}u_{i-1}-\frac{35}{36}u_{i+1}+\frac{77}{18}v_{i-1}+\frac{133}{9}v_i+\frac{77}{18}v_{i+1}.
	\end{aligned}
	\right.
\end{eqnarray*}
Define $\eta=\frac{x-x_i}{\Delta x}$ and the operator 
\begin{equation*}
	M_L([u_{i-1} ~u_i ~u_{i+1}],[v_{i-1} ~v_i ~v_{i+1}],\eta):=P_0(x(\eta))=\sum_{l=0}^{5}a^0_l\eta^l,
\end{equation*}
which is a mapping from $\mathbb{R}^{1\times3}\times\mathbb{R}^{1\times3}\times\mathbb{R}$ to $\mathbb{R}$. Using this operator, it is easy to compute the value of $P_0(x)=$ at $x_{i+\eta}$ with $P_0(x_{i+\eta})=M_L([u_{i-1} ~u_i ~u_{i+1}],[v_{i-1} ~v_i ~v_{i+1}],\eta)$. For example, 
$$P_0(x_{i+1/2})=M_L([u_{i-1} ~u_i ~u_{i+1}],[v_{i-1} ~v_i ~v_{i+1}],\frac12)=\frac{13}{108}u_{i-1}+\frac{7}{12}u_i+\frac{8}{27}u_{i+1}+\frac{25}{54}v_{i-1}+\frac{241}{54}v_i-\frac{28}{27}v_{i+1},$$ 
which is independent of the cell size $\Delta x$ and the cell center $x_i$.

The operator $M_L$ represents the reconstruction mapping for the scalar equation. In order to extend the reconstruction to the 1D RHD equations, we generalize the operator to vector cases component-wisely as follows
\begin{equation*}
	\bm{M}_L([\bm{U}_1~\bm{U}_2~\bm{U}_3],[\bm{V}_1~\bm{V}_2~\bm{V}_3],\eta):=
	\begin{pmatrix}
		M_L([{U}_1^{(1)}~{U}_2^{(1)}~{U}_3^{(1)}],[{V}_1^{(1)}~{V}_2^{(1)}~{V}_3^{(1)}],\eta) \\
		M_L([{U}_1^{(2)}~{U}_2^{(2)}~{U}_3^{(2)}],[{V}_1^{(2)}~{V}_2^{(2)}~{V}_3^{(2)}],\eta) \\
		M_L([{U}_1^{(3)}~{U}_2^{(3)}~{U}_3^{(3)}],[{V}_1^{(3)}~{V}_2^{(3)}~{V}_3^{(3)}],\eta)
	\end{pmatrix},
\end{equation*}
where $\bm{M}_L$ is the reconstruction operator from $\mathbb{R}^{3\times3}\times\mathbb{R}^{3\times3}\times\mathbb{R}$ to $\mathbb{R}^{3\times1}$. It is worth pointing out that $\bm{M}_L(\cdot,\cdot,\eta)$ is also a linear mapping for a fixed $\eta$.

\subsubsection{HWENO reconstruction operator $\bm{M}_H$}

Consider two quadratic polynomials $P_1(x)=\sum\limits_{l=0}^{2}a^1_l\left(\frac{x-x_i}{\Delta x}\right)^l$ and $P_2(x)=\sum\limits_{l=0}^{2}a^2_l\left(\frac{x-x_i}{\Delta x}\right)^l$ satisfying
\begin{align}
	&\frac{1}{\Delta x}\int_{I_i}P_1\left(x\right)\frac{x-x_i}{\Delta x}\mathrm{d}x=v_i,\quad \frac{1}{\Delta x}\int_{I_k}P_1\left(x\right)\mathrm{d}x=u_k,\quad k=i,i-1,\nonumber\\ 
	&\frac{1}{\Delta x}\int_{I_i}P_2\left(x\right)\frac{x-x_i}{\Delta x}\mathrm{d}x=v_i,\quad \frac{1}{\Delta x}\int_{I_k}P_2\left(x\right)\mathrm{d}x=u_k,\quad k=i,i+1.\nonumber
\end{align}
Similarly, we can obtain the expressions of $a_l^1$ and $a_l^2$, which are also linear combinations of $u_i$, $v_i$, $u_{i\pm1}$, $v_{i\pm1}$, given by 
\begin{eqnarray*}\label{1Dp12coeff}
	\left\{
	\begin{aligned}[c]
		&a_0^1=-\frac{1}{12}u_{i-1}+\frac{13}{12}u_i-v_i, \\
		&a_1^1=12v_i,\\
		&a_2^1=u_{i-1}-u_i+12v_i,\\
		&a_0^2=\frac{13}{12}u_i-\frac{1}{12}u_{i+1}+v_i,\\
		&a_1^2=12v_i,\\
		&a_2^2=-u_i+u_{i+1}-12v_i.
	\end{aligned}
	\right.
\end{eqnarray*}

Next, in order to measure the smoothness of the polynomial $P_n\left ( x\right )$ in the cell $I_i$, we calculate the smooth indicators, with the same definition as in \cite{zhao2020hermite},
\begin{equation}\label{smooth indicator}
	\beta_n=\sum_{\alpha=1}^{r}\int_{I_i}\Delta x^{2\alpha-1}\left ({\frac{\mathrm{d}^\alpha P_n\left ( x\right )}{\mathrm{d}x^{\alpha}}} \right )^2{\mathrm{d}x},~n=0,1,2,
\end{equation}
\noindent where $r$ is the degree of the polynomials $P_n(x)$. The expressions of $\beta_n$ are
\begin{eqnarray*}\label{smooth indicator_express}
	\left\{  
 \begin{aligned}
	\beta_0=&\left(\frac{19}{108} {u}_{i-1}-\frac{19}{108} {u}_{i+1}+\frac{31}{54} {v}_{i-1}-\frac{241}{27} {v}_i+\frac{31}{54} {v}_{i+1}\right)^2+\left(\frac{9}{4} {u}_{i-1}-\frac{9}{2} {u}_i+\frac{9}{4} {u}_{i+1}+\right. \\
	&\left.\frac{15}{2} {v}_{i-1}-\frac{15}{2} {v}_{i+1}\right)^2+\left(\frac{70}{9} {u}_{i-1}-\frac{70}{9} {u}_{i+1}+\frac{200}{9} {v}_{i-1}+\frac{1280}{9} {v}_i+\frac{200}{9} {v}_{i+1}\right)^2+ \\
	&\frac{1}{12}\left(\frac{5}{2} {u}_{i-1}-5 {u}_i+\frac{5}{2} {u}_{i+1}+9 {v}_{i-1}-9 {v}_{i+1}\right)^2+\frac{1}{12}\left(\frac{175}{18} {u}_{i-1}-\frac{175}{18} {u}_{i+1}+\frac{277}{9} {v}_{i-1}+\right. \\
	&\left.\frac{1546}{9} {v}_i+\frac{277}{9} {v}_{i+1}\right)^2+\frac{1}{180}\left(\frac{95}{18} {u}_{i-1}-\frac{95}{18} {u}_{i+1}+\frac{155}{9} {v}_{i-1}+\frac{830}{9} {v}_i+\frac{155}{9} {v}_{i+1}\right)^2+ \\
	&\frac{109341}{175}\left(\frac{5}{8} {u}_{i-1}-\frac{5}{4} {u}_i+\frac{5}{8} {u}_{i+1}+\frac{15}{4} {v}_{i-1}-\frac{15}{4} {v}_{i+1}\right)^2+\frac{27553933}{1764}\left(\frac{35}{36} {u}_{i-1}-\frac{35}{36} {u}_{i+1}+\right. \\
	&\left.\frac{77}{18} {v}_{i-1}+\frac{133}{9} {v}_i+\frac{77}{18} {v}_{i+1}\right)^2,\\
	\beta_1=&144v_i^2+\frac{13}{3}(u_{i-1}-u_i+12v_i)^2,\\
	\beta_2=&144v_i^2+\frac{13}{3}(u_i-u_{i+1}+12v_i)^2.
\end{aligned}
\right.
\end{eqnarray*}
\noindent Then the HWENO reconstruction polynomial is defined by
\begin{equation}
	P_{H}\left(x\right)=\omega_0 \left(\frac{1}{\gamma_0}P_0\left(x\right)-\sum_{n=1}^2\frac{\gamma_n}{\gamma_0}P_n\left(x\right)\right)+\sum_{n=1}^2\omega_nP_n\left(x\right), \nonumber
\end{equation}

\noindent where the nonlinear weights
\begin{equation}\label{nonlinear weights 2}
	\omega_n=\frac{\bar{\omega}_n}{\sum_{k=0}^2\bar{\omega}_k} \quad  \text{ with } \quad  \bar{\omega}_n=\gamma_n\left(1+\frac{\tau^2}{\beta_n^2+\epsilon}\right),~ n=0,1,2,
\end{equation}
$\tau:=\frac{\left |\beta_0-\beta_1 \right |+\left|\beta_0-\beta_2 \right |}{2}$, and $\epsilon$ is a small positive number to avoid the denominator being zero. These nonlinear weights possess a ``scale-invariant'' property, which means that the nonlinear weights $\{\omega_n\}$ remain unchanged when $\{u_i, v_i, u_{i\pm1}, v_{i\pm1}\}$ are replaced by $\{\lambda u_i, \lambda v_i, \lambda u_{i\pm1}, \lambda v_{i\pm1}\} $ for any $\lambda \neq 0$. 

Let $\eta:=\frac{x-x_i}{\Delta x}$. Define the operator 
\begin{equation*}
	\begin{aligned}[c]
		M_H([u_{i-1} ~u_i ~u_{i+1}],[v_{i-1} ~v_i ~v_{i+1}],\eta)&:=P_H(x(\eta))\\
		&=\omega_0 \left(\frac{1}{\gamma_0}\sum\limits_{l=0}^{5}a^0_l\eta^l-\sum_{n=1}^2\frac{\gamma_n}{\gamma_0}\sum\limits_{l=0}^{5}a^n_l\eta^l\right)+\sum_{n=1}^2\omega_n\sum\limits_{l=0}^{5}a^n_l\eta^l,
	\end{aligned}
\end{equation*}
which is a mapping from $\mathbb{R}^{1\times3}\times\mathbb{R}^{1\times3}\times\mathbb{R}$ to $\mathbb{R}$. It is easy to compute the value of $P_H(x)$ at $x_{i+\eta}$ with $P_H(x_{i+\eta})=M_H([u_{i-1} ~u_i ~u_{i+1}],[v_{i-1} ~v_i ~v_{i+1}],\eta)$.
We can generalize the scalar HWENO reconstruction operator $M_H$ to the vector cases in a component by component manner: 
\begin{equation*}
	\bm{M}_H([\bm{U}_1~\bm{U}_2~\bm{U}_3],[\bm{V}_1~\bm{V}_2~\bm{V}_3],\eta):=
	\begin{pmatrix}
		M_H([{U}_1^{(1)}~{U}_2^{(1)}~{U}_3^{(1)}],[{V}_1^{(1)}~{V}_2^{(1)}~{V}_3^{(1)}],\eta) \\
		M_H([{U}_1^{(2)}~{U}_2^{(2)}~{U}_3^{(2)}],[{V}_1^{(2)}~{V}_2^{(2)}~{V}_3^{(2)}],\eta) \\
		M_H([{U}_1^{(3)}~{U}_2^{(3)}~{U}_3^{(3)}],[{V}_1^{(3)}~{V}_2^{(3)}~{V}_3^{(3)}],\eta)
	\end{pmatrix},
\end{equation*}
where ${U}_i^{(\ell)}$ is the $\ell$th component of ${\bm U}_i$, ${V}_i^{(\ell)}$ is the $\ell$th component of ${\bm V}_i$. 
Different from $\bm{M}_L$, the operator  
$\bm{M}_H(\cdot,\cdot,\eta)$ is a nonlinear mapping for a fixed $\eta$.

\begin{rem}
When solving the relativistic hydrodynamics (RHD) equations, the wide range of variable scales arising from relativistic effects in the ultra-relativistic regime can significantly impact the effectiveness of shock capturing. As recently demonstrated in \cite{chen2022physical}, using a scale-invariant nonlinear weights can effectively suppress oscillations for simulating multiscale RHD problems.
\end{rem}

%\begin{rem}
%	In our computations, we simply take $\epsilon$ in \eqref{nonlinear weights 2} as $10^{-6}$.
%\end{rem}

\subsubsection{Detailed PCP HWENO reconstruction procedure}\label{sec:detailed}

We are now in a position to present the detailed PCP HWENO reconstruction procedure of our 1D HWENO scheme. 

\begin{description}
	\item[Step 1.]
	Use the KXRCF indicator \cite{krivodonova2004shock} to identify the troubled cells, which are the cells where the solution may be discontinuous. Then modify the first-order moment in the troubled cells by using the HWENO limiter given in \cite{zhao2020hermite}. We observe that the nonlinear weights in the HWENO limiter are also necessary to be scale-invariant, thus we modify the nonlinear weights in the HWENO limiter \cite{zhao2020hermite} to  $\omega_n^l=\frac{\bar{\omega}_n^l}{\sum_{k=0}^2\bar{\omega}_k^l}$ with  $\bar{\omega}_n^l=\gamma_n\left(1+\frac{\tau_l^2\Delta x}{\left(\beta_n^l\right)^2+\epsilon}\right), n=0,1,2$.
	
	\item[Step 2.]
	Reconstruct the point values of the solution at the four Gauss--Lobatto points.
	
	\begin{itemize}
		\item If cell $I_i$ is not a troubled cell, employ the linear reconstruction:
		$$\bm{U}_{i\oplus a_\ell}^{*}=\bm{M}_L\left(\left[\overline{\bm{U}}_{i-1} ~ \overline{\bm{U}}_{i} ~ \overline{\bm{U}}_{i+1} \right],\left[\overline{\bm{V}}_{i-1} ~ \overline{\bm{V}}_{i} ~ \overline{\bm{V}}_{i+1} \right],a_{\ell}\right), \quad \ell\in\left \{1,2,3,4\right \}.$$
		\item If $I_i$ is a troubled cell, use the HWENO reconstruction.  
		\\
		(i) Perform HWENO reconstruction in a component-by-component fashion for the second and third Gauss-Lobatto points:
		$$\bm{U}_{i\oplus a_\ell}^{*}=\bm{M}_H\left(\left[\overline{\bm{U}}_{i-1} ~ \overline{\bm{U}}_{i} ~ \overline{\bm{U}}_{i+1} \right],\left[\overline{\bm{V}}_{i-1} ~ \overline{\bm{V}}_{i} ~ \overline{\bm{V}}_{i+1} \right],a_{\ell}\right),\quad \ell\in\left\{2,3\right\}.$$
		(ii) Perform HWENO reconstruction based on characteristic decomposition for the cell interface points:
		$$\bm{U}_{i\oplus \left(\pm\frac12\right)}^{*}=\bm{R}_{i\pm\frac{1}{2}}\bm{M}_H\left(\bm{R}^{-1}_{i\pm\frac{1}{2}}\left[\overline{\bm{U}}_{i-1} ~ \overline{\bm{U}}_{i} ~ \overline{\bm{U}}_{i+1} \right],\bm{R}^{-1}_{i\pm\frac{1}{2}}\left[\overline{\bm{V}}_{i-1} ~ \overline{\bm{V}}_{i} ~ \overline{\bm{V}}_{i+1} \right],\pm\frac{1}{2}\right),$$
		where $\bm{R}^{-1}_{i+1/2}$ and $\bm{R}_{i+1/2}$ are taken as left and right eigenvector matrices of  the Roe matrix \cite{eulderink1994general} at $x_{i+1/2}$. The eigenvector matrices and fluxes are computed by using the primitive variables, which are recovered by using the proposed NR methods. 
		
		% $\bm{A}\left(\overline{\bm{U}}_{i},\overline{\bm{U}}_{i+1}\right)$
	\end{itemize}
	\item[Step 3.] Perform the PCP limiter on the reconstructed point values $\{\bm{U}_{i\oplus a_\ell}^{*} \}_{\ell =1}^4$ as follows. 
	Define $\bm{U}_{i\oplus a_\ell}^*=: (D^*_{i\oplus a_\ell}, (m_1)^*_{i\oplus a_\ell}, E^*_{i\oplus a_\ell})^\top$ and the first component of $\overline {\bm U}_i$ as $\overline D_i$.  
	\begin{itemize}
		\item  Modify the mass density to enforce its positivity via 
		\begin{align*}
			&\widetilde{D}_{i\oplus a_\ell}=\theta_D(D^*_{i\oplus a_\ell}-\overline{D}_i)+\overline{D}_i \quad \mbox{with} \quad \theta_D=\min\left \{\left |\frac{\overline{D}_i-\epsilon_D}{\overline{D}_i-D_{\min}} \right |,1\right \},
			\\
			&D_{\min}=\min\left\{\min_\ell \{D^*_{i\oplus a_\ell}\},\frac{\overline{D}_i-\hat{\omega}_1\bm{U}_{i\oplus \left(-\frac12\right)\!}^{*}-\hat{\omega}_4\bm{U}_{i\oplus \frac{1}{2}}^{*}}{1-2\hat{\omega}_1}\right\},
		\end{align*}
	  where $\epsilon_D=\min\limits_{i}\left \{10^{-13},\overline{D}_{i}\right \}$ is a small positive number introduced to avoid the effect of round-off errors. 
	  \item Define $\widetilde{\bm{U}}_{i\oplus a_\ell}=(\widetilde{D}_{i\oplus a_\ell}, (m_1)^*_{i\oplus a_\ell}, E^*_{i\oplus a_\ell})^\top$. 
	  Enforce the positivity of $g({\bm U})$ by 
	  \begin{align} \label{eq:PCPlimiter}
	  	&  \bm{U}_{i\oplus a_\ell}=\theta_g(\widetilde{\bm{U}}_{i\oplus a_\ell}-\overline{\bm{U}}_i)+\overline{\bm{U}}_i \quad \mbox{with} \quad \theta_g=\min\left \{\left |\frac{g(\overline{\bm{U}}_i)-\epsilon_g}{g(\overline{\bm{U}}_i)-g_{\min}} \right |,1\right \},
	  	\\ \nonumber
	  	& g_{\min} = \min\left\{ \min_\ell g(\widetilde {\bm U}_{i\oplus a_\ell}^{*}) ,g\left (\frac{\overline{\bm{U}}_i-\hat{\omega}_1\widetilde {\bm U}_{i\oplus \left(-\frac12\right)}^{*}-\hat{\omega}_4\widetilde {\bm U}_{i\oplus \frac{1}{2}}^{*}}{1-2\hat{\omega}_1} \right )\right\},
	  \end{align}
  	where $\epsilon_g=\min\limits_{i}\left \{10^{-13},g(\overline{\bm{U}}_i)\right \}$.
	\end{itemize}
%	where $\epsilon_D=\min\limits_{i}\left \{10^{-13},\overline{D}_{i}^n\right \}$, $\epsilon_g=\min\limits_{i}\left \{10^{-13},g(\overline{\bm{U}}_i^n)\right \}$ are two small positive numbers, $g^*$ and $D^*$ are parameters that are designed to be different for different point values. For $\bm{U}_{i+a_2}^{*}$, $g^*=g(\bm{U}_{i+a_2}^{*})$, $D^*=D_{i+a_2}^{*}$; for $\bm{U}_{i+a_3}^{*}$, $g^*=g(\bm{U}_{i+a_3}^{*})$, $D^*=D_{i+a_3}^{*}$; and for the reconstruction point values on the interface $\bm{U}_{i\pm\frac{1}{2}}^{*}$,
%	\begin{equation*}
%		\left\{
%		\begin{aligned}
%		g^*&=\min\left\{g(\bm{U}_{i-\frac{1}{2}}^{*}),g(\bm{U}_{i+\frac{1}{2}}^{*}),g\left (\frac{\overline{\bm{U}}_i^n-\hat{\omega}_1\bm{U}_{i-\frac{1}{2}}^{*}-\hat{\omega}_4\bm{U}_{i+\frac{1}{2}}^{*}}{1-2\hat{\omega}_1} \right )\right\},\\
%		D^*&=\min\left\{D_{i-\frac{1}{2}}^{*},D_{i+\frac{1}{2}}^{*},\frac{\overline{D}_i^n-\hat{\omega}_1\bm{U}_{i-\frac{1}{2}}^{*}-\hat{\omega}_4\bm{U}_{i+\frac{1}{2}}^{*}}{1-2\hat{\omega}_1}\right\}.
%		\end{aligned}
%		\right.
%	\end{equation*}
\end{description}

\begin{rem}\label{rem:eigvector}
	The characteristic decomposition is very important for 
	 the HWENO reconstruction  for the cell interface points. 
	However, the wide range of characteristic variable scales resulting from relativistic effects can lead to large numerical errors in characteristic decomposition, especially in challenging test problems, where this problem is particularly severe in the HWENO scheme. To mitigate this issue, we propose to rescale the characteristic vectors. 
	Let $(\bm{r},\bm{l})$ be a pair of right and left eigenvectors of the Jacobian matrix $\frac{\partial {\bm F}}{\partial {\bm U}}$, then $(c\bm{r},\bm{l}/c)$ is also a pair of eigenvectors, where $c\neq0$.  
	We consider the following two rescaling approaches:
	\begin{enumerate}[\bfseries (i)]
		\item (Unitization) Unitize the left eigenvectors with $c=\left |\bm{l} \right |$. However, this method may result in extremely large values of $\left |c\bm{r}\right |$. 
		\item (Matching) Matching the norms between $c\bm{r}$ and $\bm{l}/c$ with $c=\sqrt{\frac{\left |\bm{l}\right|}{\left |\bm{r}\right|}}$, such that $\left |\bm{l}\right|/c=c\left |\bm{r}\right|$.
	\end{enumerate}
	In Section \ref{sec:numerical_tests}, we will present a numerical example to demonstrate the necessity of rescaling eigenvectors and compare the effectiveness of these two approaches. The results show that the ``matching'' approach is superior to the ``unitization'' approach.
\end{rem}

\begin{rem}
	As demonstrated in Section \ref{sec:Pressure recovering algorithms}, the algorithms recovering primitive variables from $\bm{U}$ are convergent only when $\bm{U}\in \mathcal{G}$. Therefore, for the RHD equations, the pointwise PCP property at a given point is necessary whenever the fluxes are computed at that point. This differs from the Euler equations \cite{zhang2010positivity}. For this reason, we employ the PCP limiting procedures for the reconstructed values at the inner Gauss--Lobatto points $\bm{U}_{i+a_\ell}$ $(\ell=2,3)$ in Step 3. % of Section \ref{sec:detailed}
\end{rem}

\subsubsection{Time discretization}\label{section:RK}
To obtain a fully discrete scheme, 
we use the strong-stability-preserving (SSP) Runge--Kutta methods to further discretize the semi-discrete scheme (\ref{eq:1D discrete equation1})--(\ref{eq:1D discrete equation2}) in time. For example, the third-order SSP Runge--Kutta method reads 
\begin{equation}\label{eq:SSPRK3}
	\begin{aligned}
		\bm{\overline{W}}^{(1)}_i &=\bm{\overline{W}}_i^n+\Delta t \mathcal{L}_W(\bm{U}_h^n,i) ,\\
		\bm{\overline{W}}^{(2)}_i &=\frac{3}{4}\bm{\overline{W}}^n_i+\frac{1}{4}\left(\bm{\overline{W}}^{(1)}_i+\Delta t\mathcal{L}_W(\bm{U}_h^{(1)},i)\right) ,\\
		\bm{\overline{W}}^{n+1}_i &=\frac{1}{3}\bm{\overline{W}}^n_i+\frac{2}{3}\left(\bm{\overline{W}}^{(2)}_i+\Delta t\mathcal{L}_W(\bm{U}_h^{(2)},i)\right),
	\end{aligned}
\end{equation}
where the symbol ``$\bm W$'' can be replaced with ``$\bm U$'' or ``$\bm V$''.

\subsection{Rigorous analysis of PCP property}\label{1D PCP limiter}

In this subsection, we present a rigorous analysis of PCP property of the proposed HWENO scheme.

First, we recall the following Lax--Friedrichs splitting property, which was proved for the RHD equations in \cite{2015High}. 

\begin{lemma}\label{lemma:LF pcp}
	If $\bm{U} \in{\mathcal{G}}$, then 
	$$\bm{F}^{\pm}_\ell (\bm{U},\alpha):=\bm{U}\pm\alpha^{-1}\bm{F}_\ell (\bm{U} )\in {\mathcal{G}}$$ 
	for any $\alpha\geq \varrho_\ell (\bm{U})$, $\ell=1,\dots,d.$
\end{lemma}

Similar to \cite[Proposition 3.1]{chen2022physical}, we can prove that the PCP limited point values \eqref{eq:PCPlimiter} satisfy the following property. 

\begin{lemma}\label{1D pcp inequality}
	If $\overline{\bm{U}}_i\in {\mathcal{G}}$, then the PCP limited point values \eqref{eq:PCPlimiter} satisfy 
	\begin{equation}\label{key311}
		\bm{U}_{i\oplus a_\ell}\in {\mathcal{G}} \quad \forall \ell \in \{1,2,3,4\}, \qquad 
		{\bm \Pi}_{i} := \frac{\overline{\bm{U}}_i-\hat{\omega}_1 {\bm U}_{i\oplus \left(-\frac12\right)}-\hat{\omega}_4 {\bm U}_{i\oplus \frac{1}{2}}}{1-2\hat{\omega}_1} 
		\in {\mathcal{G}}. 
	\end{equation}
\end{lemma}

Based on the above two lemmas, we are now ready to show the PCP property of the proposed HWENO scheme. 

\begin{Theorem}\label{1D PCP thm}
	Consider the proposed 1D HWENO scheme \eqref{eq:1D discrete equation1}--\eqref{eq:1D discrete equation2} with the Lax--Friedrichs flux \eqref{1D LF}. If $\overline {\bm U}_i \in {\mathcal G}$ for all $i$, then 
	the updated cell averages 
	\begin{equation}\label{key333}
		\bm{\overline{U}}_i+\Delta t \mathcal{L}_U(\bm{U}_h(t),i) \in {\mathcal G} \qquad \forall i
	\end{equation}
	under the CFL condition
	\begin{equation}\label{cond:cfl1D}
		\Delta t\leq \frac{\hat{\omega}_1\Delta x }{ \alpha}.
	\end{equation}
\end{Theorem}

\begin{proof}
	Based on \eqref{key311}, we have the following convex decomposition for the cell average:
	\begin{equation*}
		\bm{\overline{U}}_i = (1-2\hat{\omega}_1)\bm{\Pi}_i +\hat{\omega}_1\bm{U}_{i\oplus \left(-\frac12\right)}+\hat{\omega}_1\bm{U}_{i\oplus \frac12}
	\end{equation*}
with $\bm{\Pi}_i\in {\mathcal G}$, $\bm{U}_{i\oplus \left(-\frac12\right)}\in {\mathcal G}$, and $\bm{U}_{i\oplus \frac12} \in {\mathcal G}$. 
It follows that 
\begin{align*}
		\bm{\overline{U}}_i+\Delta t \mathcal{L}_U(\bm{U}_h(t),i) &= \overline{\bm{U}}_i- \frac{\Delta t}{\Delta x} \Big( \hat{\bm{F}}_{i+1/2}-\hat{\bm{F}}_{i-1/2} \Big)\\
		&=(1-2\hat{\omega}_1){\bm \Pi}_{i}+\hat{\omega}_1\bm{U}_{i\oplus \left(-\frac12\right)}+\hat{\omega}_1\bm{U}_{i\oplus \frac12}
		\\
		& \quad -\frac{\Delta t}{2\Delta x}\Big\{\left[\bm{F}_1\left(\bm{ U}_{i\oplus \frac12}\right)\!+\!\bm{F}_1\left(\bm{ U}_{(i+1)\oplus \left(-\frac12\right)}\right) - \alpha\left(\bm{ U}_{(i+1)\oplus \left(-\frac12\right)}\!-\!\bm{U}_{i\oplus \frac12}\right)\right]\!
		\\
		& \quad 
		-\!\left[\bm{F}_1\left(\bm{ U}_{i\oplus \left(-\frac12\right)}\right)\!+\!\bm{F}_1\left(\bm{ U}_{(i-1)\oplus \frac12}\right)\!-\!\alpha\left(\bm{U}_{i\oplus \left(-\frac12\right)}\!-\!\bm{ U}_{(i-1)\oplus \frac12}\right)\right]\Big\}
		\\
		&=(1-2\hat{\omega}_1)\bm{\Pi }_i+ 
		\left ( \hat{\omega}_1-\frac{\alpha\Delta t}{\Delta x}\right )\bm{U}_{i\oplus \left(-\frac12\right)}+
		\left (\hat{\omega}_1-\frac{\alpha\Delta t}{\Delta x}\right )\bm{U}_{i\oplus \frac12}
		\\
		& \quad + \frac{\alpha\Delta t}{2\Delta x}\Big  [\bm{F}^{-}_1(\bm{U}_{(i+1)\oplus \left(-\frac12\right)},\alpha)+\bm{F}^{-}_1(\bm{U}_{i\oplus \frac12},\alpha)\Big]\\
		& \quad +\frac{\alpha\Delta t}{2\Delta x}\Big [\bm{F}^{+}_1(\bm{U}_{(i-1)\oplus \frac12},\alpha)+\bm{F}^{+}_1(\bm{U}_{i\oplus \left(-\frac12\right)},\alpha)\Big ].
\end{align*}
Under the CFL condition \eqref{cond:cfl1D}, $\bm{\overline{U}}_i+\Delta t \mathcal{L}_U(\bm{U}_h(t),i)$
has been reformulated into a convex combination form. 
Thanks to Lemma \ref{lemma:LF pcp} and the convexity of $\mathcal G$, we have $\bm{\overline{U}}_i+\Delta t \mathcal{L}_U(\bm{U}_h(t),i) \in {\mathcal G}$ for all $i$. The proof is completed. 
\end{proof}

Theorem \ref{1D PCP thm} implies that the proposed HWENO scheme is PCP if the forward Euler method is used for time discretization. Since the SSP Runge--Kutta method \eqref{eq:SSPRK3} is formally a convex combination of forward Euler, the PCP property remains valid for the fully discrete HWENO scheme \eqref{eq:SSPRK3}, due to the convexity of $\mathcal G$.

%In addition to PCP property, the limiter applied in Step 3 in Section \ref{section:1D PCP HWENO} has the following properties, which can be proved easily with methods similar to \cite{chen2022physical,zhang2010maximum,zhang2010positivity,zhang2011maximum,cai2016positivity}:
%\begin{enumerate}[\bfseries (1)]
%	\setlength{\itemindent}{1em}  % item项目空1格！
%	\item The limiter doesn't destroy the conservativeness of the numerical solution. 
%	\item The limiter maintains the high-order accuracy of the HWENO scheme when the entropy solution is smooth.
%\end{enumerate}
\section{Two-dimensional PCP HWENO scheme}\label{sec:2D HWENO scheme}
In this section, we present the PCP finite volume HWENO scheme for the 2D special RHD equations
\begin{equation}\label{eq:rhd_2D}
	\frac{\partial\bm{U}}{\partial t}+\frac{\partial\bm{F}_1(\bm{U})}{\partial x}+\frac{\partial\bm{F}_2(\bm{U})}{\partial y}={\bf 0},
\end{equation}
where 
\begin{equation*}%\label{eq:2D flux and conserv variable}
	\bm{U}=\left(D,m_1,m_2,E\right)^\top,\quad \bm{F}_1=\left(Dv_1,m_1v_1+p,m_2v_1,m_1\right)^\top,\quad
	\bm{F}_2=\left(Dv_2,m_1v_2+p,m_2v_2+p,m_2\right)^\top.
\end{equation*}
Divide the computational domain into $N_x\times N_y$ uniform cells $I_{i,j}=\left [ x_{i-1/2},x_{i+1/2}\right] \times \left [y_{j-1/2},y_{j+1/2}\right]$, $1\leq i\leq N_x$, $1\leq j\leq N_y$, with the cell center $(x_i,y_j)=\left(\frac{x_{i-1/2}+x_{i+1/2}}{2},\frac{y_{j-1/2}+y_{j+1/2}}{2}\right)$. Let $\Delta x=x_{i+1/2}-x_{i-1/2}$ and $\Delta y=y_{j+1/2}-y_{j-1/2}$ denote the spatial step-sizes in the $x$- and $y$-directions, respectively.

%For simplicity, let $x_{i+a}=x_i+a\Delta x$, $y_{j+b}=y_j+b\Delta y$, where $a$ and $b$ are real numbers.

\subsection{2D PCP finite volume HWENO scheme}\label{2D HWENO}

In this subsection, we give the 2D finite volume PCP HWENO scheme. For the notational convenience, we denote  $x_{i+a}=x_i+a\Delta x$ and $y_{j+b}=y_j+b\Delta x$, where $a \in [-\frac12,\frac12]$ and $b \in [-\frac12,\frac12]$ are real numbers.

The semi-discrete finite volume HWENO scheme for the RHD equations \eqref{eq:rhd_2D} is given by 
\begin{align}\label{eq:2D discrete equation1}
		\frac{\mathrm{d}\overline{\bm{U}}_{i,j}\left(t\right)}{\mathrm{d}t}=&-\sum_{\ell=1}^{3}\omega_\ell\frac{\hat{\bm{F}}^1_{i+\frac{1}{2},j+b_\ell}-\hat{\bm{F}}^1_{i-\frac{1}{2},j+b_\ell}}{\Delta x}-\sum_{\ell=1}^{3}\omega_\ell\frac{\hat{\bm{F}}^2_{i+b_\ell,j+\frac{1}{2}}-\hat{\bm{F}}^2_{i+b_\ell,j-\frac{1}{2}}}{\Delta y}:=\mathcal{L}_U(\bm{U}_h(t),i,j),
		\\
		\frac{\mathrm{d}\overline{\bm{V}}_{i,j}\left(t\right)}{\mathrm{d}t}=&-\sum_{\ell=1}^{3}\omega_\ell\frac{\hat{\bm{F}}^1_{i+\frac{1}{2},j+b_\ell}+\hat{\bm{F}}^1_{i-\frac{1}{2},j+b_\ell}}{2\Delta x}-\sum_{\ell=1}^{3}\omega_\ell b_\ell\frac{\hat{\bm{F}}^2_{i+b_\ell,j+\frac{1}{2}}-\hat{\bm{F}}^2_{i+b_\ell,j-\frac{1}{2}}}{\Delta y}\nonumber
		\\
		\label{eq:2D discrete equation2}
		&+\sum_{k=1}^{3}\sum_{\ell=1}^{3}\frac{\omega_k\omega_\ell \bm{F}_1\left(\bm{U}_{i\oplus b_\ell,j\oplus b_k}\right)}{\Delta x}:=\mathcal{L}_V(\bm{U}_h(t),i,j),
		\\
		\frac{\mathrm{d}\overline{\bm{W}}_{i,j}\left(t\right)}{\mathrm{d}t}=&-\sum_{\ell=1}^{3}\omega_\ell b_\ell\frac{\hat{\bm{F}}^1_{i+\frac{1}{2},j+b_\ell}-\hat{\bm{F}}^1_{i-\frac{1}{2},j+b_\ell}}{\Delta x}-\sum_{\ell=1}^{3}\omega_\ell\frac{\hat{\bm{F}}^2_{i+b_\ell,j+\frac{1}{2}}+\hat{\bm{F}}^2_{i+b_\ell,j-\frac{1}{2}}}{2\Delta y}\nonumber
		\\
		\label{eq:2D discrete equation3}
		&+\sum_{k=1}^{3}\sum_{\ell=1}^{3}\frac{\omega_k\omega_\ell \bm{F}_2\left(\bm{U}_{i\oplus b_\ell,j\oplus b_k}\right)}{\Delta y}:=\mathcal{L}_W(\bm{U}_h(t),i,j),
\end{align}

\noindent where 
$$\bm{\overline{U}}_{i,j}\left(t\right) \approx \frac{1}{\Delta x\Delta y}\int_{I_{i,j}}\bm{U}\left(x,y,t\right)\mathrm{d}x\mathrm{d}y$$
denotes the approximation to the zeroth order moment in $I_{i,j}$, 
\begin{eqnarray*}
	\begin{aligned}
		&\bm{\overline{V}}_{i,j}\left(t\right) \approx \frac{1}{\Delta x\Delta y}\int_{I_{i,j}}\bm{U}\left(x,y,t\right)\frac{x-x_i}{\Delta x}\mathrm{d}x\mathrm{d}y,\\
		&\bm{\overline{W}}_{i,j}\left(t\right) \approx \frac{1}{\Delta x\Delta y}\int_{I_{i,j}}\bm{U}\left(x,y,t\right)\frac{y-y_j}{\Delta x}\mathrm{d}x\mathrm{d}y,
	\end{aligned}
\end{eqnarray*}
denote the approximations to the first order moments in the $x$- and $y$-directions in $I_{i,j}$, respectively, $$\bm{U}_h(t)\!:=\!\{\bm{\overline{U}}_{i,j}\left(t\right)\}_{1\leq i\leq N_x,1\leq j\leq N_y}\cup \{\bm{\overline{V}}_{i,j}\left(t\right)\}_{1\leq i\leq N_x,1\leq j\leq N_y}\cup \{\bm{\overline{W}}_{i,j}\left(t\right)\}_{1\leq i\leq N_x,1\leq j\leq N_y}$$
denotes the set that contains all cells’ zeroth and first order moments, and $\bm{U}_{i\oplus b_\ell,j\oplus b_k}$ denotes the approximation to $\bm{U}\left(x_{i+b_\ell},y_{j+b_k},t\right)$ within the cell $I_{i,j}$. Here we follow \cite{zhao2020hermite} and use the three-point Gauss quadrature with the quadrature weights and nodes given by
\begin{align*}
	&\omega_1=\omega_3=\frac{5}{18}, \qquad \omega_2=\frac{4}{9},
	\\
	&\{ b_{\ell}\}_{\ell=1}^3=\left\{ -\frac{\sqrt{15}}{10},~ 0,~ \frac{\sqrt{15}}{10} \right\}.
\end{align*} 
%Denote $\bm{U}_{i+a,j+b}$ as approximation of $\bm{U}(x_{i+a},y_{j+b},t)$ inside the cell $I_{i,j}$.  Note that the symbol ``+'' used in the subscript of $\bm{U}_{i+a,j+b}$ is not an arithmetic operator, for example, $\bm{U}_{i+1/2,j+0}$ represent the approximate value at $(x_i+\frac{1}{2}\Delta x,y_j)$ inside $I_{i,j}$, while $\bm{U}_{(i+1)-1/2,j+0}$ represent the approximate value at $(x_i+\frac{1}{2}\Delta x,y_j)$ inside the cell $I_{i+1,j}$.
In \eqref{eq:2D discrete equation1}--\eqref{eq:2D discrete equation3}, $\hat{\bm{F}}^1_{i+\frac{1}{2},j+b_\ell}$ and $\hat{\bm{F}}^2_{i+b_\ell,j+\frac{1}{2}}$ denote the numerical flux at the cell interface points $(x_{i+\frac{1}{2}},y_{j+b_\ell})$ and $(x_{i+b_\ell},y_{j+\frac{1}{2}})$, respectively. In this paper, we employ the Lax--Friedrichs numerical fluxes
\begin{eqnarray}\label{2D LF}
	\left\{
	\begin{aligned}
	&	\hat{\bm{F}}^1_{i+\frac{1}{2},j+b_\ell}=\frac{1}{2}\left(\bm{F}_1\left(\bm{U}_{i\oplus \frac{1}{2},j\oplus b_\ell}\right)\!+\!\bm{F}_1\left(\bm{U}_{(i+1)\oplus \left(-\frac12\right),j\oplus b_\ell}\right)\!-\!\alpha_1\left(\bm{U}_{(i+1)\oplus \left(-\frac12\right),j\oplus b_\ell}\!-\!\bm{U}_{i\oplus \frac{1}{2},j\oplus b_\ell}\right)\right),\\
	&	\hat{\bm{F}}^2_{i+b_\ell,j+\frac{1}{2}}=\frac{1}{2}\left(\bm{F}_2\left(\bm{U}_{i\oplus b_\ell,j\oplus \frac{1}{2}}\right)\!+\!\bm{F}_2\left(\bm{U}_{i\oplus b_\ell,(j+1)\oplus \left(-\frac12\right)}\right)\!-\!\alpha_2\left(\bm{U}_{i\oplus b_\ell,(j+1)\oplus \left(-\frac12\right)}\!-\!\bm{U}_{i\oplus b_\ell,j\oplus \frac{1}{2}}\right)\right),
	\end{aligned}
	\right.
\end{eqnarray}
which will be useful for achieving the PCP property of our HWENO shceme. Here
\begin{eqnarray}
	\left\{
	\begin{aligned}
		&\alpha_1=\max_{i,j,\ell}\left\{\max\left\{\varrho_1\left(\bm{U}_{i\oplus \frac{1}{2},j\oplus b_\ell}\right),\varrho_1\left(\bm{U}_{(i+1)\oplus \left(-\frac12\right),j\oplus b_\ell}\right)\right\}\right\},\\
		&\alpha_2=\max_{i,j,\ell}\left\{\max\left\{\varrho_2\left(\bm{U}_{i\oplus b_\ell,j\oplus \frac{1}{2}}\right),\varrho_2\left(\bm{U}_{i\oplus b_\ell,(j+1)\oplus \left(-\frac12\right)}\right)\right\}\right\},
	\end{aligned}
	\right.
\end{eqnarray}
with $\varrho_1\left(\bm{U}\right)$ and $\varrho_2\left(\bm{U}\right)$ denoting the spectral radius of the Jacobian matrices $\frac{\partial \bm F_1(\bm U)}{\partial \bm U}$ and $\frac{\partial \bm F_2(\bm U)}{\partial \bm U}$, respectively.

%The zeroth-order moment $\bm{\overline{U}}_{i,j}$ is the approximation of  $\int_{I_{i,j}}\bm{U}(x,y)\mathrm{d}x\mathrm{d}y$, the first-order moments $\bm{\overline{V}}_{i,j}$, $\bm{\overline{W}}_{i,j}$ are the approximation of  $\int_{I_{i,j}}\bm{U}(x,y)\frac{x-x_i}{\Delta x}\mathrm{d}x\mathrm{d}y$ and       $\int_{I_{i,j}}\bm{U}(x,y)\frac{y-y_j}{\Delta y}\mathrm{d}x\mathrm{d}y$ respectively. 

%{\color{blue}
%\begin{rem}
%	It should be pointed out that the symbol ``+" in the subscript of $\bm{U}_{i+a,j+b}$ is not a standard addition operation, but rather a symbol representing the position relative to the cell center $\left(x_i,y_j\right)$. For example, $\bm{U}_{i+1/2,j+0}$ represents the approximate value at the point $\left(x_i+\frac{1}{2}\Delta x,y_j\right)$ computed within the cell $I_{i,j}$, while $\bm{U}_{(i+1)-1/2,j+0}$ stands for the approximate value at the same point $\left(x_{i+1}-\frac{1}{2}\Delta x,y_j\right)$ but computed within the cell $I_{i+1,j}$.
%\end{rem}
%}

To compute $\mathcal{L}_U(\bm{U}_h(t),i)$, $\mathcal{L}_V(\bm{U}_h(t),i)$ and $\mathcal{L}_W(\bm{U}_h(t),i)$ in \eqref{eq:2D discrete equation1}--\eqref{eq:2D discrete equation3}, one needs to reconstruct point values $\{\bm{U}_{i\oplus b_\ell,j\oplus b_k}\}_{\ell,k=1}^3$, $\{\bm{U}_{i\oplus b_\ell,j\oplus \left(\pm\frac12\right)}\}_{\ell=1}^3$ and $\{\bm{U}_{i\oplus \left(\pm\frac12\right),j\oplus b_\ell}\}_{\ell=1}^3$ by using $\{\bm{\overline{U}}_{i,j}\}$ and $\{\bm{\overline{V}}_{i,j}\}$.

%To facilitate the subsequent description, we first introduce two reconstruction operators $\bm{M}_L$ and $\bm{M}_H$, before giving the detailed spatial reconstruction procedures of our 2D PCP finite volume HWENO scheme.

%\begin{figure}[!htb]
%	\centering
%	\includegraphics[width=0.5\textwidth]{rhdresults//2Dstencil.pdf}
%	\caption{\label{FIG:2Dneighbour}}
%\end{figure}

\begin{figure}[!htbp]
	\centering
	\begin{tikzpicture}
		\begin{tikzpicture}[x=0.5cm, y=0.5cm, thick, xshift=1.5cm,yshift=0.8cm]
			% 绘制外边框
			\draw (0, 0) rectangle (9, 9);
			% 绘制九宫格竖线
			\foreach \x in {3,6}
			\draw (\x, 0) -- (\x, 9);
			% 绘制九宫格横线
			\foreach \y in {3,6}
			\draw (0, \y) -- (9, \y);
			% 标出cell名
			\foreach \y in {1,2,3}
			\node (\y) at(3*\y-1.5,1.5) {$I_{i,j}^\y$};
			\foreach \y in {4,5,6}
			\node (\y) at(3*\y-1.5-3*3,4.5) {$I_{i,j}^\y$};
			\foreach \y in {7,8,9}
			\node (\y) at(3*\y-1.5-3*6,7.5) {$I_{i,j}^\y$};
			% 标出x坐标
			\node (1) at(0,-0.8) {$x_{i-3/2}$};
			\node (1) at(3,-0.8) {$x_{i-1/2}$};
			\node (1) at(6,-0.8) {$x_{i+1/2}$};
			\node (1) at(9,-0.8) {$x_{i+3/2}$};
			% 标出y坐标
			\node (1) at(-1.1,0) {$y_{j-3/2}$};
			\node (1) at(-1.1,3) {$y_{j-1/2}$};
			\node (1) at(-1.1,6) {$y_{j+1/2}$};
			\node (1) at(-1.1,9) {$y_{j+3/2}$};
		\end{tikzpicture}
	\end{tikzpicture}
	\caption{\label{FIG:2Dneighbour}}
\end{figure}

For simplicity, denote $\overline{\bm{U}}_{i,j,r}$ as the zeroth-order moment in $I_{i,j}^r$ (see Figure \ref{FIG:2Dneighbour}), and $\overline{\bm{V}}_{i,j,r}$, $\overline{\bm{W}}_{i,j,r}$ as the first-order moments in $x$-direction and $y$-direction in $I_{i,j}^r$. Define 
\begin{align*}
	\bm{U}_{i,j}^S&=\left[\overline{\bm{U}}_{i,j,1}~\dots~\overline{\bm{U}}_{i,j,9} \right],\\
	\bm{V}_{i,j}^S&=\left[\overline{\bm{V}}_{i,j,2}~\overline{\bm{V}}_{i,j,4}~\overline{\bm{V}}_{i,j,5}~\overline{\bm{V}}_{i,j,6}~\overline{\bm{V}}_{i,j,8} \right],\\
	\bm{W}_{i,j}^S&=\left[\overline{\bm{W}}_{i,j,2}~\overline{\bm{W}}_{i,j,4}~\overline{\bm{W}}_{i,j,5}~\overline{\bm{W}}_{i,j,6}~\overline{\bm{W}}_{i,j,8} \right].
\end{align*}
%\subsubsection{Detailed 2D PCP HWENO reconstruction procedure}

The 2D HWENO reconstruction procedure is also based on two operators $\bm{M}_L$ and $\bm{M}_H$, which are introduced in Appendices \ref{app1} and \ref{app2} for better readability. 
The detailed PCP HWENO reconstruction procedure of our 2D HWENO scheme is summarized as follows:  
	\begin{description}
	\item[Step 1.]
	Use the KXRCF indicator \cite{krivodonova2004shock} dimension-by-dimension to identify the troubled cells. A cell is regarded as a troubled cell as long as it is identified as troubled cell in either $x$-direction or $y$-direction. Then modify the first-order moments in the troubled cells by using the HWENO limiter given in \cite{zhao2020hermite}.
	
	\item[Step 2.]
	Reconstruct the point values $\{\bm{U}_{i\oplus b_\ell,j\oplus b_k}\}_{\ell,k=1}^3$, $\{\bm{U}_{i\oplus b_\ell,j\oplus \left(\pm\frac12\right)}\}_{\ell=1}^3$ and $\{\bm{U}_{i\oplus \left(\pm\frac12\right),j\oplus b_\ell}\}_{\ell=1}^3$.
	
	\begin{itemize}
	\item Employ the linear reconstruction for the inner Gaussian points $\{\bm{U}_{i\oplus b_\ell,j\oplus b_k}\}_{\ell,k=1}^3$:
	$$\bm{U}_{i\oplus b_\ell,j\oplus b_k}^{*}=\bm{M}_L\left(\bm{U}_{i,j}^S,\bm{V}_{i,j}^S,\bm{W}_{i,j}^S,b_{\ell},b_{k}\right)\quad \forall \ell,k\in\left \{1,2,3\right \}.$$
		
	\item Reconstruction for the cell interface points $\{\bm{U}_{i\oplus b_\ell,j\oplus \left(\pm\frac12\right)}\}_{\ell=1}^3$ and $\{\bm{U}_{i\oplus \left(\pm\frac12\right),j\oplus b_\ell}\}_{\ell=1}^3$. 
	
	(i) If $I_{i,j}$ is not a troubled cell, perform linear reconstruction for the cell interface points:
	\begin{align*}
		&\bm{U}_{i\oplus\left(\pm\frac12\right),j\oplus b_\ell}^{*}=\bm{M}_L\left(\bm{U}_{i,j}^S,\bm{V}_{i,j}^S,\bm{W}_{i,j}^S,\pm\frac{1}{2},b_{\ell}\right)\quad \forall \ell\in\left \{1,2,3\right \},\\
		&\bm{U}_{i\oplus b_\ell,j\oplus \left(\pm\frac12\right)}^{*}=\bm{M}_L\left(\bm{U}_{i,j}^S,\bm{V}_{i,j}^S,\bm{W}_{i,j}^S,b_{\ell},\pm\frac{1}{2}\right) \quad \forall \ell\in\left \{1,2,3\right \}.
	\end{align*}
	
	(ii) If $I_{i,j}$ is a troubled cell, perform HWENO reconstruction based on characteristic decomposition for the cell interface points:
	\begin{align*}
		&\bm{U}_{i\oplus \left(\pm\frac12\right),j\oplus b_\ell}^{*}=\bm{R}_{i\pm\frac{1}{2},j}^1\bm{M}_H\left(\left(\bm{R}_{i\pm\frac{1}{2},j}^1\right)^{-1}\bm{U}_{i,j}^S,\left(\bm{R}_{i\pm\frac{1}{2},j}^1\right)^{-1}\bm{V}_{i,j}^S,\left(\bm{R}_{i\pm\frac{1}{2},j}^1\right)^{-1}\bm{W}_{i,j}^S,\pm\frac{1}{2},b_{\ell}\right),\\
		&\bm{U}_{i\oplus b_\ell,j\oplus \left(\pm\frac12\right)}^{*}=\bm{R}_{i,j\pm\frac{1}{2}}^2\bm{M}_H\left(\left(\bm{R}_{i,j\pm\frac{1}{2}}^2\right)^{-1}\bm{U}_{i,j}^S,\left(\bm{R}_{i,j\pm\frac{1}{2}}^2\right)^{-1}\bm{V}_{i,j}^S,\left(\bm{R}_{i,j\pm\frac{1}{2}}^2\right)^{-1}\bm{W}_{i,j}^S,b_{\ell},\pm\frac{1}{2}\right),
	\end{align*}
	for all $\ell\in\left \{1,2,3\right \}$. $\left(\bm{R}_{i+\frac{1}{2},j}^1\right)^{-1}$ and $\bm{R}_{i+\frac{1}{2},j}^1$ are taken as left and right eigenvector matrices of $\bm{A}_1\left(\overline{\bm{U}}_{i,j},\overline{\bm{U}}_{i+1,j}\right)$, which is the Roe matrix \cite{eulderink1994general} associated with $\frac{\partial \bm{F}_1}{\partial {\bm U}}$,
	$\left(\bm{R}_{i,j+\frac{1}{2}}^2\right)^{-1}$ and $\bm{R}_{i,j+\frac{1}{2}}^2$ are taken as left and right eigenvector matrices of $\bm{A}_2\left(\overline{\bm{U}}_{i,j},\overline{\bm{U}}_{i,j+1}\right)$, which is the Roe matrix  associated with $\frac{\partial \bm{F}_2}{\partial {\bm U}}$. The eigenvector matrices and fluxes are computed by using the primitive variables, which are recovered by using the proposed NR methods. 
	\end{itemize}
	
	\item[Step 3.] Perform the PCP limiter on the reconstructed point values $\{\bm{U}_{i\oplus b_\ell,j\oplus b_k}^{*}\}_{\ell,k=1}^3$, $\{\bm{U}_{i\oplus \!\left(\!\pm\!\frac12\!\right)\!,j\oplus b_\ell}^{*}\}_{\ell=1}^3$, and $\{\bm{U}_{i\oplus b_\ell,j\oplus \!\left(\!\pm\!\frac12\!\right)\!}^{*}\}_{\ell=1}^3$ as follows. 
	Define $\Theta := \{ (b_\ell, b_k) \} _{\ell,k=1}^3 \cup \{ (\pm\frac{1}{2}, b_\ell )\}_{\ell=1}^3 \cup \{ ( b_\ell, \pm\frac{1}{2} )\}_{\ell=1}^3$ and 
	$$
	\bm{U}_{i\oplus a,j\oplus b}^*=[D^*_{i\oplus a,j\oplus b}~(m_1)^*_{i\oplus a,j\oplus b}~(m_2)^*_{i\oplus a,j\oplus b}~E^*_{i\oplus a,j\oplus b}]^\top \quad \forall (a,b)\in \Theta. 
	$$
	Denote the first component of $\overline {\bm U}_{i,j}$ as $\overline D_{i,j}$. Let $$\mu_1=\frac{\lambda_1\alpha_1} {\lambda_1\alpha_1+\lambda_2\alpha_2}, \qquad  \mu_2=\frac{\lambda_2\alpha_2}{\lambda_1\alpha_1+\lambda_2\alpha_2)}$$
	with $\lambda_1=\Delta t/\Delta x$, $\lambda_2=\Delta t/\Delta y$.
	\begin{itemize}
		\item  Modify the mass density to enforce its positivity via 
		\begin{align*}
			&\tilde{D}_{i\oplus a,j\oplus b}=\theta_D(D^*_{i\oplus a,j\oplus b}-\overline{D}_{i,j})+\overline{D}_{i,j}  \quad \mbox{with} \quad \theta_D=\min\left \{\left |\frac{\overline{D}_{i,j}-\epsilon_D}{\overline{D}_{i,j}-D_{\min}} \right |,1\right \} ,
			\\
			&D_{\min}=\min\left\{\min\limits_{\ell}\left\{D_{i\oplus \left(\pm\frac12\right),j\oplus b_\ell}^{*}\right\},\min\limits_{\ell}\left\{D_{i\oplus b_\ell,j\oplus \left(\pm\frac12\right)}^{*}\right\},\min\limits_{\ell,k}\left\{D_{i\oplus b_\ell,j\oplus b_k}^{*}\right\},  \widehat {D}_{i,j} \right\}, 
			\\
			& \widehat {D}_{i,j}:=  \frac{\overline{D}_{i,j}\!-\!\sum\limits_{\ell=1}^3\omega_{\ell}\hat{\omega}_1\left [\mu_1\left(D_{i\oplus \frac12,j\oplus b_\ell}^{*}\!+D_{i\oplus \left(-\frac12\right),j\oplus b_\ell}^{*}\right)\!+\!\mu_2\left(D_{i\oplus b_\ell,j\oplus \frac12}^{*}\!+D_{i\oplus b_\ell,j\oplus \left(-\frac12\right)}^{*}\right) \right ]}{1-2\hat{\omega}_1}, 
		\end{align*}
		where $\epsilon_D=\min\limits_{i,j}\left \{10^{-13},\overline{D}_{i,j}\right \}$ is a small positive number introduced to avoid the effect of round-off errors. 
		\item Define $\tilde{\bm{U}}_{i\oplus a,j\oplus b}=[\tilde{D}_{i\oplus a,j\oplus b}~(m_1)^*_{i\oplus a,j\oplus b}~(m_2)^*_{i\oplus a,j\oplus b}~E^*_{i\oplus a,j\oplus b}]^\top$ for all $(a,b)\in \Theta$.
		Enforce the positivity of $g({\bm U})$ by 
		\begin{align} 
			&\bm{U}_{i\oplus a,j\oplus b}=\theta_g(\tilde{\bm{U}}_{i+a,j+b}-\overline{\bm{U}}_{i,j})+\overline{\bm{U}}_{i,j}~~\mbox{with}~~\theta_g=min\left \{\left |\frac{g(\overline{\bm{U}}_{i,j})-\epsilon_g}{g(\overline{\bm{U}}_{i,j})-g_{\min}} \right |,1\right \},\label{eq:PCPlimiter2D}
			\\ 
			&g_{\min} \!=\!\min\left\{\min\limits_{\ell}\left\{g(\tilde{\bm U}_{i\oplus \left(\pm\frac12\right),j\oplus b_\ell})\right\}\!,\min\limits_{\ell}\left\{g(\tilde{\bm U}_{i\oplus b_\ell,j\oplus \left(\pm\frac12\right)})\right\}\!,\min\limits_{\ell,k}\left\{g(\tilde{\bm U}_{i\oplus b_\ell,j\oplus b_k})\right\}\!, g({\widetilde {\bf \Pi}}_{i,j} )  \right\},\nonumber\\
			& {\widetilde {\bf \Pi}}_{i,j} :=  \!\frac{\bm{\overline{U}}_{i,j}\!-\!\sum\limits_{\ell=1}^{3}\!\omega_{\ell}\hat{\omega}_1\!\left [\mu_1\!\left(\tilde{\bm U}_{i\oplus \frac12,j\oplus b_\ell}+\tilde{\bm U}_{i\oplus \left(-\frac12\right),j\oplus b_\ell}\right)\!+\!\mu_2\!\left(\tilde{\bm U}_{i\oplus b_\ell,j\oplus \frac12}+\tilde{\bm U}_{i\oplus b_\ell,j\oplus \left(-\frac12\right)}\right) \right ]}{1-2\hat{\omega}_1} \nonumber,
		\end{align}
		where $\epsilon_g=\min\limits_{i,j}\left \{10^{-13},g(\overline{\bm{U}}_{i,j})\right \}$.
	\end{itemize}

\end{description}

\subsection{Rigorous analysis of PCP property}\label{2D PCP limiter}
In this subsection, we present a rigorous analysis of PCP property of the proposed 2D HWENO scheme. 

Similar to \cite[Proposition 3.1]{chen2022physical}, we can prove that the PCP limited point values \eqref{eq:PCPlimiter2D} satisfy the following property. 

\begin{lemma}\label{2D pcp inequality} If $\overline{\bm{U}}_{i,j}\in{\mathcal{G}}$, then:
	
	\begin{enumerate}[(1)]
		\item
		$\bm{U}_{i\oplus \left(\pm\frac12\right),j\oplus b_\ell}\in{\mathcal{G}}$, $\bm{U}_{i\oplus b_\ell,j\oplus \left(\pm\frac12\right)}\in{\mathcal{G}}$ and  $\bm{U}_{i\oplus b_\ell,j\oplus b_k}\in{\mathcal{G}}$ for all  $\ell,k \in \{1,2,3\}$.
		\item 
		${\bm \Pi}_{i,j}\in{\mathcal{G}}$, where
		\begin{equation}
			{\bm \Pi}_{i,j}\!:=\!\frac{\bm{\overline{U}}_{i,j}\!-\!\sum\limits_{\beta=1}^{3}\!\omega_{\beta}\hat{\omega}_1\!\left [\mu_1\!\left(\bm{U}_{i\oplus \frac12,j\oplus b_\beta}\!+\bm{U}_{i\oplus \left(-\frac12\right),j\oplus b_\beta}\right)\!+\!\mu_2\!\left(\bm{U}_{i\oplus b_\beta,j\oplus \frac12}\!+\bm{U}_{i\oplus b_\beta,j\oplus \left(-\frac12\right)}\right) \right ]}{1-2\hat{\omega}_1}.\label{2Dmidpoint}
		\end{equation}
	\end{enumerate}
\end{lemma}

Base on the Lemmas \ref{lemma:LF pcp} and \ref{2D pcp inequality}, we are now ready to show the PCP property of the proposed 2D HWENO scheme.

\begin{Theorem}\label{thm:2D pcp}
	Consider the proposed 2D HWENO scheme \eqref{eq:2D discrete equation1}--\eqref{eq:2D discrete equation3} with the Lax--Friedrichs flux \eqref{2D LF}. If $\overline {\bm U}_{i,j} \in {\mathcal G}$ for all $i$, $j$, then 
	the updated cell averages 
	\begin{equation}\label{eq:celling}
		\bm{\overline{U}}_{i,j}+\Delta t \mathcal{L}_U(\bm{U}_h(t),i,j) \in {\mathcal G} \qquad \forall i,j
	\end{equation}
	under the CFL condition
	\begin{equation}\label{cond:2DPCPCFL}
	\Delta t\leq \frac{\hat{\omega}_1 }{  \alpha_1/\Delta x+ \alpha_2/\Delta y}.
\end{equation}
\end{Theorem}

\begin{proof}
	Based on \eqref{2Dmidpoint}, we have the following convex decomposition for the cell average:
	\begin{equation*}%\label{cond:2DPCPpoly}
		\bm{\overline{U}}_{i,j}
		\!=\!(1\!-\!2\hat{\omega}_1){\bm \Pi}_{i,j}+\sum_{\beta=1}^{3}\omega_{\beta}\hat{\omega}_1\left[\mu_1\!\left(\bm{U}_{i\oplus \frac12,j\oplus b_\beta}+\bm{U}_{i\oplus \left(-\frac12\right),j\oplus b_\beta}\right)\!+\!\mu_2\!\left(\bm{U}_{i\oplus b_\beta,j\oplus \frac12}+\bm{U}_{i\oplus b_\beta,j\oplus \left(-\frac12\right)}\right) \right]
	\end{equation*}
	with ${\bm \Pi}_{i,j}\in{\mathcal{G}}$, $\bm{U}_{i\oplus \!\left(\!\pm\!\frac12\!\right)\!,j\oplus b_\ell}\in{\mathcal{G}}$, $\bm{U}_{i\oplus b_\ell,j\oplus \!\left(\!\pm\!\frac12\!\right)\!}\in{\mathcal{G}}$, and   $\bm{U}_{i\oplus b_\ell,j\oplus b_k}\in{\mathcal{G}}$ for all $\ell,k \in \{1,2,3\}.$ It follows that
	\begin{equation*}
		\begin{aligned}
			&\bm{\overline{U}}_{i,j}+\Delta t\mathcal{L}_U(\bm{U}_h(t),i,j)\\
			=&(1-2\hat{\omega}_1){\bm \Pi}_{i,j}+
			\sum_{\ell=1}^{3}\omega_{\ell}\hat{\omega}_1\left[\mu_1\!\left(\bm{U}_{i\oplus \frac12,j\oplus b_\ell}+\bm{U}_{i\oplus \!\left(\!-\!\frac12\!\right)\!,j\oplus b_\ell}\right)\!+\!\mu_2\!\left(\bm{U}_{i\oplus b_\ell,j\oplus \frac12}+\bm{U}_{i\oplus b_\ell,j\oplus \!\left(\!-\!\frac12\!\right)\!}\right) \right]
			\\
			&-\Delta t
			\sum_{\ell=1}^{3}\omega_\ell\frac{\hat{\bm{F}}^1_{i+\frac{1}{2},j+b_\ell}-\hat{\bm{F}}^1_{i-\frac{1}{2},j+b_\ell}}{\Delta x}-
			\Delta t\sum_{\ell=1}^{3}\omega_\ell\frac{\hat{\bm{F}}^2_{i+b_\ell,j+\frac{1}{2}}-\hat{\bm{F}}^2_{i+b_\ell,j-\frac{1}{2}}}{\Delta y}\\
			=&(1-2\hat{\omega}_1){\bm \Pi}_{i,j}+\mu_1\sum_{\ell=1}^{3}\omega_{\ell}\hat{\omega}_1\left [\left ( 1-\frac{\lambda_1\alpha_1}{\mu_1\hat{\omega}_1}\right )(\bm{U}_{i\oplus \!\left(\!-\!\frac12\!\right)\!,j\oplus b_\ell}+\bm{U}_{i\oplus \frac{1}{2},j\oplus b_\ell})+\right.\\
			&\left.\frac{\lambda_1\alpha_1}{2\mu_1\hat{\omega}_1}\left(
			\bm{F}_1^+(\bm{U}_{(i\!-\!1)\oplus \frac{1}{2},j\oplus b_\ell},\!\alpha_1\!)\!+\!
			\bm{F}_1^+(\bm{U}_{i\oplus \!\left(\!-\!\frac12\!\right)\!,j\oplus b_\ell},\!\alpha_1\!)\!+\!
			\bm{F}_1^-(\bm{U}_{i\oplus \frac{1}{2},j\oplus b_\ell},\!\alpha_1\!)\!+\!
			\bm{F}_1^-(\bm{U}_{(i\!+\!1)\oplus \!\left(\!-\!\frac12\!\right)\!,j\oplus b_\ell},\!\alpha_1\!)
			\right)\right]\!\\
			& + \mu_2\sum_{\ell=1}^{3}\omega_{\ell}\hat{\omega}_1
			\left [\left ( 1-\frac{\lambda_2\alpha_2}{\mu_2\hat{\omega}_1}
			\right )
			(
			\bm{U}_{i\oplus b_\ell,j\oplus \!\left(\!-\!\frac12\!\right)\!}+\bm{U}_{i\oplus b_\ell,j\oplus \frac{1}{2}}
			)+\right.\\
			&\left.\frac{\lambda_2\alpha_2}{2\mu_2\hat{\omega}_1}\left (
			\bm{F}_2^+(\bm{U}_{i\oplus b_\ell,(j\!-\!1)\oplus \frac{1}{2}},\!\alpha_2\!)\!+\!
			\bm{F}_2^+(\bm{U}_{i\oplus b_\ell,j\oplus \!\left(\!-\!\frac12\!\right)\!},\!\alpha_2\!)\!+\!
			\bm{F}_2^-(\bm{U}_{i\oplus b_\ell,j\oplus \frac{1}{2}},\!\alpha_2\!)\!+\!
			\bm{F}_2^-(\bm{U}_{i\oplus b_\ell,(j\!+\!1)\oplus \!\left(\!-\!\frac12\!\right)\!},\!\alpha_2\!)\right)\right].\\
		\end{aligned}
	\end{equation*}
Under the CFL condition \eqref{cond:2DPCPCFL}, $\bm{\overline{U}}_{i,j}+\Delta t\mathcal{L}_U(\bm{U}_h(t),i,j)$ has been reformulated into a convex combination form. Thanks to Lemma \ref{lemma:LF pcp} and the convexity of ${\mathcal{G}}$, we obtain \eqref{eq:celling}. The proof is completed.
\end{proof}

Theorem \ref{thm:2D pcp} implies that the proposed 2D HWENO scheme is PCP if the forward Euler method is used for time discretization. Since the SSP Runge--Kutta method \eqref{eq:SSPRK3} is formally a convex combination of forward Euler, the PCP property remains valid for the fully discrete HWENO scheme, due to the convexity of $\mathcal G$.

\subsection{Application to axisymmetric RHD equations in cylindrical coordinates}\label{section:cylindrical coordinates}

In this subsection, we present the HWENO scheme for the axisymmetric RHD euqations in cylindrical coordinates $(r,z)$, which can be written as
\begin{equation*}\label{eq:rhd with source}
	\frac{\partial\bm{U}}{\partial t}+\frac{\partial\bm{F}_1(\bm{U})}{\partial r}+\frac{\partial\bm{F}_2(\bm{U})}{\partial z}=\bm{S}(\bm{U},r),
\end{equation*} 
 where the definition of fluxes $\bm{F}_1$ and $\bm{F}_2$ are the same as \eqref{eq:intro_fluxvecter}, and the source term
\begin{equation*}
	\bm{S}(\bm{U},r)=-\frac{1}{r}(Dv_1,m_1v_1,m_2v_1,m_1)^\top.
\end{equation*} 
The semi-discrete HWENO scheme for axisymmetric RHD equations reads 
\begin{eqnarray}\label{2Dscheme with source}
	\left\{
	\begin{aligned}
		\frac{\mathrm{d}\overline{\bm{U}}_{i,j}}{\mathrm{d}t}&=\mathcal{L}_U(\bm{U}_h(t),i,j)+\overline{\bm{S}}_{i,j}, \\
		\frac{\mathrm{d}\overline{\bm{V}}_{i,j}}{\mathrm{d}t}&=\mathcal{L}_V(\bm{U}_h(t),i,j)+\bm{\overline{S}}^1_{i,j}, \\
		\frac{\mathrm{d}\overline{\bm{W}}_{i,j}}{\mathrm{d}t}&=\mathcal{L}_W(\bm{U}_h(t),i,j)+\bm{\overline{S}}^2_{i,j}, 
	\end{aligned}
	\right.
\end{eqnarray}
where
\begin{align*}
	&\bm{\overline{S}}_{i,j}:=\sum_{\ell=1}^{3}\sum_{k=1}^{3}\omega_\ell\omega_k\bm{S}\left(\bm{U}_{i\oplus b_\ell,j\oplus b_k}\right),\\
	&\bm{\overline{S}}_{i,j}^1:=\sum_{\ell=1}^{3}\sum_{k=1}^{3}\omega_\ell\omega_kb_\ell\bm{S}\left(\bm{U}_{i\oplus b_\ell,j\oplus b_k}\right),\\
	&\bm{\overline{S}}_{i,j}^2:=\sum_{\ell=1}^{3}\sum_{k=1}^{3}\omega_\ell\omega_kb_k\bm{S}\left(\bm{U}_{i\oplus b_\ell,j\oplus b_k}\right)
\end{align*}
are the numerical approximations to $\int_{I_{i,j}}\bm{S}(\bm{U},r)\mathrm{d}r\mathrm{d}z$,  $\int_{I_{i,j}}\bm{S}(\bm{U},r)\frac{r-r_i}{\Delta r}\mathrm{d}r\mathrm{d}z$, and       $\int_{I_{i,j}}\bm{S}(\bm{U},r)\frac{z-z_j}{\Delta z}\mathrm{d}r\mathrm{d}z$, respectively. The definitions of spatial operators $\mathcal{L}_U$, $\mathcal{L}_V$ and $\mathcal{L}_W$ are the same as \eqref{eq:2D discrete equation1}--\eqref{eq:2D discrete equation3} (with the variables $x$, $y$ replaced by $r$, $z$).

To analysis the PCP property of the scheme \eqref{2Dscheme with source}, we recall the following lemma proposed in \cite[Section 3.2]{2015High}:

\begin{lemma}\label{source pcp}
	If $\bm{U}\in{\mathcal{G}}$, then $\bm{U}+\Delta t\bm S(\bm U,r)\in{\mathcal{G}}$ under the time step restriction
	$$ v_1  \Delta t\leq\frac{rg(\bm U)}{p+g(\bm U)}.$$
\end{lemma}

Assume $\beta$ is a positive number, then we have
\begin{equation*}
	\begin{aligned}
		&\bm{\overline{U}}_{i,j}+\Delta t(\mathcal{L}_U(\bm{U}_h(t),i,j)+\overline{\bm{S}}_{i,j})\\
		=&(1-\beta)(\bm{\overline{U}}_{i,j}+\frac{\Delta t}{1-\beta}\mathcal{L}_U(\bm{U}(t),i,j))+\beta(\bm{\overline{U}}_{i,j}+\frac{\Delta t}{\beta}\overline{\bm S}_{i,j})\\
		=&(1-\beta) \left(\bm{\overline{U}}_{i,j}+\frac{\Delta t}{1-\beta}\mathcal{L}_U(\bm{U}(t),i,j)\right)+\beta\left(
		\sum_{\ell=1}^{3}\sum_{k=1}^{3}
		\omega_\ell\omega_k\left(\bm{U}_{i\oplus b_\ell,j\oplus b_k}+\frac{\Delta t}{\beta}\bm{S}\left(\bm{U}_{i\oplus b_\ell,j\oplus b_k}\right)\right)
		\right).
	\end{aligned}
\end{equation*}
By using Theorem \ref{thm:2D pcp}, Lemma \ref{2D pcp inequality}, Lemma \ref{source pcp}, and the convexity of ${\mathcal{G}}$, one can deduce that the HWENO scheme \eqref{2Dscheme with source} preserves $$\bm{\overline{U}}_{i,j}+\Delta t(\mathcal{L}_U(\bm{U}_h(t),i,j)+\overline{\bm{S}}_{i,j})\in {\mathcal{G}}$$  
under the time step restriction
\begin{equation}\label{key212}
\frac{\Delta t}{1-\beta}\leq \frac{\hat{\omega}_1 }{  \alpha_1/\Delta x+ \alpha_2/\Delta y}, \qquad 
v_1(\bm{U}_{i\oplus b_\ell,j\oplus b_k})  \frac{\Delta t}{\beta} \le \frac{(r_i+b_\ell\Delta r)g(\bm{U}_{i\oplus b_\ell,j\oplus b_k})}{\left(p(\bm{U}_{i\oplus b_\ell,j\oplus b_k})+g(\bm{U}_{i\oplus b_\ell,j\oplus b_k})\right)}.
\end{equation}
Taking a special $\beta$ leads to the following time step restriction 
\begin{equation*}
	\Delta t\leq \beta A_s,
\end{equation*}
where 
\begin{align*}
	A_s&=\displaystyle \min_{i,j}\left\{\min_{ (\ell,k) \in \Lambda_{ij} }\left \{
	\frac{(r_i+b_\ell\Delta r)g(\bm{U}_{i\oplus b_\ell,j\oplus b_k})}{\left(p(\bm{U}_{i\oplus b_\ell,j\oplus b_k})+g(\bm{U}_{i\oplus b_\ell,j\oplus b_k})\right)v_1(\bm{U}_{i\oplus b_\ell,j\oplus b_k})}\right 
	\}\right\},\\
	\beta&=\frac{\hat{\omega}_1}{\hat{\omega}_1+A_s\left (\alpha_1/\Delta x+ \alpha_2/\Delta y \right )}, \qquad   \Lambda_{ij} :=\left\{ (\ell,k):  
	v_1(\bm{U}_{i\oplus b_\ell,j\oplus b_k}) > 0 \right\}.
\end{align*}
Again, the PCP property remains valid for the fully discrete HWENO scheme if the SSP Runge--Kutta method is used for time discretization.

%
%
%Following \cite[Section 3.7]{chen2022physical}, 
%we can prove that 
%the HWENO scheme \eqref{2Dscheme with source} preserves  $$\bm{\overline{U}}_{i,j}+(\mathcal{L}_U(\bm{U}_h(t),i,j)+\overline{\bm{S}}_{i,j})\Delta t\in {\mathcal{G}}$$  
%under the time step restriction
%\begin{equation*}
%	\Delta t\leq \beta A_s,
%\end{equation*}
%where 
%\begin{align*}
%	A_s:&=\displaystyle \min_{i,j}\left \{\frac{r_i g(\bm{U}_{i,j})}{\left(p(\bm{U}_{i,j})+g(\bm{U}_{i,j})\right)|v_1(\bm{U}_{i,j})|}\right \},\\
%	\beta:&=\frac{\hat{\omega}}{\hat{\omega}+2A_s\left (\tau_1+\tau_2 \right )}.
%\end{align*}

\section{Numerical tests}\label{sec:numerical_tests}
This section will conduct several ultra-relativistic numerical experiments, to demonstrate the accuracy, robustness, and effectiveness of the proposed PCP finite volume HWENO schemes and NR methods. The CFL number is set as 0.6, and unless otherwise specified, the adiabatic index $\gamma$ is set as $5/3$. 
 
\subsection{Numerical experiments on primitive variables recovering algorithms}

We conduct several tests to evaluate the accuracy and efficiency of the three proposed quadratically convergent PCP primitive variable recovering algorithms (the target accuracy $\epsilon_{target}$ is set as $10^{-14}$), by comparing them with two existing algorithms from the literature, including the hybrid linearly convergent PCP algorithm introduced in \cite{chen2022physical} (which we refer to as ``Hybrid-linear'') and the velocity-proxy-based  
recovery algorithm from \cite{riccardi2008primitive} (which we refer to as ``Vel-Proxy''). 
All the tests are implemented by the C++ language with double precision, and performed with one core on the same Windows environment with Intel(R) Core(TM) i5-7300HQ CPU @ 2.50GHz.

\begin{exa} [Random tests] Three sets of random tests are provided to validate the accuracy and efficiency of our NR methods presented in Section \ref{sec:Pressure recovering algorithms}. 
	Let $U_{\tt rand}$ denote the uniform random variables independent generated in $[0,1]$. 
	The first set of random tests involve the mild primitive variables:
	\begin{eqnarray}
		\begin{cases}
			\rho=1000 U_{\tt rand}+10^{-10},\\
			v=1.9999 U_{\tt rand} -1.9999/2,\\
			p=10 U_{\tt rand}+10^{-10},\\
			\gamma=1.0001+0.9999 U_{\tt rand}.
		\end{cases}
	\end{eqnarray}
	The second set of random tests involve more extreme primitive variables:
	\begin{eqnarray}
		\begin{cases}
			\rho= 10^{-3} U_{\tt rand} +10^{-10},\\
			v= 1.9999 U_{\tt rand}-1.9999/2,\\
			p= 0.1 U_{\tt rand} +10^{-10},\\
			\gamma=1.0001+0.9999 U_{\tt rand}. 
		\end{cases}
	\end{eqnarray}
	The last set of tests involve large density and small velocities:
	\begin{eqnarray}
		\begin{cases}
			\rho=10000 U_{\tt rand}  +10^{-10},\\
			v=0.001 U_{\tt rand},\\
			p=10 U_{\tt rand} +10^{-10},\\
			\gamma=1.0001+0.9999 U_{\tt rand}.
		\end{cases}
	\end{eqnarray}
	In our experimental setup, we generate primitive variables randomly and use them to calculate the corresponding conservative variables $\bm{U}$ through (\ref{eq:intro_conver2conserv}). We then apply the primitive variables recovering algorithms to recompute the primitive variables from $\bm{U}$. The results of our tests, which consist of $10^8$ independent random experiments, are presented in Tables \ref{table:recovering algorithms tests1}--\ref{table:recovering algorithms tests3}. These tables provide information on the total running time, maximum errors in $p$, total errors in $p$, average iteration numbers, and total numbers of negative $p$ appearing in iterations. 
	We observe that no negative pressure is produced, confirming the PCP property. 
	Our experimental results indicates that the proposed hybrid NR method exhibits superior performance in terms of speed, robustness, efficiency, and accuracy across all tests. As a result, it can be regarded as the most effective algorithm overall.

	\begin{table}[!htb]
		\centering
		\captionsetup{font=small}
		\caption{The first set of random tests: CPU time, maximum errors, average errors, average iterations and negative numbers in $10^8$ independent random experiments.}\label{table:recovering algorithms tests1}
		
		\begingroup
		\setlength{\tabcolsep}{8pt} % Default value: 6pt
		\renewcommand{\arraystretch}{1.3} % Default value: 1
		
		\centering
		
		\small
		\begin{tabular}{cccccc}
			\hline
			algorithms & total time (s) & max error  & average error & average iteration & negative number \\ \hline
			Hybrid-linear         & 120.152        & 1.49E-07    & 9.19E-14     & 13.4924            & 0               \\
			NR-I     & 36.946          & 1.81E-04 & 6.66E-12   & 4.25424            & 0               \\
			NR-II     & 60.776          & 1.75E-09    & 9.18E-14     & 4.51075            & 0               \\
			Hybrid NR  & 38.775         & 1.15E-09    & 6.57E-14     & 3.95072            & 0               \\
			Analytical  & 367.549         & 2.39E-00    & 1.57E-06     & -            & 0               \\ 
			Vel-Proxy & 62.906         & 1.25E-09    & 7.36E-14     & 5.68068            & 0               \\ \hline
		\end{tabular}
		
		\endgroup
		
	\end{table}

	\begin{table}[!htb]
		\centering
		\captionsetup{font=small}
		\caption{The second set of random tests: CPU time, maximum errors, average errors, average iterations and negative numbers in $10^8$ independent random experiments.}\label{table:recovering algorithms tests2}
		
		\begingroup
		\setlength{\tabcolsep}{8pt} % Default value: 6pt
		\renewcommand{\arraystretch}{1.3} % Default value: 1
		
		\centering
		
		\small
		\begin{tabular}{cccccc}
			\hline
			& total time (s) & max error & average error & average iteration & negative number \\ \hline
			Hybrid-linear         & 120.075        & 2.51E-07   & 2.44E-14  & 13.4426            & 0               \\
			NR-I     & 79.845         & 2.46E-06   & 3.86E-13     & 13.1856            & 0               \\
			NR-II     & 49.493         & 2.10E-11   & 1.07E-16     & 3.20572            & 0               \\
			Hybrid NR  & 49.734         & 2.10E-11   & 1.07E-16     & 3.25100            & 0               \\ 
			Analytical  & 369.785         & 3.43E-02   & 4.24E-08     & -            & 0               \\
			Vel-Proxy  & 62.420         & 1.70E-11    & 1.26E-16     & 5.69168            & 0               \\ \hline
		\end{tabular}
		
		\endgroup
		
	\end{table}
	
	\begin{table}[!htb]
		\centering
		\captionsetup{font=small}
		\caption{The third set of random tests: CPU time, maximum errors, average errors, average iterations and negative numbers in $10^8$ independent random experiments.}\label{table:recovering algorithms tests3}
		
		\begingroup
		\setlength{\tabcolsep}{8pt} % Default value: 6pt
		\renewcommand{\arraystretch}{1.3} % Default value: 1
		
		\centering
		
		\small
		\begin{tabular}{cccccc}
			\hline
			& total time (s) & max error & average error & average iteration & negative number \\ \hline
			Hybrid-linear         & 33.876        & 6.74E-12   & 8.96E-14  & 2.07067            & 0               \\
			NR-I     & 35.504         & 2.80E-04   & 6.53E-12     & 3.8696            & 0               \\
			NR-II     & 58.742         &  6.00E-12  &  8.24E-14    & 4.10408            & 0               \\
			Hybrid NR  & 37.629         & 5.08E-12   & 7.14E-14     & 3.74777            & 0               \\
			Analytical  & 367.152         & 6.51E-04   & 6.60E-10     & -            & 0               \\
			Vel-Proxy  & 56.883         & 5.91E-12    & 8.51E-14     & 4.99994            & 0               \\ \hline
		\end{tabular}
		
		\endgroup
		
	\end{table}

\end{exa}

\subsection{One-dimensional examples}

\begin{exa} [1D accuracy test]\label{test:1Dsmooth} This is an ultra-relativistic smooth problem that serves the purpose of testing the accuracy of the 1D PCP HWENO scheme. The initial condition is given as
\begin{equation*}
	\bm{Q}(x,0)=\left(1+0.99999\sin(x),0.9999,0.0001\right)^T, x\in[0,2\pi).
\end{equation*}
Due to the low density, large velocity close to the speed of light, and low pressure, this test is challenging, and  the PCP limiting produce is necessary for successful simulation. 
We perform this test on the mesh of $N$ uniform cells with $N \in \{30,60,90,\dots,180\}$. 
Table \ref{table:1D accuracy test} lists the numerical errors of the rest-mass density $\rho$ and the convergence rates in $L^1$, $L^2$ and $L^{\infty}$ norms at time $t=2\pi$. 
The results indicate that the 1D PCP HWENO scheme achieves sixth-order accuracy, which is not destroyed by the PCP limiter.
	\begin{table}[!htbp]
		\centering
		
		\captionsetup{font=small}
		\caption{ Example \ref{test:1Dsmooth}: Numerical errors in $L^1$, $L^2$ and $L^{\infty}$ norms and the corresponding convergence rates.}\label{table:1D accuracy test}
		
		\begingroup
		\setlength{\tabcolsep}{8pt} % Default value: 6pt
		\renewcommand{\arraystretch}{1.3} % Default value: 1
		
		\centering
		
		\small
		\begin{tabular}{ccccccc}
			\hline
			$N$ & $L^1$ error     & Order    & $L^2$ error     & Order    & $L^{\infty}$ error   & Order  \\ \hline
			30  & 3.56E-01 & -       & 7.14E-01 & -       & 2.35E-00 & -       \\
			60  & 8.40E-06 & 15.37 & 2.85E-05 & 14.61 & 1.82E-04 & 13.66 \\
			90  & 1.11E-07 & 10.68 & 1.23E-07 & 13.43 & 1.74E-07 & 17.15 \\
			120 & 1.97E-08 & 6.00 & 2.19E-08 & 6.00 & 3.10E-08 & 6.00 \\
			150 & 5.17E-09 & 5.99 & 5.75E-09 & 5.99  & 8.21E-09 & 5.95 \\
			180 & 1.74E-09 & 5.97 & 1.94E-09 & 5.97 & 2.82E-09 & 5.86  \\ \hline
		\end{tabular}

		\endgroup
		
\end{table}
\end{exa}

\begin{exa} [1D Riemann problem]\label{test:1Driemann} The second test problem is a Riemann problem with the following initial data: 
	\begin{equation}
		\bm{Q}(x,0) = \begin{cases}
			(10,0,40/3)^\top,~~&x\leq0.5,\\
			(1,0,10^{-6})^\top,~~&x>0.5.
		\end{cases}
	\end{equation}
This initial discontinuity results in a left-moving shock wave, a right-moving contact discontinuity, and a right-moving shock wave. Figure \ref{fig:1Driemann} presents the numerical results at $t=0.4$, obtained by the PCP HWENO scheme with 320 uniform cells in the computational domain $[0,1]$.  
We see that the waves are well captured without any oscillation, and the numerical solution agrees well with the exact solution. 
The PCP limited cells from $t=0$ to $t=0.4$ are also displayed in Figure \ref{fig:1Driemann}, showing that only two cells are limited during the simulation.

\begin{figure}[!htb]
	\centering
	\subfloat[$\rho$]{
		\includegraphics[width=0.45\textwidth]{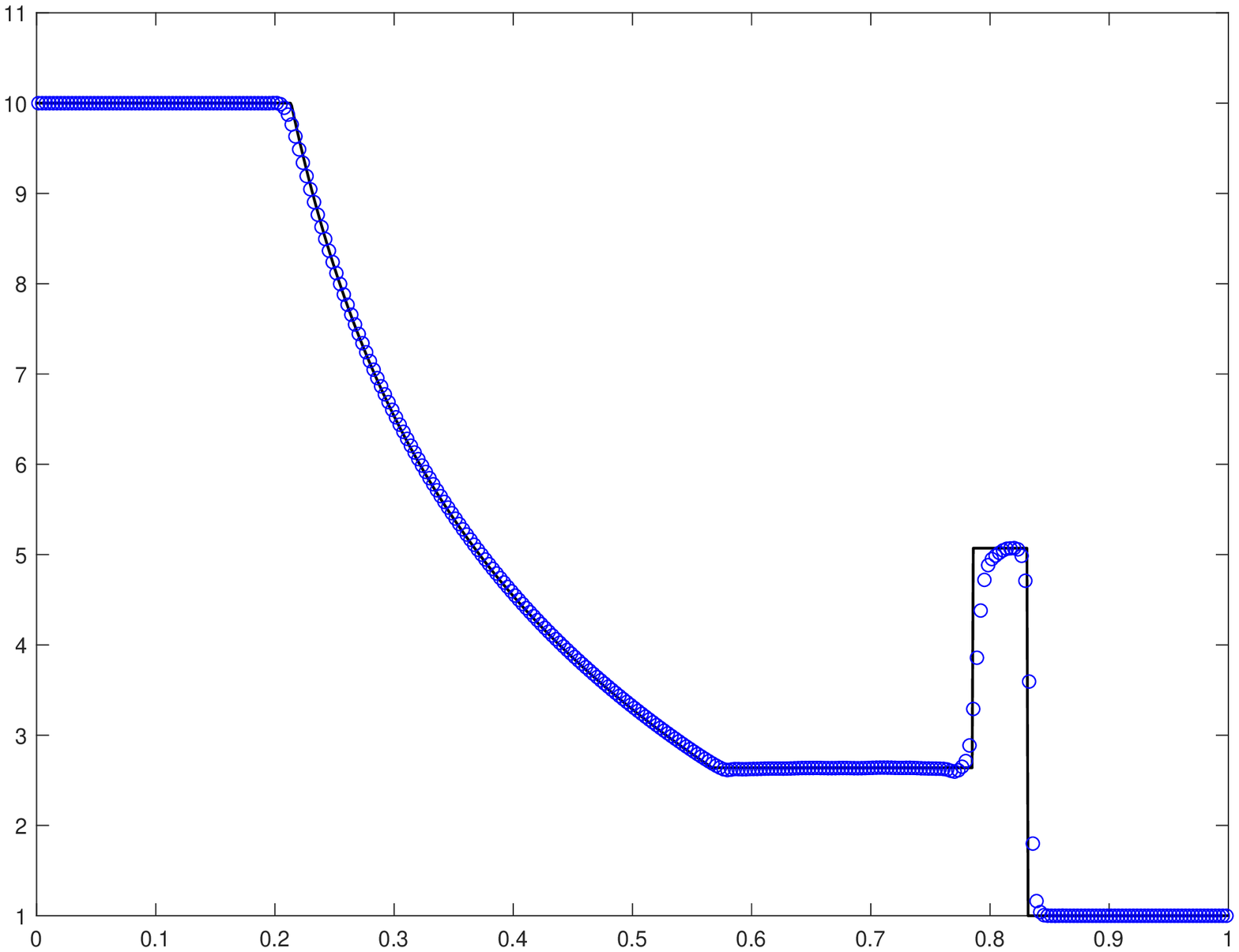}\label{fig:1Driemann_rho}
	}
	\subfloat[$v_1$]{
		\includegraphics[width=0.45\textwidth]{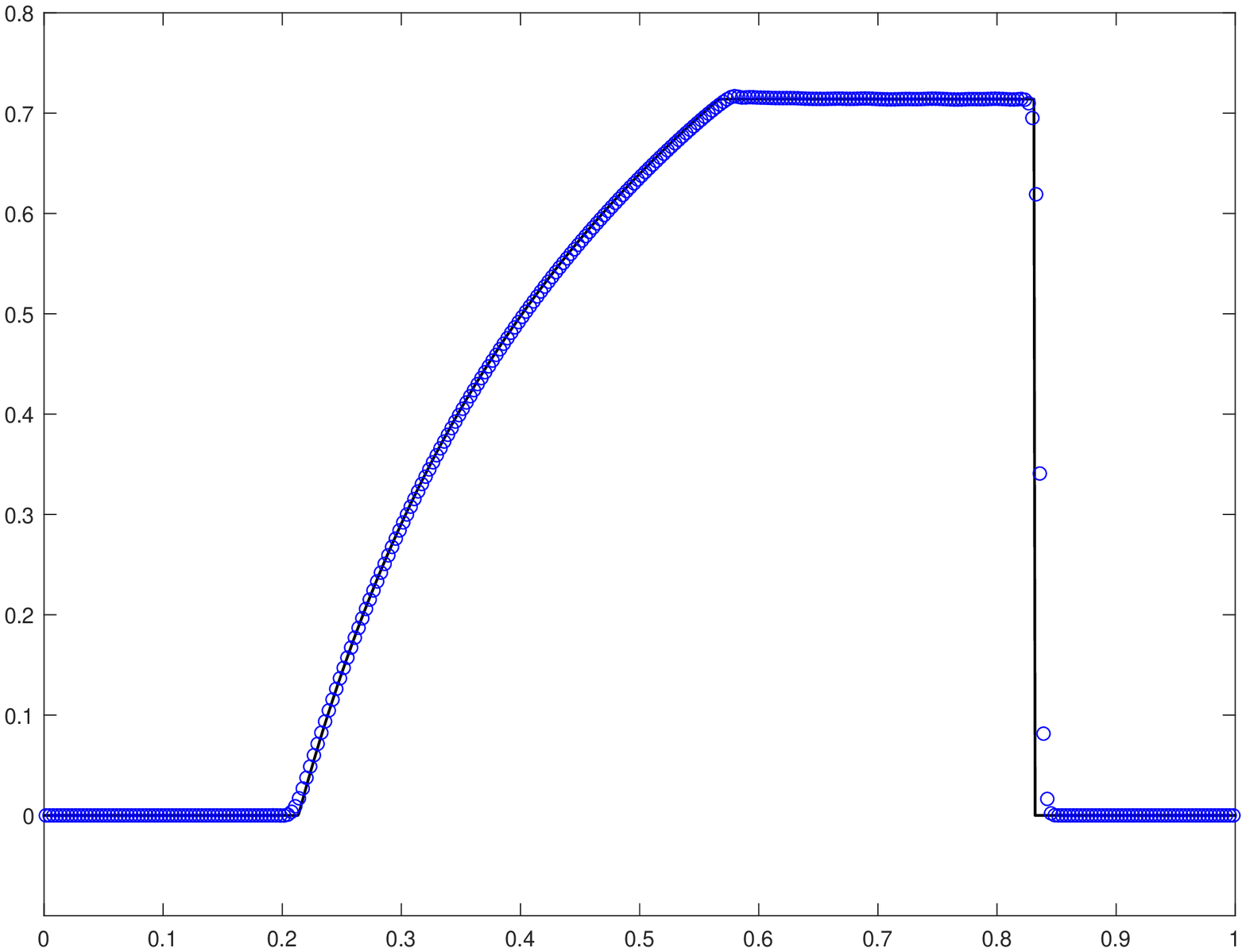}\label{fig:1Driemann_v1}
	}

	\subfloat[$p$]{
		\includegraphics[width=0.45\textwidth]{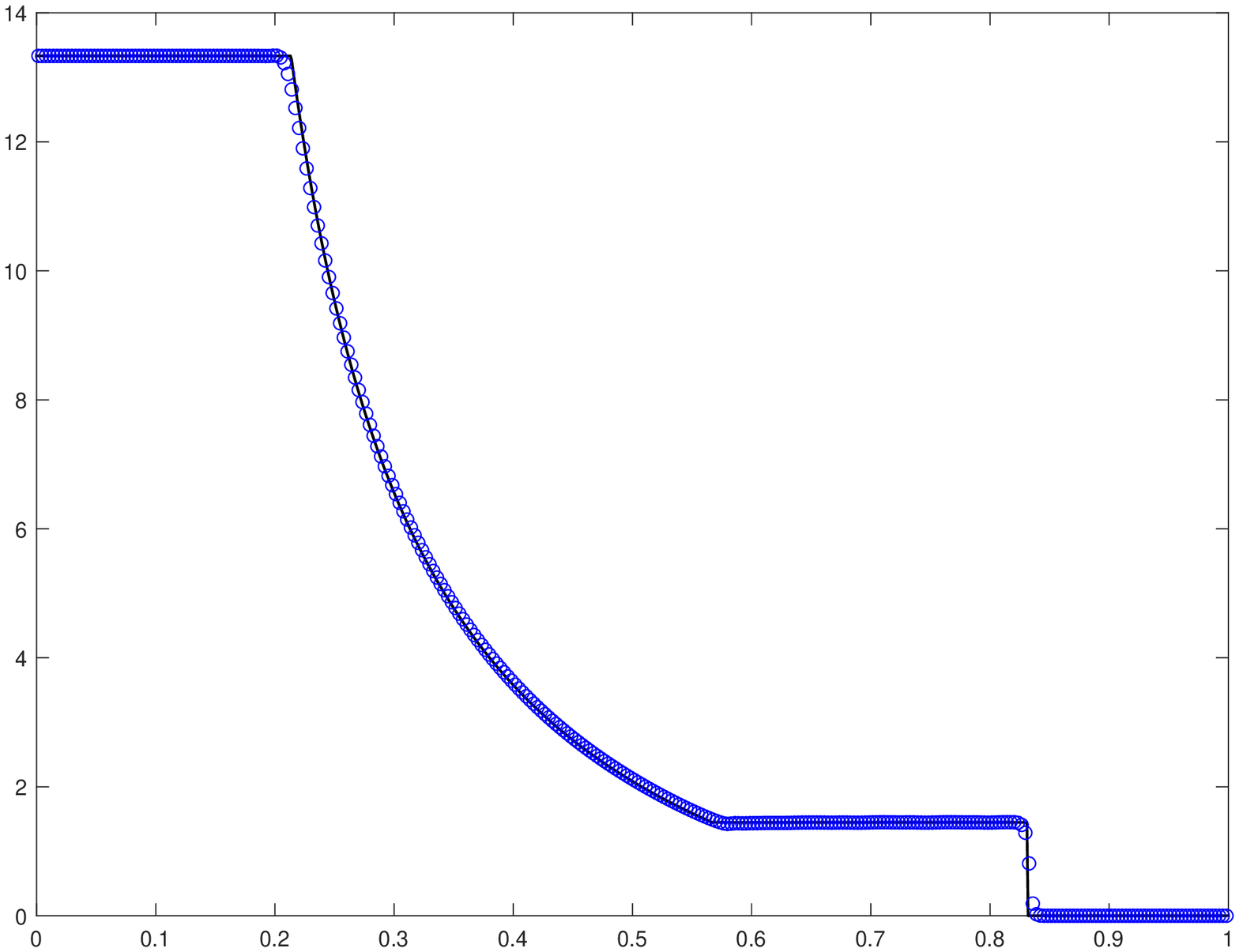}\label{fig:1Driemann_p}
	}
	\subfloat[PCP limited cells]{
		\includegraphics[width=0.45\textwidth]{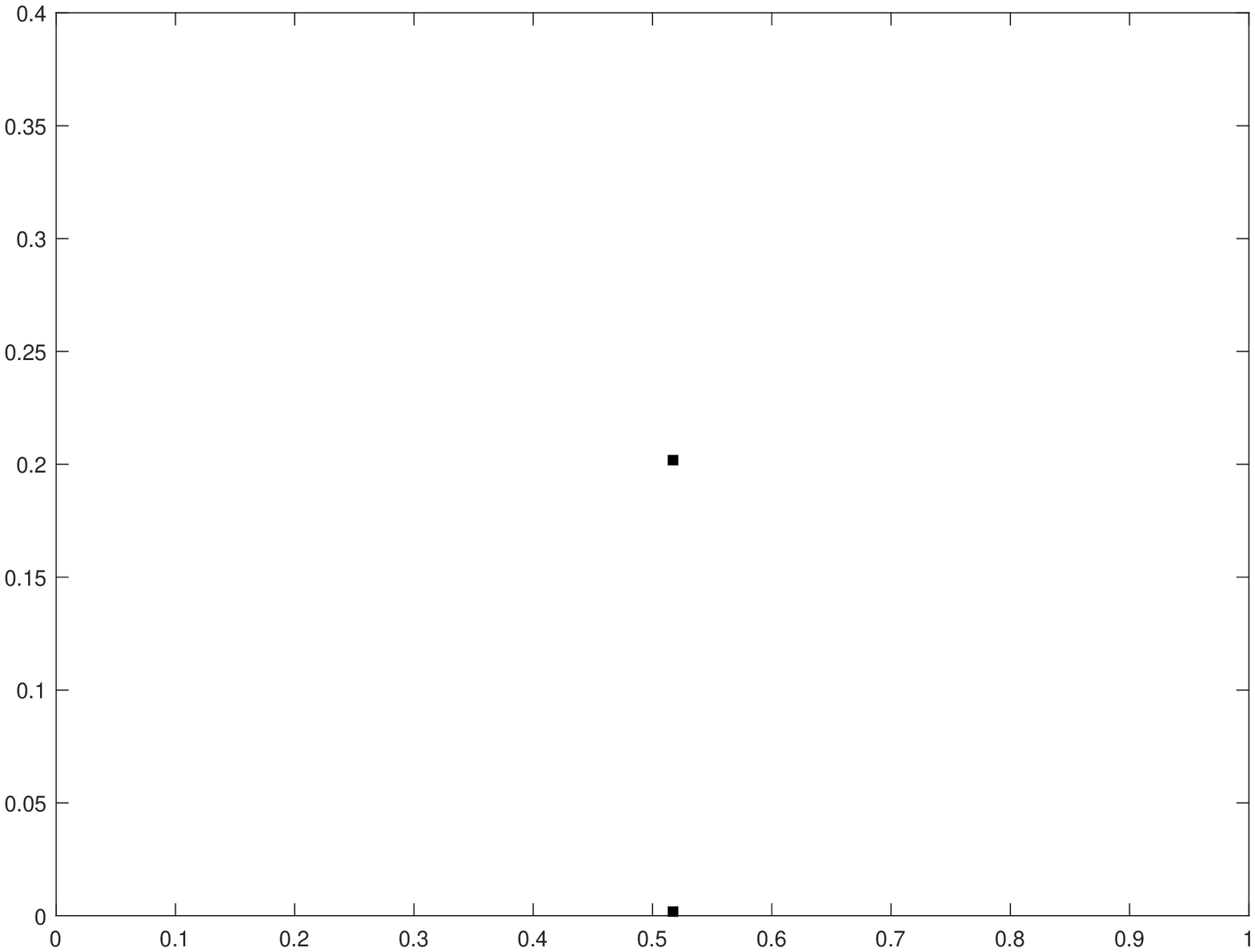}\label{fig:1Driemann_troublecell}
	}

	\captionsetup{font=small}
	\caption{Example \ref{test:1Driemann}. The numerical solution (symbols ``$\circ$") and exact solution (solid lines) of density $\rho$, velocity $v_1$, and pressure $p$.  The PCP limited cells are also displayed.}\label{fig:1Driemann}
\end{figure}

\end{exa}

\begin{exa} [Shock heating problem]\label{test:shockheating} This example simulates the shock heating problem, has become a standard test for evaluating the ability of numerical schemes to handle strong shocks without generating excessive postshock oscillations. The initial data in the computational domain $\left[0,1\right]$ is given as 
\begin{equation}
	\bm{Q}(x,0) =\left(1,1-10^{-10},\frac{10^{-4}}{3} \right)^\top,
\end{equation}
and the adiabatic index is take as $\gamma=4/3$. 
The proposed model involves a scenario in which a gas with rightward velocity close to the speed of light collides with a wall. Upon impact, the kinetic energy of the gas is converted into internal energy, resulting in compression and heating. As a result of this process, a strong shock wave is generated, which propagates towards the left at a velocity of $v_s=(\gamma-1)W_0|v_0|/(W_0+1)$. Here, $v_0=1-10^{-10}$ represents the initial velocity of the gas, while $W_0$ denotes the corresponding Lorentz factor. The gas behind the shock wave comes to a rest and possesses a specific internal energy of $W_0-1$, as deduced through energy conservation across the wave. The compression ratio across the shock is given by $\sigma=(\gamma+1)/(\gamma-1)+(\gamma/(\gamma-1))(W_0-1)$.

To evaluate the necessity of the PCP limiter, we perform the simulation without using this limiter and observe that the code breaks down after only one time step. We then apply the PCP limiter and plot the results at time $t=2$ in Figure \ref{fig:shockheating}, which also displays the cells where the PCP limiter is activated from  $t=0$ to $2$. We observe that the PCP limiter is only activated in a few cells near the moving shock.

	\begin{figure}[!h]
		\centering
		\subfloat[$\rho$]{
			\includegraphics[width=0.45\textwidth]{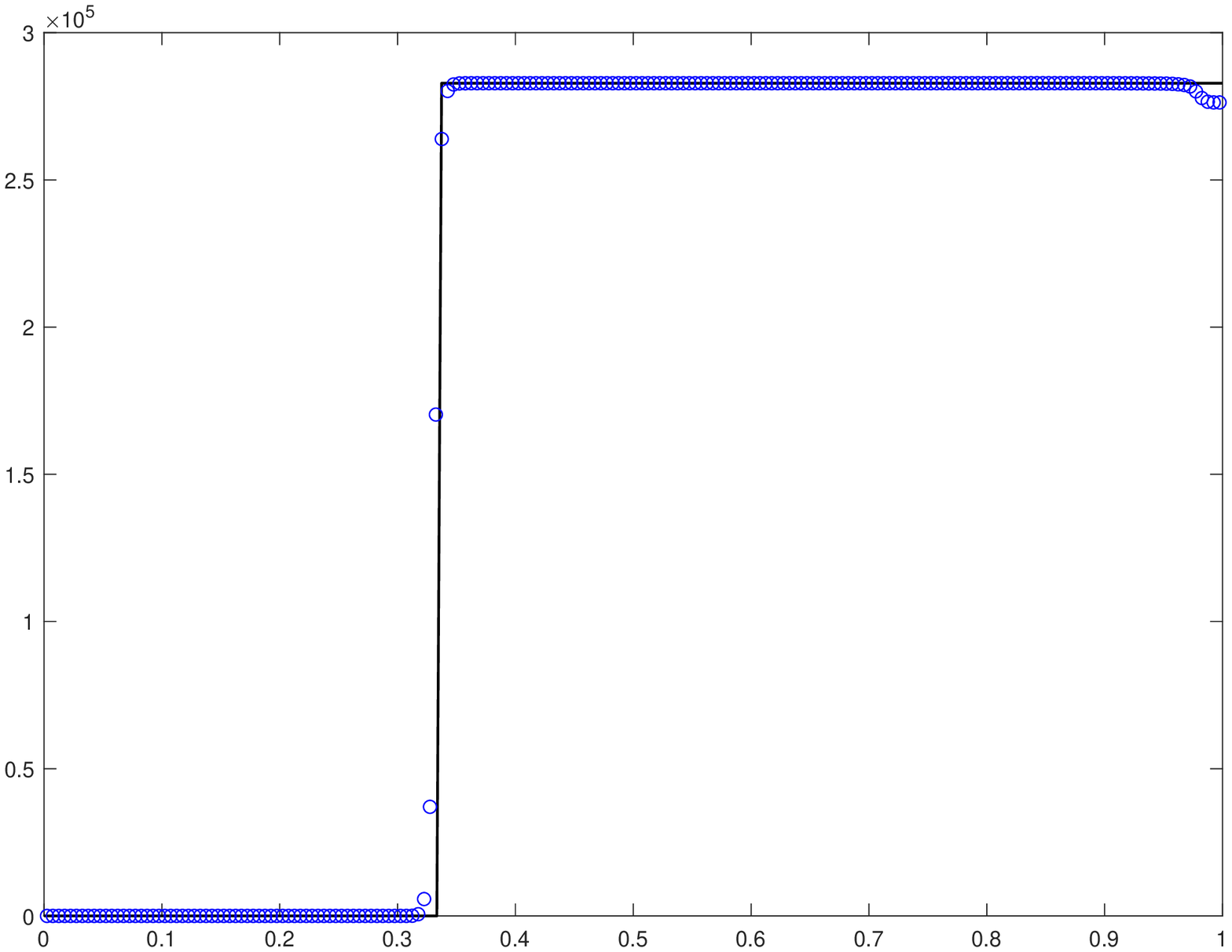}\label{fig:1Dshockheating_rho}
		}
		\subfloat[$v_1$]{
			\includegraphics[width=0.45\textwidth]{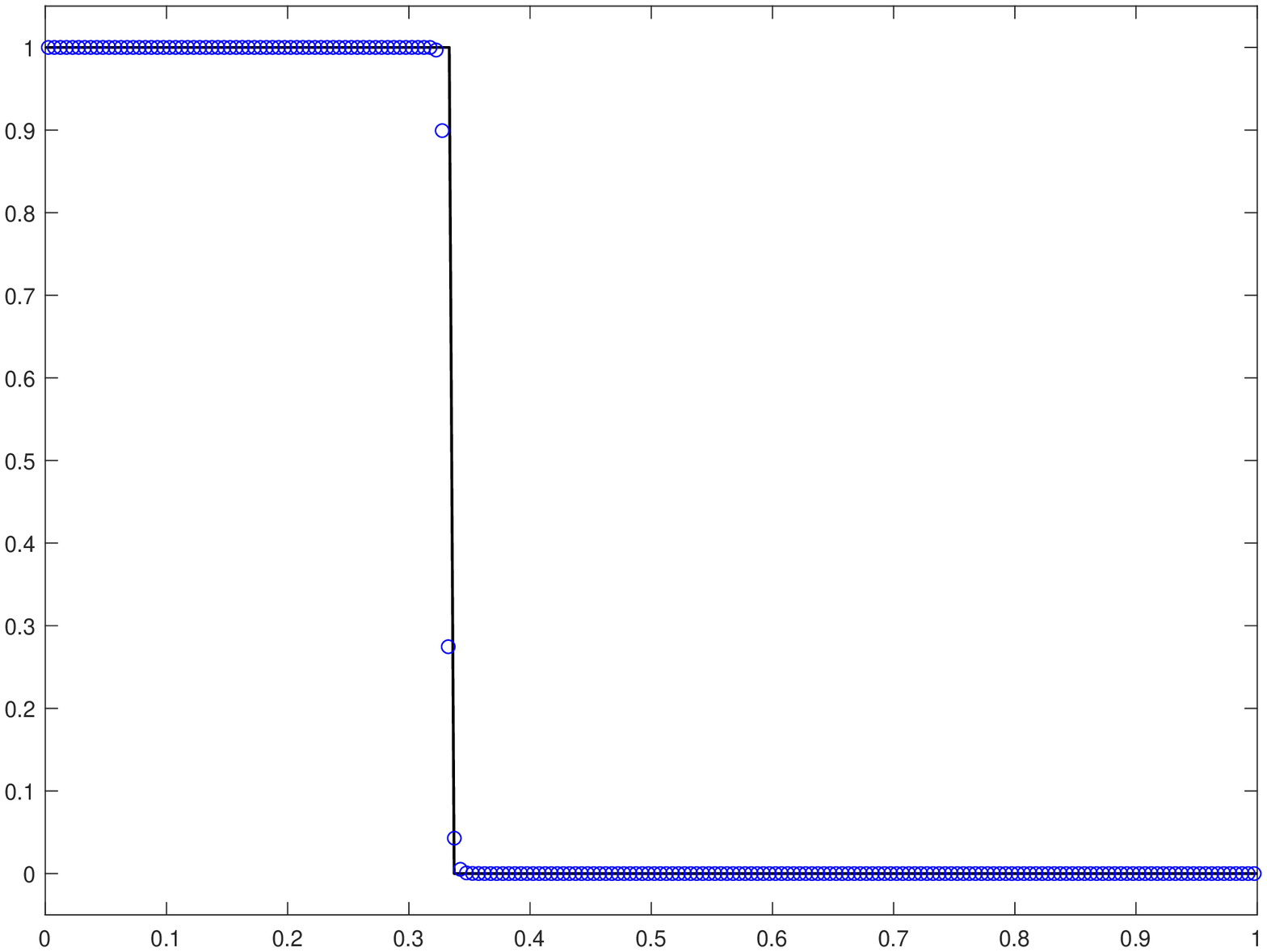}\label{fig:1Dshockheating_v1}
		}
		
		\subfloat[$p$]{
			\includegraphics[width=0.45\textwidth]{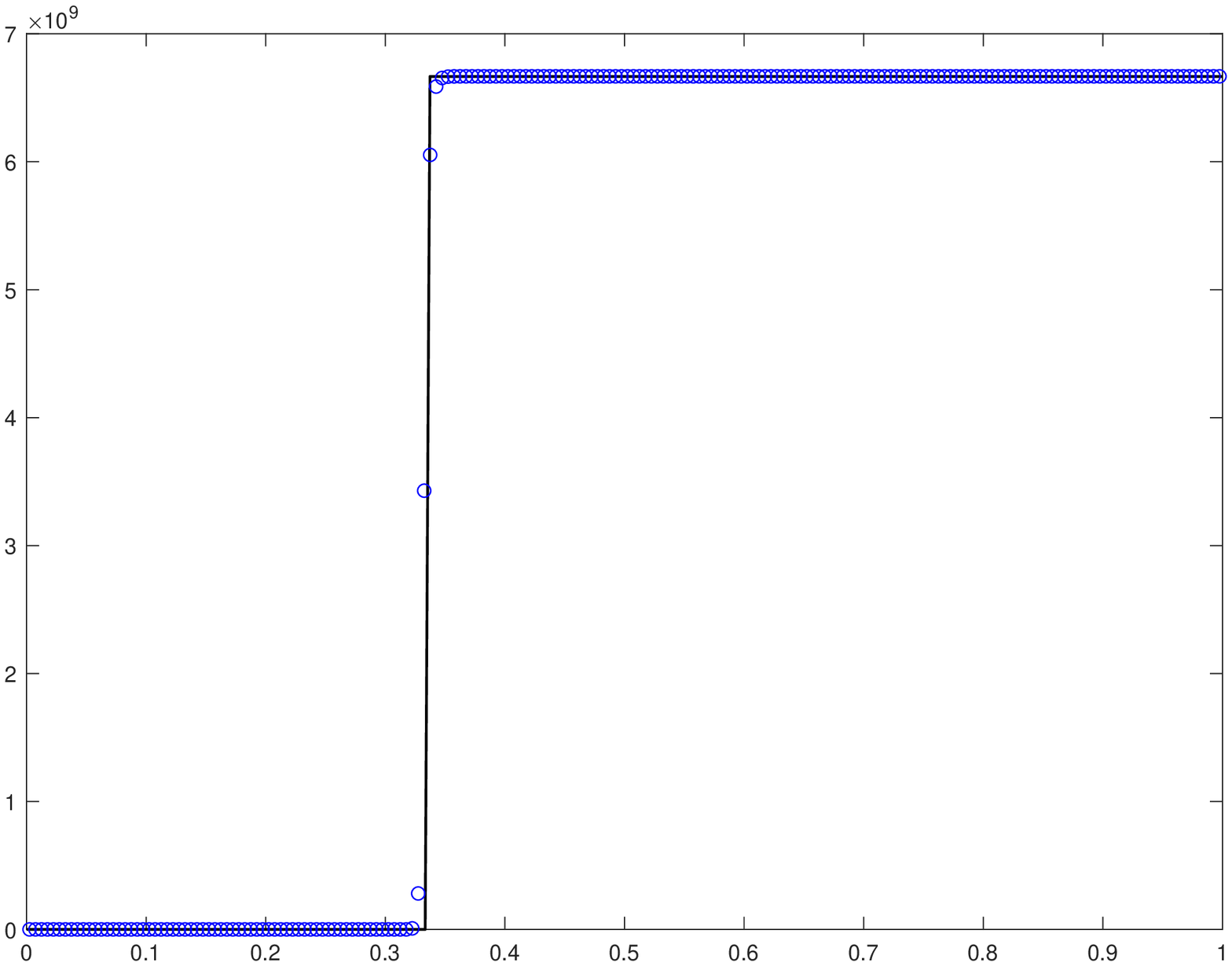}\label{fig:1Dshockheating_p}
		}
		\subfloat[PCP limited cells]{
			\includegraphics[width=0.45\textwidth]{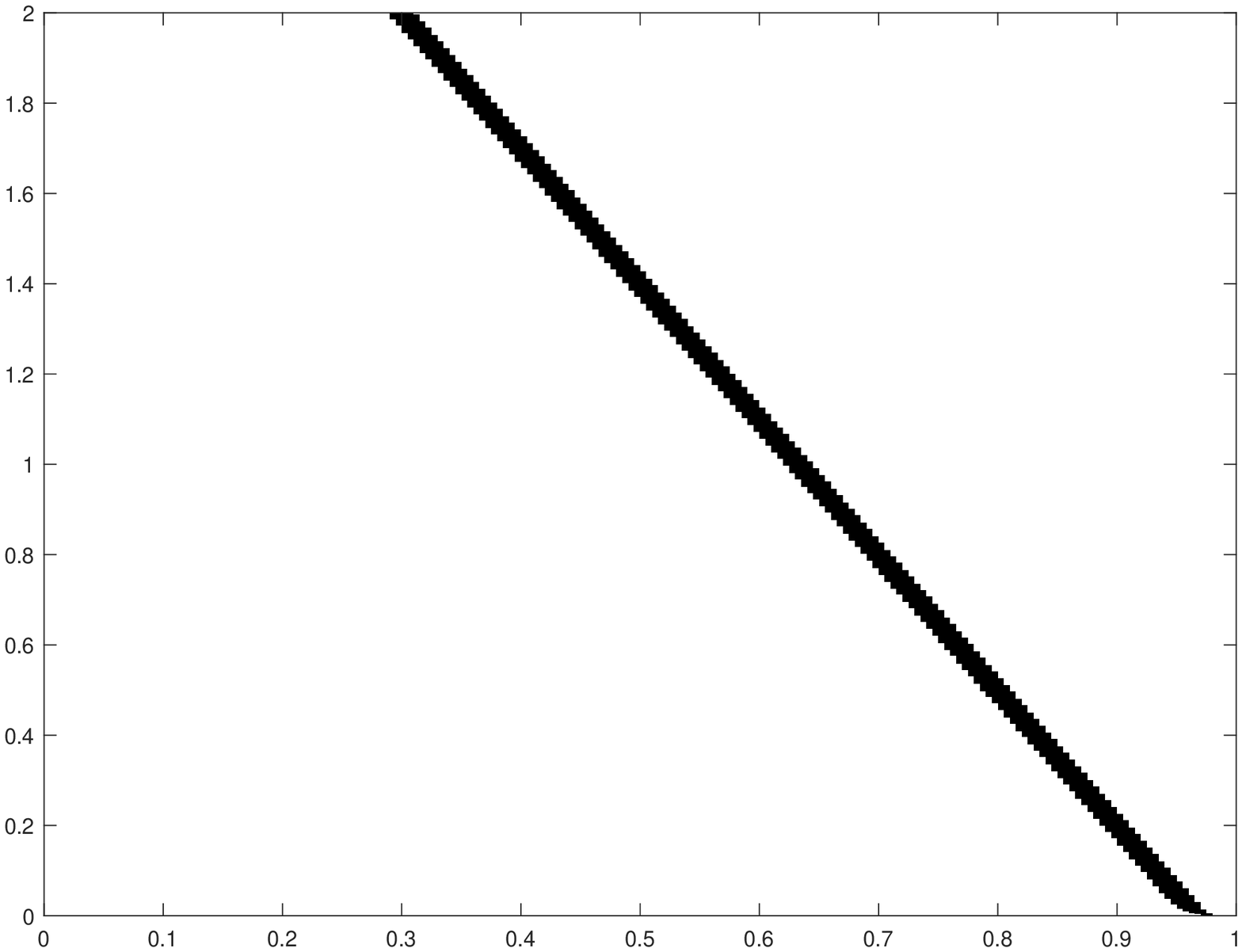}\label{fig:1Dshockheating_limitcell}
		}
	
		\captionsetup{font=small}
		\caption{Example \ref{test:shockheating}. The numerical solution (symbols ``$\circ$") and exact solution (solid lines) of density $\rho$, velocity $v_1$, and pressure $p$ at $t=2$ as well as the PCP limited cells.}\label{fig:shockheating}
	\end{figure}
	
\end{exa}

\subsection{Two-dimensional examples}

\begin{exa} [2D smooth problem]\label{test:2Dsmooth} 
	This example considers a 2D smooth problem 
	in the domain $\Omega=[0,2/\sqrt{3}]\times [0,2]$ with the initial data
	\begin{eqnarray}
		\bm{Q}=\left(1+0.999 \sin[2\pi(x\cos\alpha+y\sin\alpha)],\frac{0.9}{\sqrt{2}},\frac{0.9}{\sqrt{2}},0.01\right)^\top, 
	\end{eqnarray}
	where $\alpha=\pi/6$.  
	Due to the low density, large velocity close to the speed of light, and low pressure, the PCP limiting produce is necessary for successful simulation of this problem. 
	The simulations are performed on the meshes of $N \times N$ uniform cells with varied $N\in \{30,60,90,\dots,180\}$.  
	Table \ref{table:2D accuracy test} lists the numerical errors of the rest-mass density $\rho$ and the convergence rates in $L^1$, $L^2$ and $L^{\infty}$ norms at time $t=0.05$. 
	The results indicate that the 2D PCP HWENO scheme achieves fifth-order accuracy, which is not destroyed by the PCP limiter.

	%We use this example to validate the accuracy of our 2D HWENO scheme, as well as  
%	to assess the efficiency, robustness, and accuracy of our hybrid NR algorithm. 
%	The simulations are performed on the meshes of $N \times N$ uniform cells with varied $N\in \{30,60,90,\dots,180\}$. 
%	The results are summarized in Table \ref{table:2Dsmooth}, which reports the numerical errors and orders in $L^1$ norms, as well as the percentage of CPU time spent on recovering primitive variables. The results indicate that the hybrid NR algorithm outperforms the ``Hybrid-linear'' and ``Vel-Proxy'' algorithms, demonstrating its superior efficiency in this problem.

		\begin{table}[!htbp]
		\centering
		
		\captionsetup{font=small}
		\caption{ Example \ref{test:2Dsmooth}: Numerical errors in $L^1$, $L^2$ and $L^{\infty}$ norms and the corresponding convergence rates.}\label{table:2D accuracy test}
		
		\begingroup
		\setlength{\tabcolsep}{8pt} % Default value: 6pt
		\renewcommand{\arraystretch}{1.3} % Default value: 1
		
		\centering
		
		\small
\begin{tabular}{ccccccc}
	\hline
	$N$ & $L^1$ error     & Order    & $L^2$ error     & Order    & $L^{\infty}$ error   & Order  \\ \hline
	
	$30$  & 1.44E-03 & -       & 4.14E-03 & -       & 2.46E-02 & -       \\
	
	$60$  & 4.53E-06 & 8.31 & 1.79E-05 & 7.86 & 1.30E-04 & 7.56 \\
	
	$90$  & 1.18E-10 & 26.02 & 1.32E-10 & 29.15 & 1.86E-10 & 33.19 \\
	
	$120$ & 2.54E-11 & 5.35 & 2.83E-11 & 5.35 & 4.00E-11 & 5.34 \\
	
	$150$ & 7.97E-12 & 5.20 & 8.85E-12 & 5.20  & 1.27E-11 & 5.14 \\
	
	$180$ & 3.16E-12 & 5.08 & 3.51E-12 & 5.08 & 5.25E-12 & 4.85  \\ \hline
\end{tabular}
		
		\endgroup
		
	\end{table}

	%\begin{rem}
	%	Due to the characteristic decomposition and HWENO reconstruction performed in troubled-cells, the reconstruction time is much greater than that in non-troubled-cells. If we do not set all cells as troubled-cells, the proportion of time consumed for calculating fluxes will be even greater, and the efficiency improvement of the Hybrid NR will be more significant. Therefore, this test is actually the ``worst case scenario" for the efficiency improvement of the Hybrid NR.
	%\end{rem}
	
\end{exa}

\begin{exa} [2D Riemann problem I]\label{test:2Driemann1} The use of 2D Riemann problems as benchmark tests has become widespread to evaluate the ability of a scheme to capture complex 2D relativistic wave configurations. Both this and the next tests simulate 2D Riemann problems of the ideal relativistic fluid within the domain $[-0.5,0.5]^2$, which is divided into $400 \times 400$ uniform cells.
	
The initial conditions for this test are defined as follows: 
	\begin{equation*}
	\bm{Q}(x,y,0) = \begin{cases}
		(0.1,0,0,0.01)^\top,\quad&x>0,~y>0,\\
		(0.1,0.99,0,1)^\top,\quad&x<0,~y>0,\\
		(0.5,0,0,1)^\top,\quad&x<0,~y<0,\\
		(0.1,0,0.99,1)^\top,\quad&x>0,~y<0.
	\end{cases}
\end{equation*}	
Figure \ref{fig:2Driemann1} gives the contour of the density logarithm $\ln \rho$ and the cells where the PCP limiter is activated at $t=0.4$. It is shown in the figure that the initial discontinuities in the four regions cause two reflected curved shock waves and a complex mushroom structure. The details in structure are consistent with those reported in previous works \cite{2015High,chen2022physical}. We observe that there are only a few PCP limited cells near the two reflected curved shocks.

	\begin{figure}[!htb]
		\centering
		\subfloat[$ln~\rho$]{
			\includegraphics[width=0.45\textwidth]{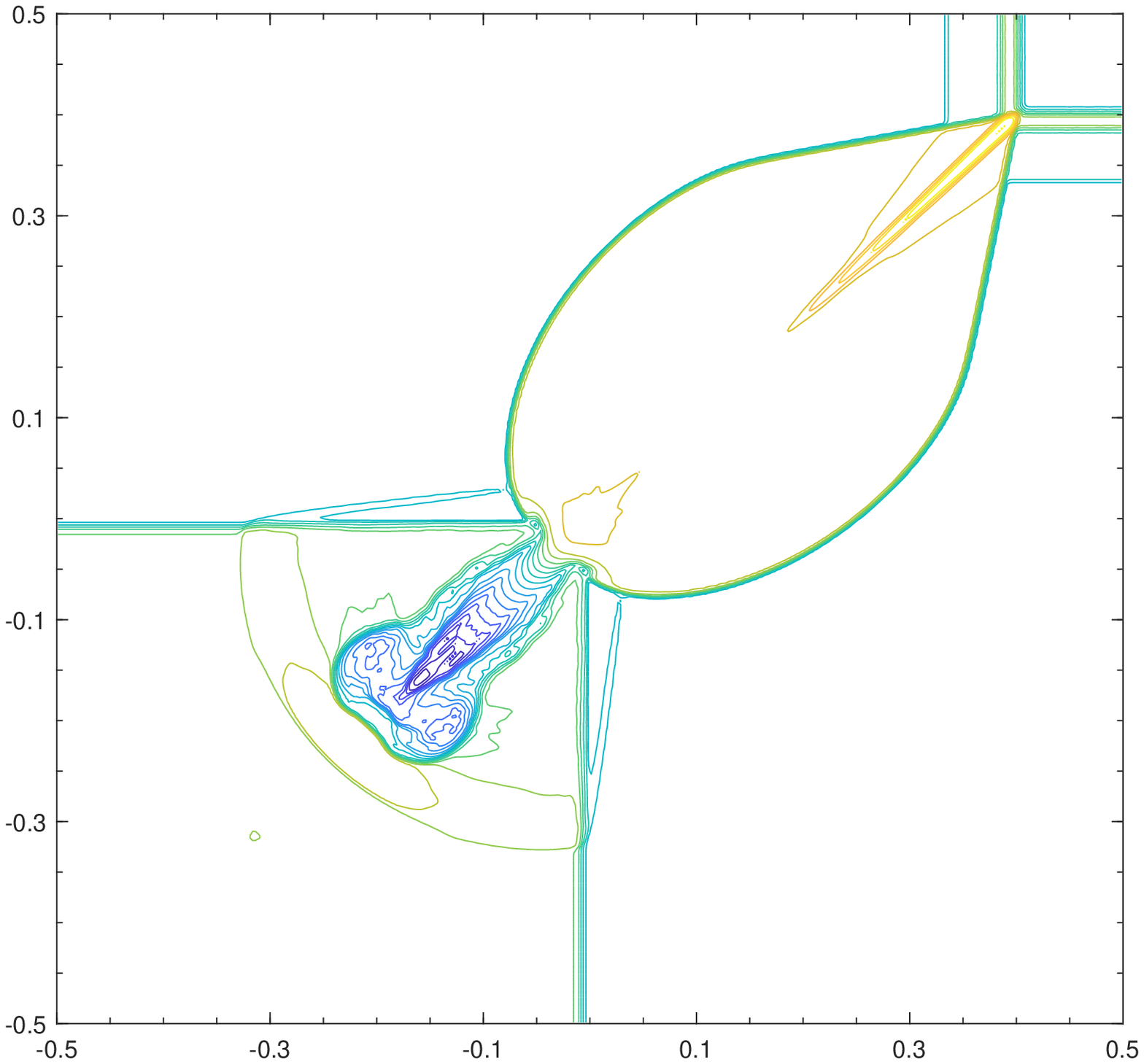}\label{fig:2Driemann4_logrho}
		}
		\subfloat[PCP limited cells]{
			\includegraphics[width=0.45\textwidth]{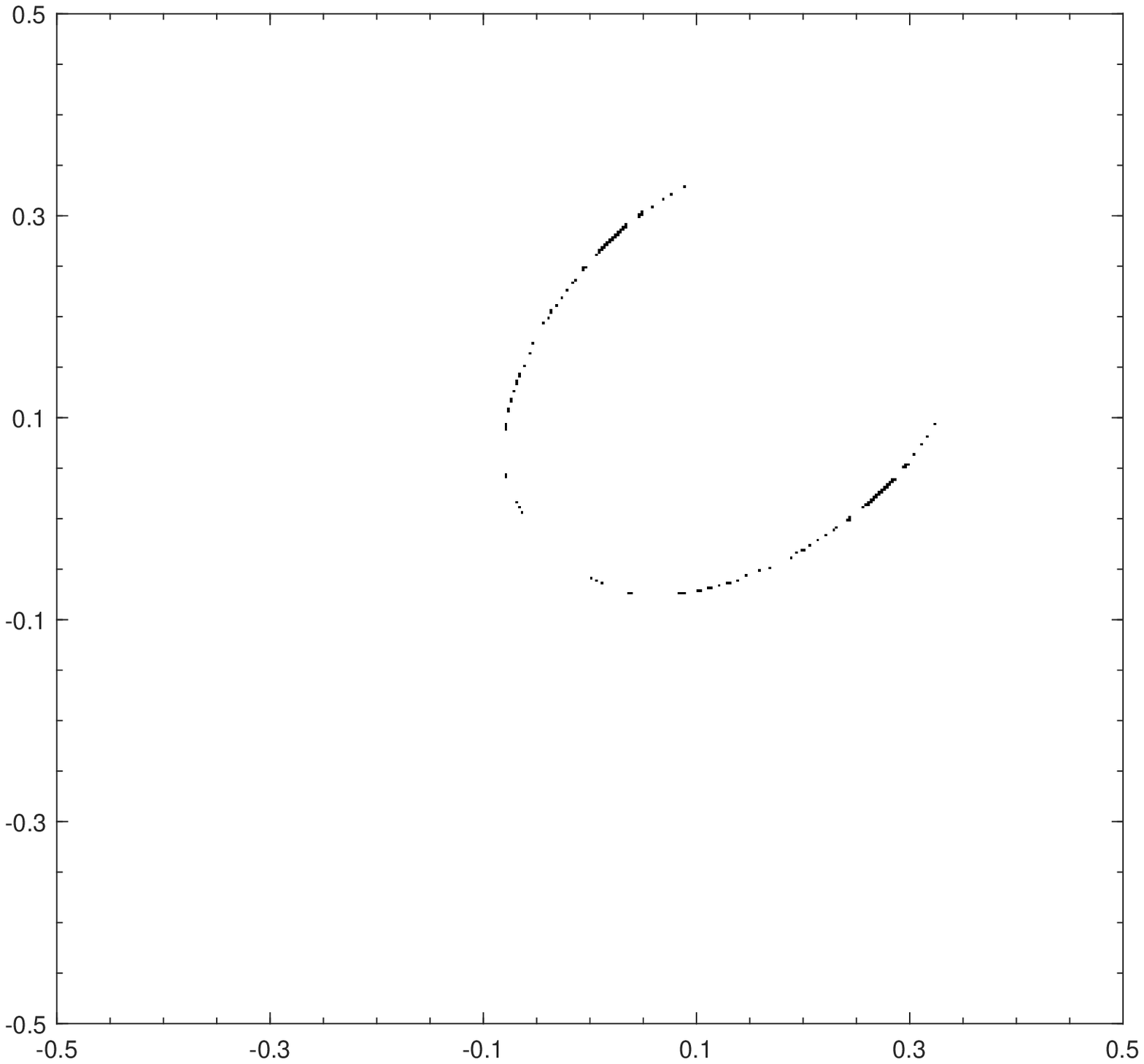}\label{fig:2Driemann4_limitcell}
		}
	
		\captionsetup{font=small}
		\caption{Example \ref{test:2Driemann1}.  The contours of the density logarithm $\ln \rho$ with 25 equally spaced contour lines from -6 to 1.9 within the domain $[0,1]^2$ and the PCP limited cells at $t=0.4$.}\label{fig:2Driemann1} 
	\end{figure}

\end{exa}

\begin{exa} [2D Riemann problem II]\label{test:2Driemann2} This example investigates a more ultra-relativistic 2D Riemann problem, which was first proposed in \cite{2015High}. The initial conditions are defined as  
	\begin{equation*}
		\bm{Q}(x,y,0) = \begin{cases}
			(0.1,0,0,20)^\top,~~&x>0,y>0,\\
			(0.00414329639576,0.9946418833556542,0,0.05)^\top,~~&x<0,y>0,\\
			(0.01,0,0,0.05)^\top,~~&x<0,y<0,\\
			(0.00414329639576,0,0.9946418833556542,0.05)^\top,~~&x>0,y<0.
		\end{cases}
	\end{equation*}
In this problem, the maximum initial velocity of fluid is larger than that in 2D Riemann problem I. We use the proposed PCP HWENO scheme to simulate the problem on a mesh of $400 \times 400$ uniform cells. Furthermore, we utilize this problem to demonstrate the importance of rescaling eigenvectors in characteristic decomposition. We compare two rescaling approaches discussed in Remark \ref{rem:eigvector}.

Figure \ref{fig:2Driemann2} shows the contours of the density logarithm $\ln \rho$ at $t=0.4$, obtained by our PCP HWENO scheme using three different methods: the ``unitization" rescaling approach, the ``matching" rescaling approach, and no rescaling. 
As we can see, 
the ``matching" rescaling approach exhibits the best performance, while the numerical solutions computed using the ``unitization" rescaling approach and without rescaling exhibit serious oscillations.

\begin{figure}[!htb]
	\centering
	\subfloat[$\ln \rho$: without rescaling]{
		\includegraphics[width=0.45\textwidth]{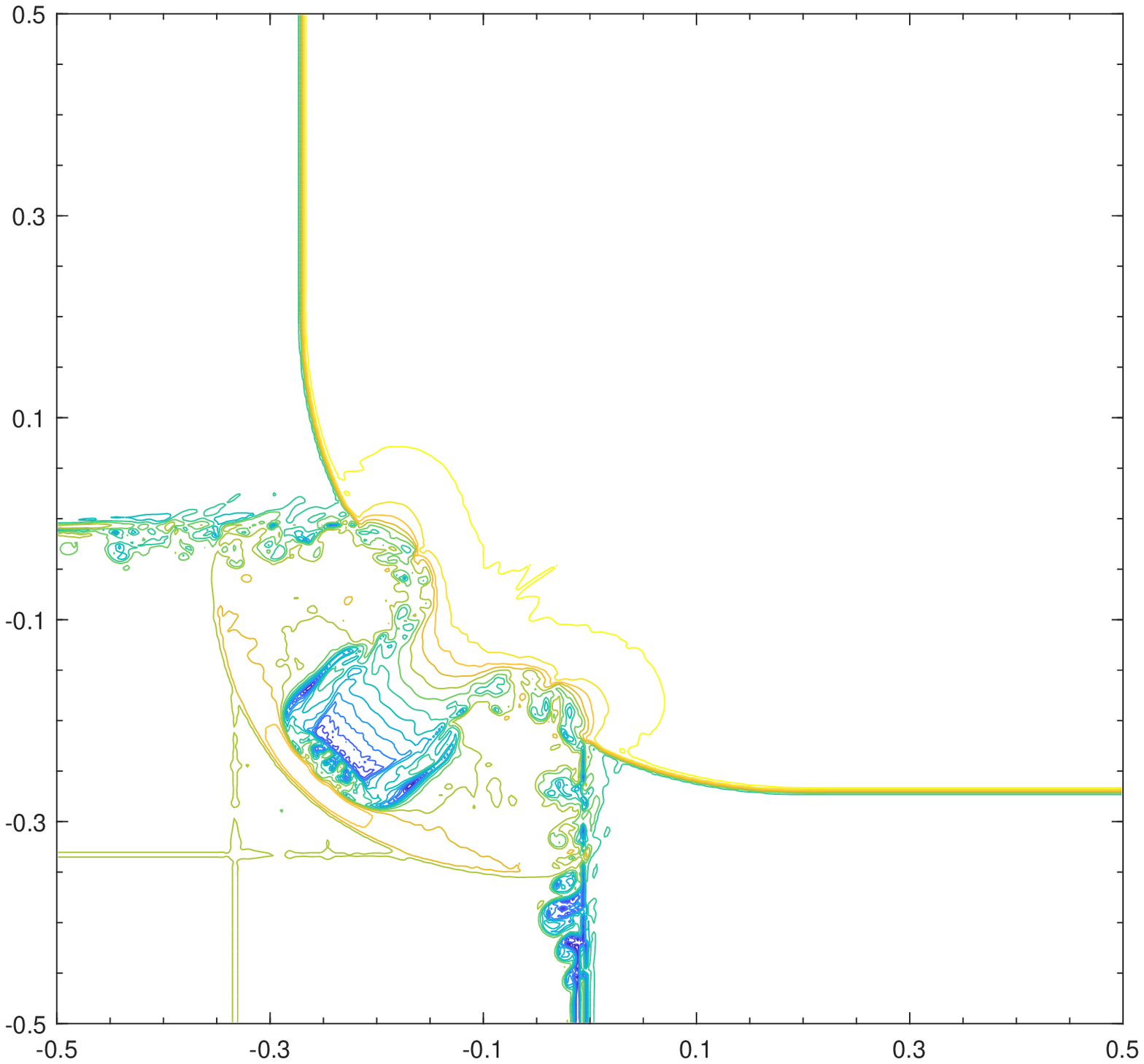}\label{fig:2Driemann5_logrho2}
	}
	\subfloat[$\ln \rho$: with ``unitization" rescaling]{
		\includegraphics[width=0.45\textwidth]{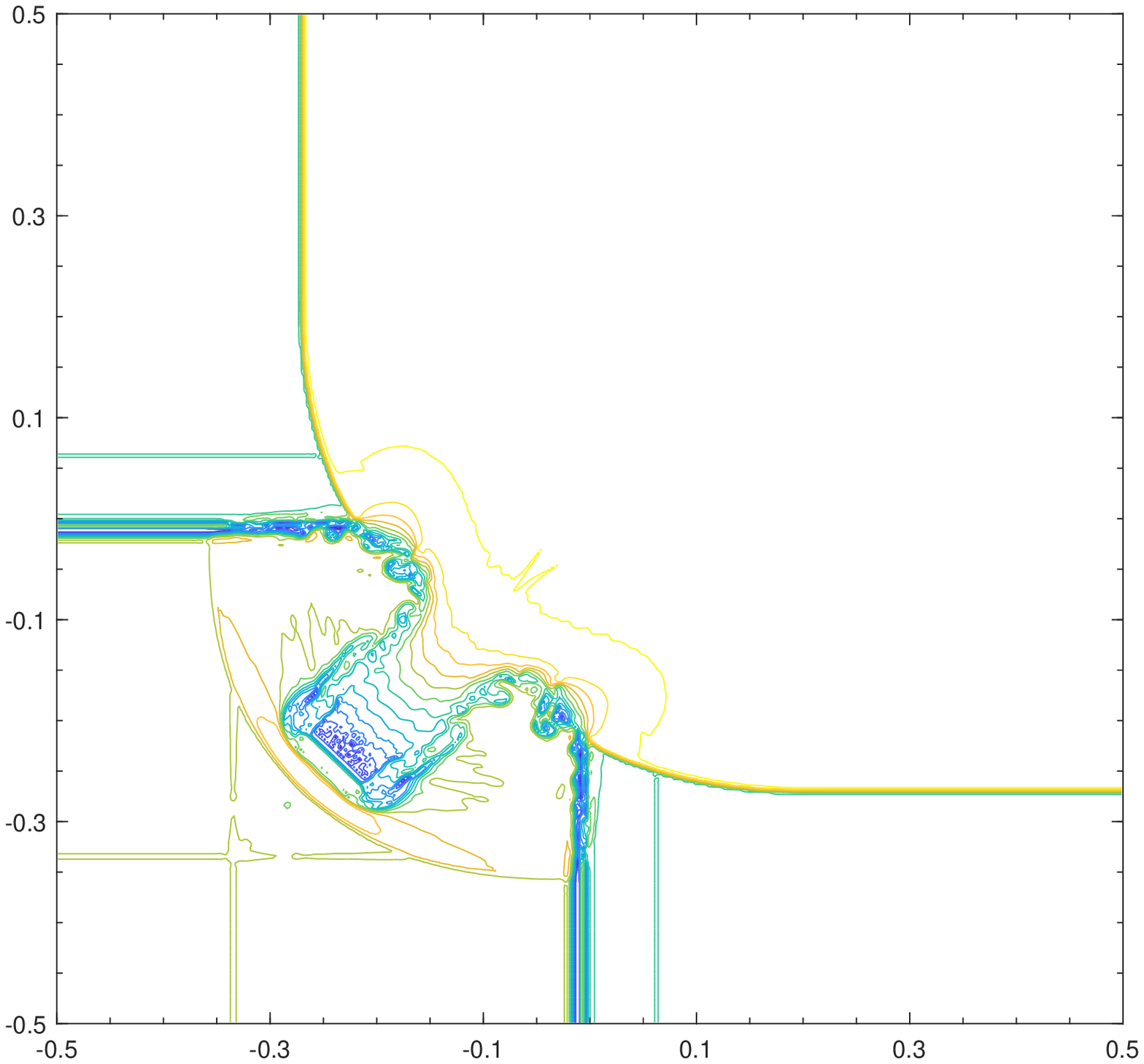}\label{fig:2Driemann5_logrho1}
	}

	\subfloat[$\ln \rho$: with ``matching" rescaling]{
		\includegraphics[width=0.45\textwidth]{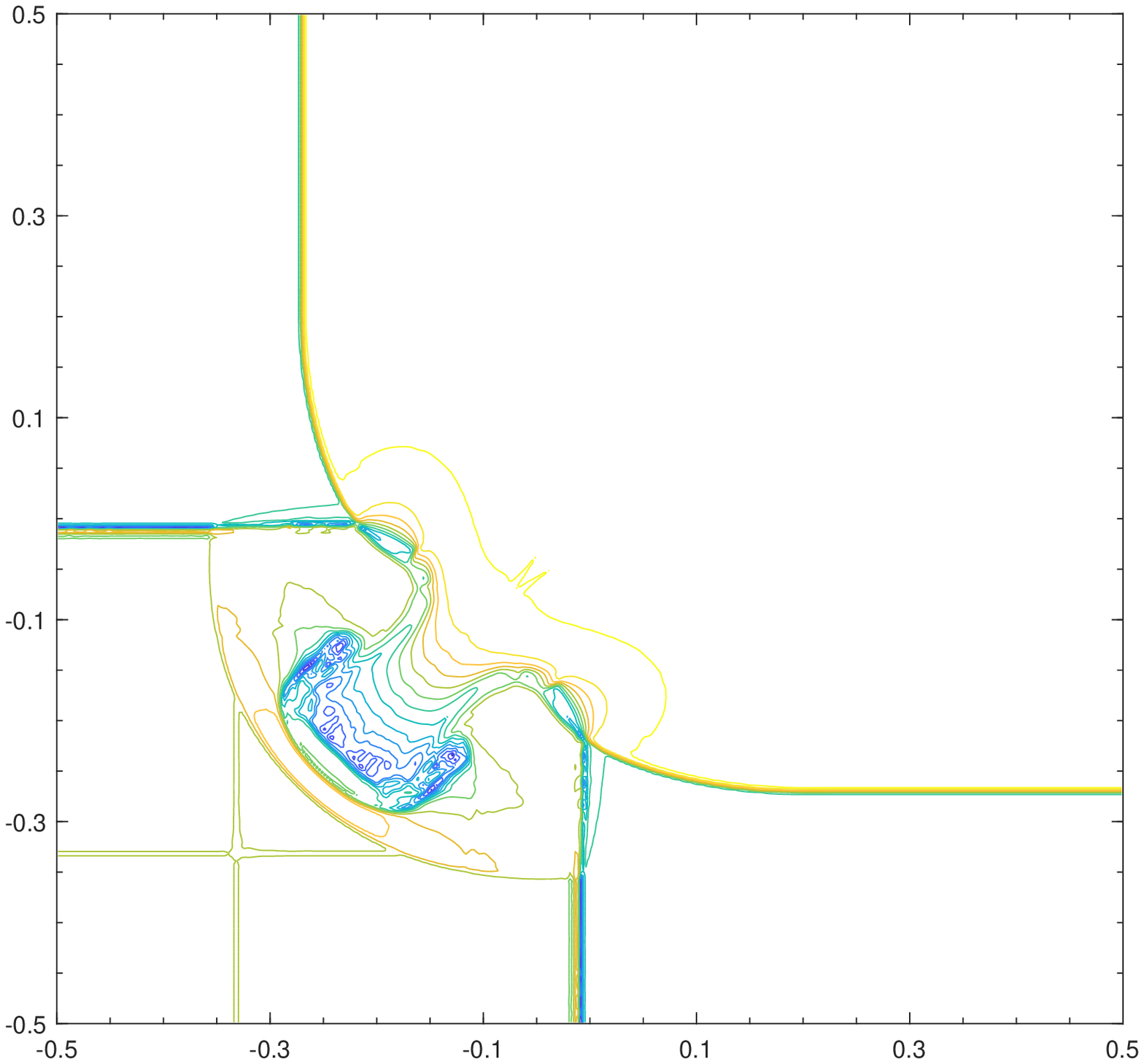}\label{fig:2Driemann5_logrho0}
	}
	\subfloat[PCP limited cells: with ``matching" rescaling]{
		\includegraphics[width=0.45\textwidth]{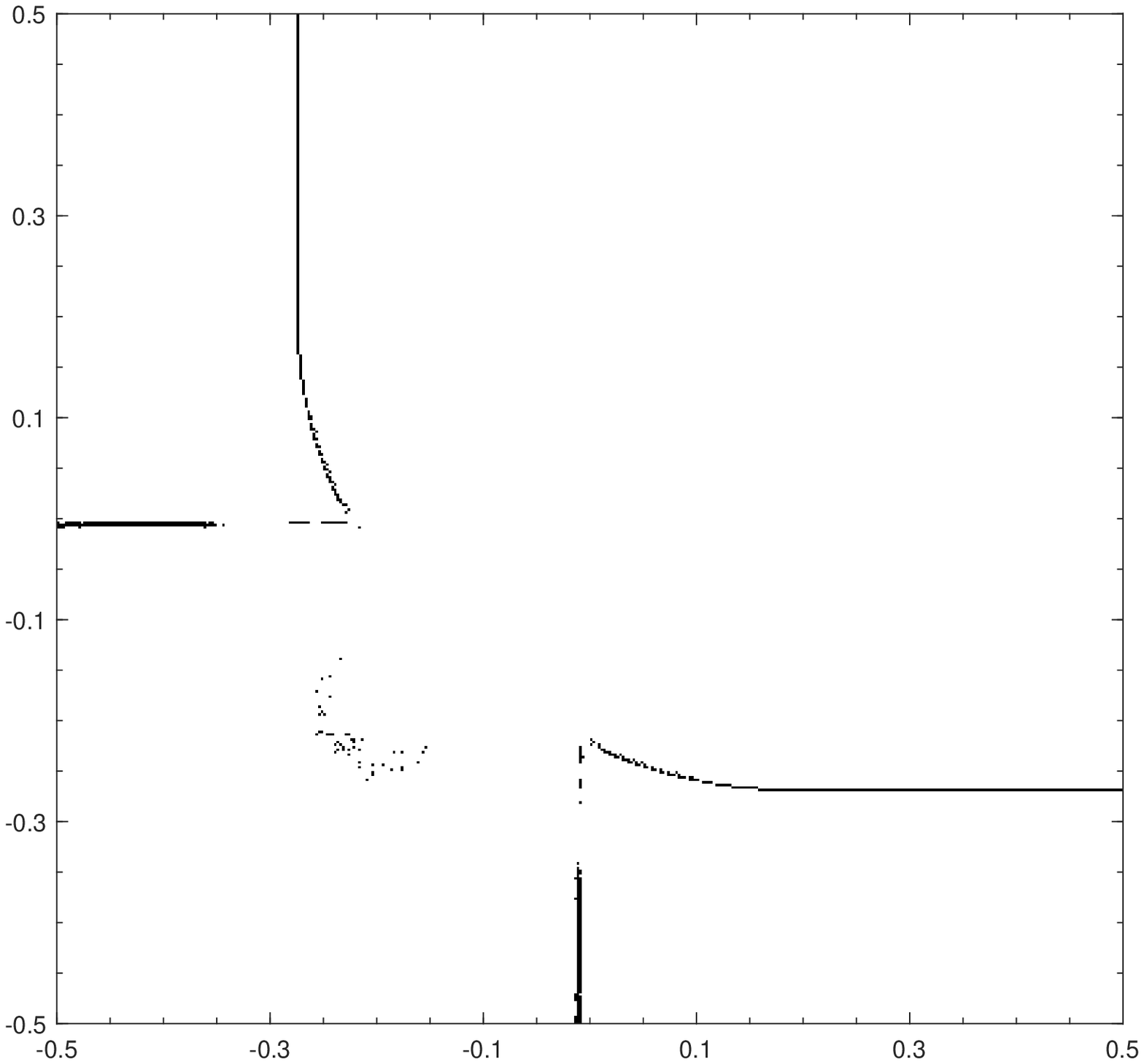}\label{fig:2Driemann5_limitcell_0}
	}

	\captionsetup{font=small}
	\caption{Example \ref{test:2Driemann2}\label{fig:2Driemann2}: The contours of the density logarithm $\ln \rho$ and the PCP limited cells at $t=0.4$ (25 equally spaced contour lines from -9 to 2).}
\end{figure}

\end{exa}

\begin{exa} [Shock-vortex interaction problems]\label{test:shockvortex interaction} 
	This example studies 
	the interaction of a vortex with a shock. Pao and Salas \cite{pao1981numerical} were the first to show this problem computationally in the non-relativistic case, while the special RHD case was studied in \cite{balsara2016subluminal,duan2019high}. In our case, we have set the velocity magnitude of the vortex as $w=0.9$ and the adiabatic index $\gamma$  as $1.4$. The initial rest-mass density and pressure are given by
	\begin{equation*}
		\rho(x,y)=(1-\alpha e^{1-r^2})^{\frac{1}{\gamma-1}}, ~~p=\rho^{\gamma}
	\end{equation*}
where
	\begin{equation*}
		\alpha=\frac{(\gamma-1)/ \gamma}{8\pi^2}\mathcal{\epsilon}^2,~~r=\sqrt{x_0^2+y_0^2},
	\end{equation*}
and $\epsilon$ represents the vortex strength. 
Using the Lorentz transformation, we can deduce that
	\begin{equation*}
		x_0=xW_w,~~y_0=y,~~W_w=\frac{1}{\sqrt{1-w^2}}.
	\end{equation*}
The initial velocities are given by 
	\begin{equation}
		v_1=\frac{v_1^0-w}{1-v_1^0w}, \qquad v_2=\frac{v_2^0}{W_w(1-v_1^0w)}, \nonumber
	\end{equation}
where
	\begin{equation}
		(v_1^0,v_2^0)=(-y_0,x_0)f,~~f=\sqrt{\frac{\beta}{1+\beta r^2}}, ~~\beta=\frac{2\gamma\alpha e^{1-r^2}}{2\gamma-1-\gamma\alpha e^{1-r^2}}.\nonumber
	\end{equation}
The computation domain is $\left[-17,3\right]\times\left[-5,5\right]$, which is divided into $800 \times 400$ uniform cells. 
	The initial vortex is centered at $(0,0)$, and there is a shock at $x=-6$ that far away from the vortex. The initial data in $x>6$ can be calculated by the vortex condition above, and the post-shock state in $x<6$ is given by 
	\begin{equation*}
		\bm{Q}(x,y,0)=(4.891497310766981,-0.388882958251919,0,11.894863258311670)^\top.
	\end{equation*}
	We apply inflow and outflow boundary conditions at the right and left boundaries of the domain, respectively, and reflection boundary conditions are applied on the bottom and top boundaries.

	We test our scheme in two different vortex strengths:
	\begin{itemize}
		\item A mild vortex with ${\epsilon}=5$ as in \cite{duan2019high}.
		\item A demanding vortex with ${\epsilon}=10.0828$ as in \cite{chen2022physical}. In this case, the minimum rest-mass density and pressure are $7.8337\times 10^{-15}$ and $1.7847\times10^{-20}$, respectively. We observe that the HWENO code without the PCP limiter cannot run this challenging test for even one time step, demonstrating the importance of the PCP limiter.
	\end{itemize}
	Figures \ref{fig:shockvortex1} and \ref{fig:shockvortex2} show the schlieren images of $\log_{10}(1+\left | \nabla \rho\right |)$ and $\left | \nabla  p \right |$. The subtle structures in our results are in good agreement with those reported in \cite{duan2019high,chen2022physical}, validating the effectiveness of our proposed PCP HWENO schemes in capturing complex waves and shocks.
	
\begin{figure}[!htb]
	\centering
	\subfloat[$\log_{10}(1+\left | \nabla \rho\right |)$ at $t=0$]{
		\includegraphics[width=0.45\textwidth]{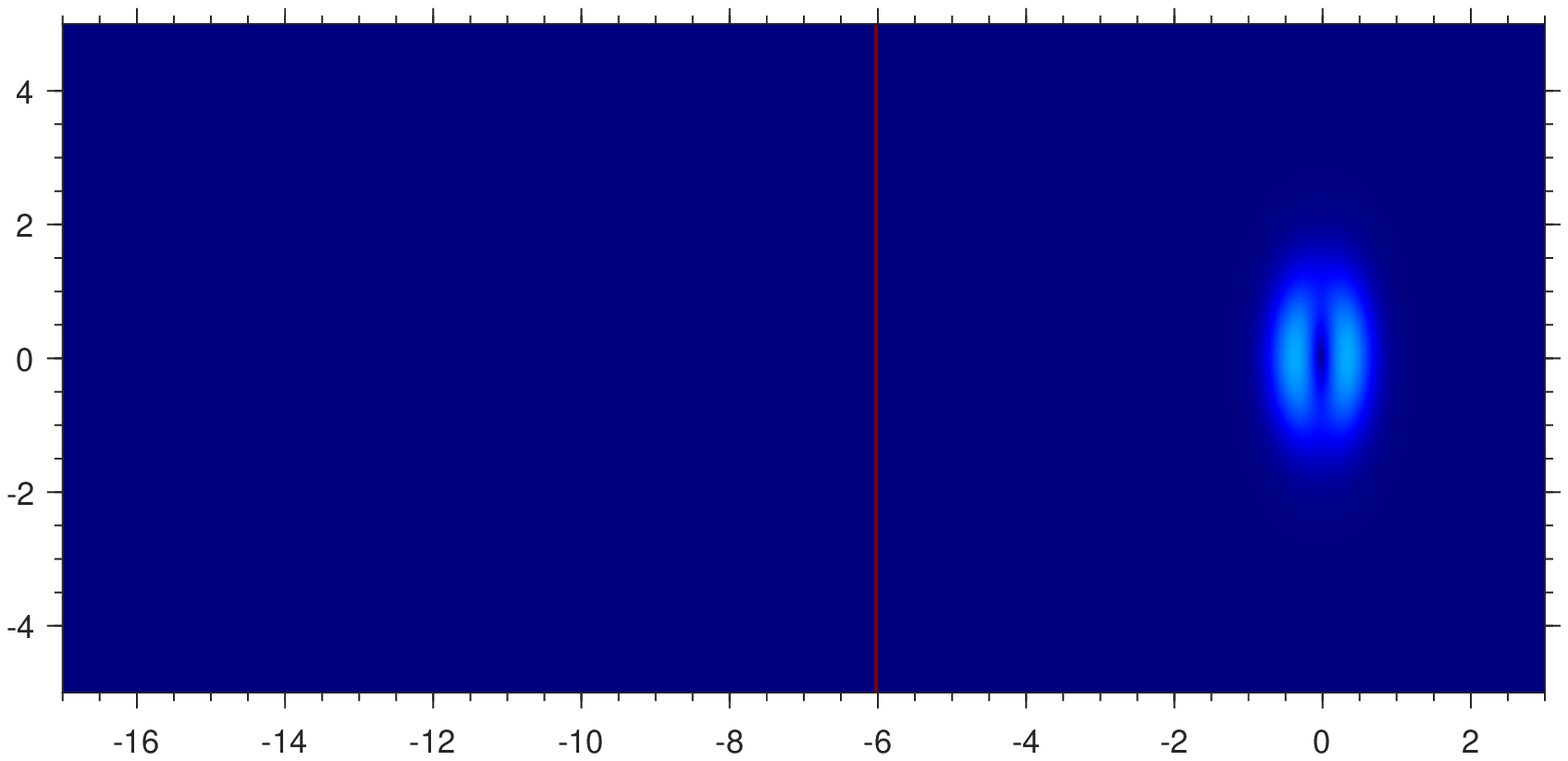}\label{fig:shockvortex1rho0}
	}
	\subfloat[$\log_{10}(1+\left | \nabla \rho\right |)$ at $t=19$]{
		\includegraphics[width=0.45\textwidth]{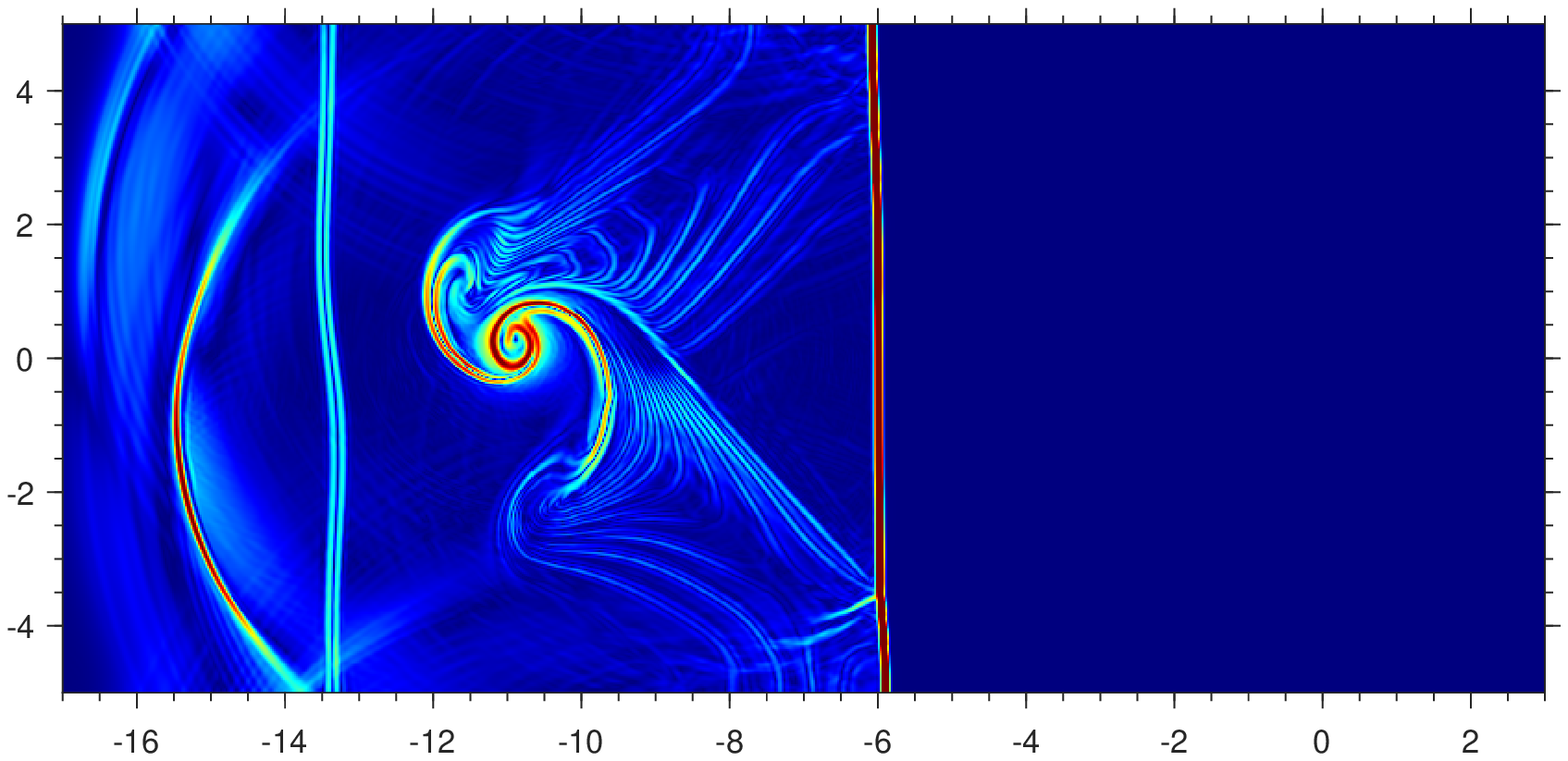}\label{fig:shockvortex1rho19}
	}

	\subfloat[$\left | \nabla  p \right |$ at $t=0$]{
		\includegraphics[width=0.45\textwidth]{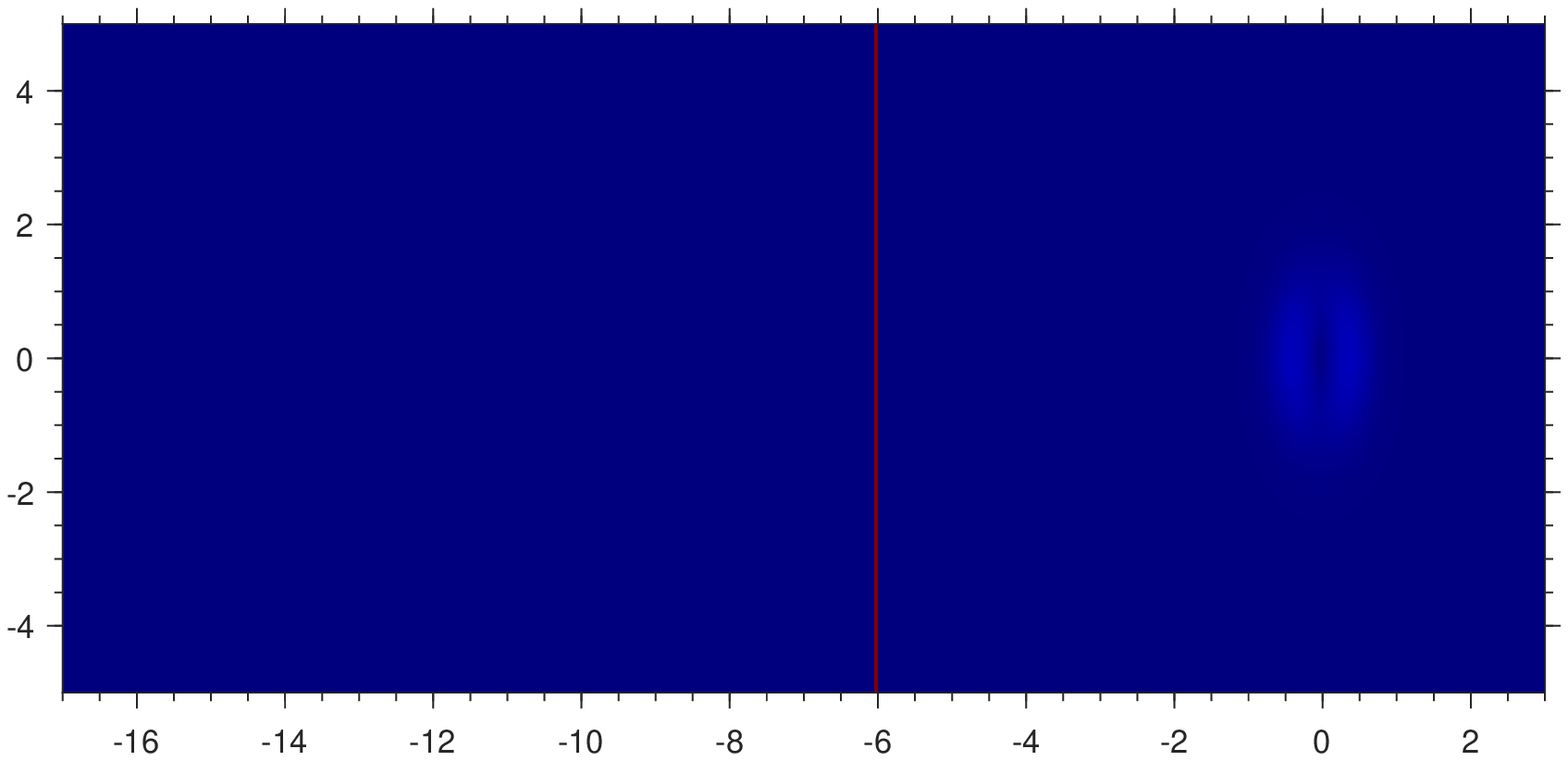}\label{fig:shockvortex1p0}
	}
	\subfloat[$\left | \nabla  p \right |$ at $t=19$]{
		\includegraphics[width=0.45\textwidth]{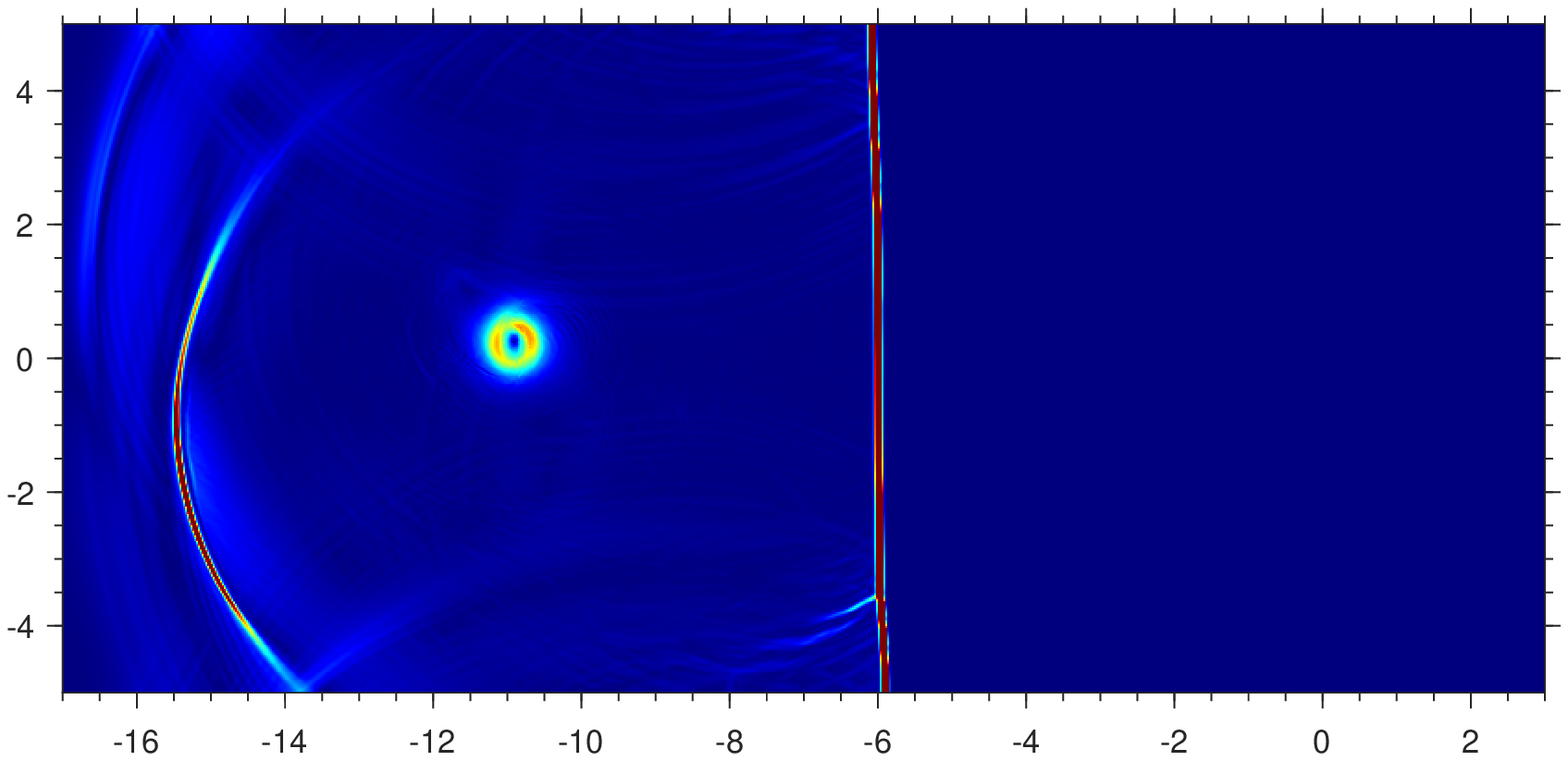}\label{fig:shockvortex1p19}
	}

%	\subfloat[The ``troubled" cells at $t$=0]{
%		\includegraphics[width=0.45\textwidth]{rhd results//2Dshockvortex1_ini_troublecell.eps}\label{fig:shockvortex1_troublecell0}
%	}
%	\subfloat[The ``troubled" cells at $t$=19]{
%		\includegraphics[width=0.45\textwidth]{rhd results//2Dshockvortex1_troublecell.eps}\label{fig:shockvortex1_troublecell19}
%	}

	\captionsetup{font=small}
	\caption{Example \ref{test:shockvortex interaction}\label{fig:shockvortex1} with the mild vortex: The schlieren images of $\log_{10}(1+\left | \nabla  \rho\right |)$ from 0 to 1 and schlieren images of $\left | \nabla  p\right |$ from 0 to 20. }
\end{figure}

\begin{figure}[!htb]
	\centering
	\subfloat[$\log_{10}(1+\left | \nabla \rho\right |)$ at $t=0$]{
		\includegraphics[width=0.45\textwidth]{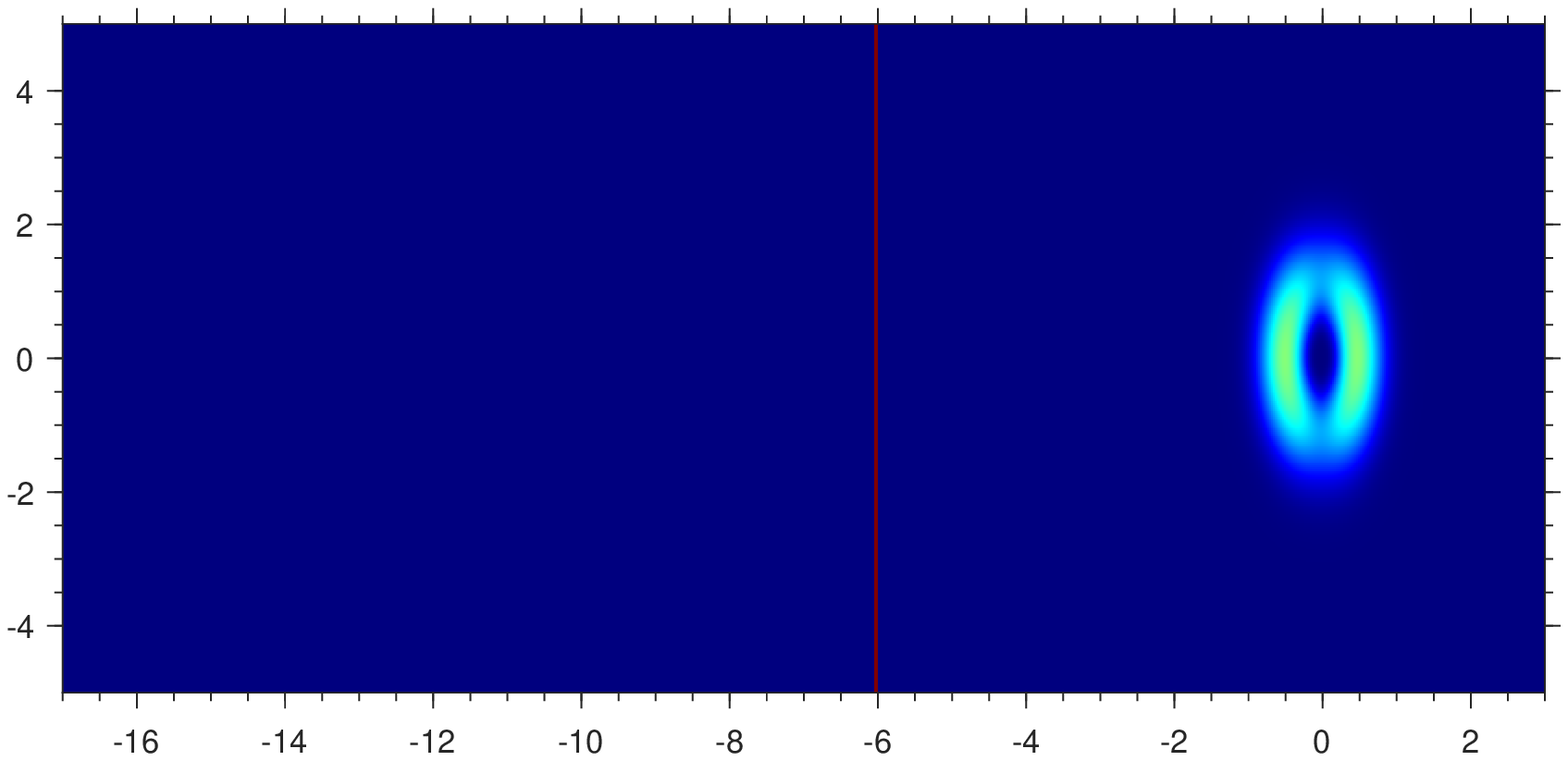}\label{fig:shockvortex2rho0}
	}
	\subfloat[$\log_{10}(1+\left | \nabla \rho\right |)$ at $t=19$]{
		\includegraphics[width=0.45\textwidth]{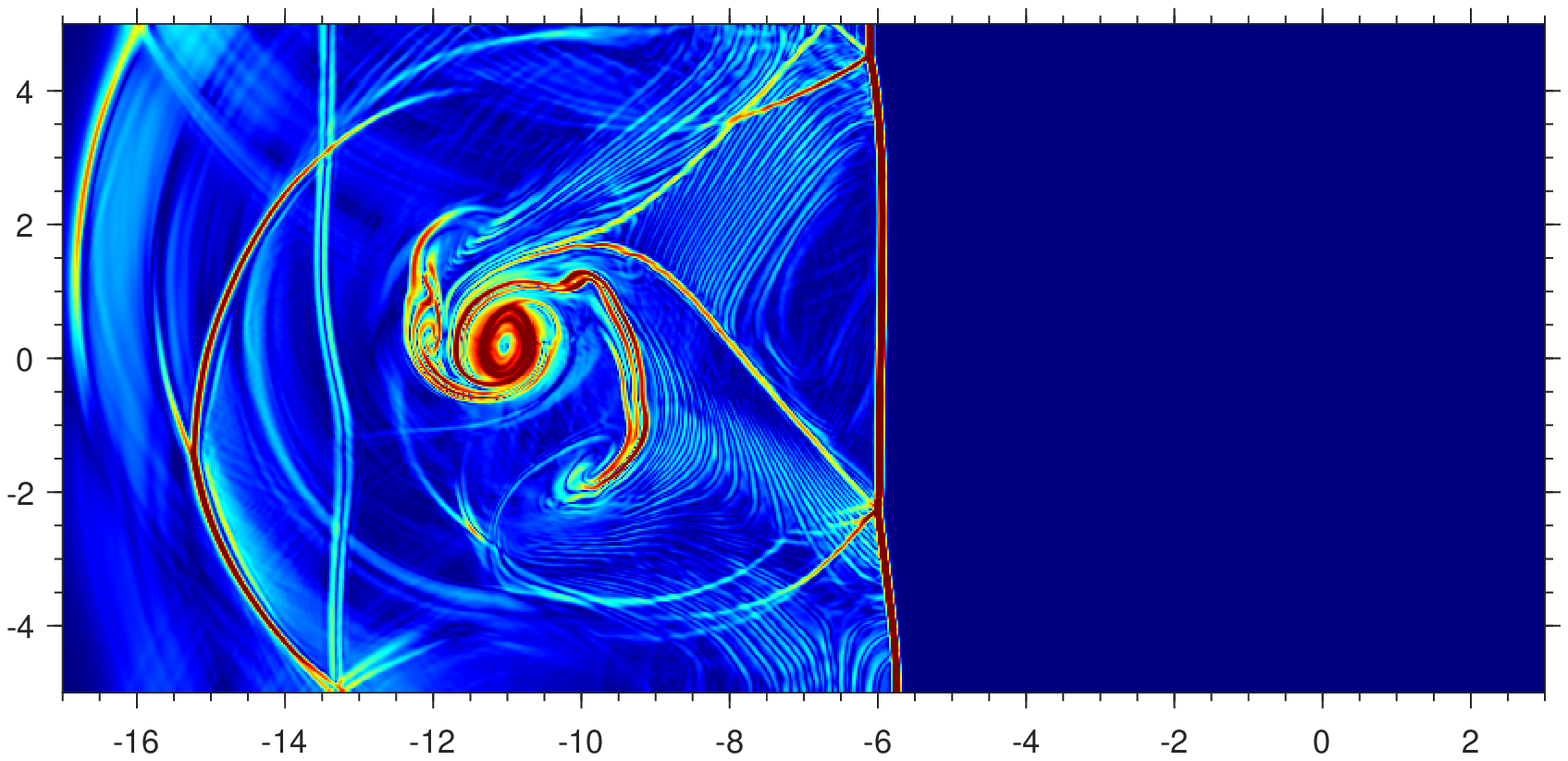}\label{fig:shockvortex2rho19}
	}
	
	\subfloat[$\left | \nabla  p \right |$ at $t=0$]{
		\includegraphics[width=0.45\textwidth]{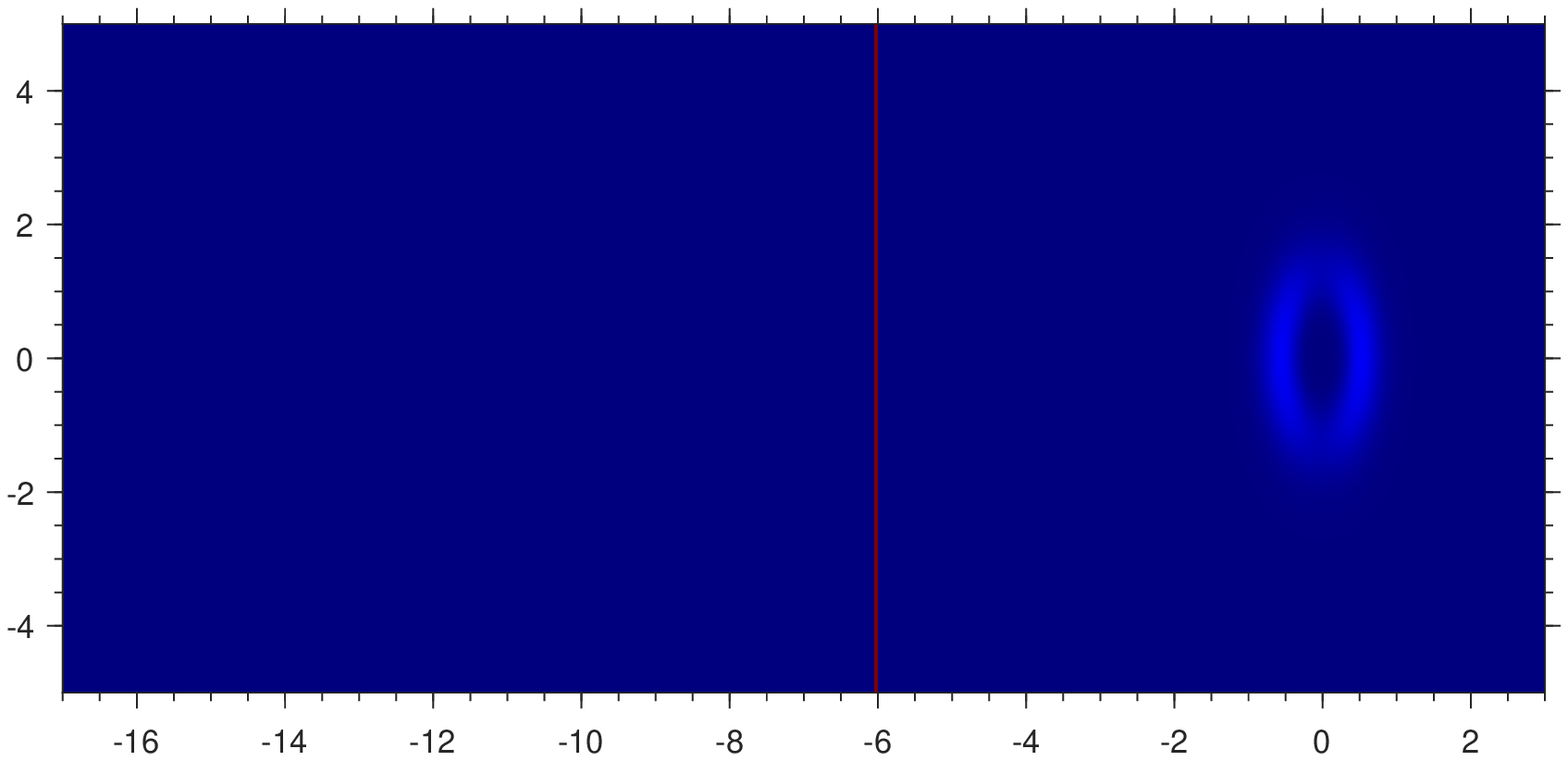}\label{fig:shockvortex2p0}
	}
	\subfloat[$\left | \nabla  p \right |$ at $t$=19]{
		\includegraphics[width=0.45\textwidth]{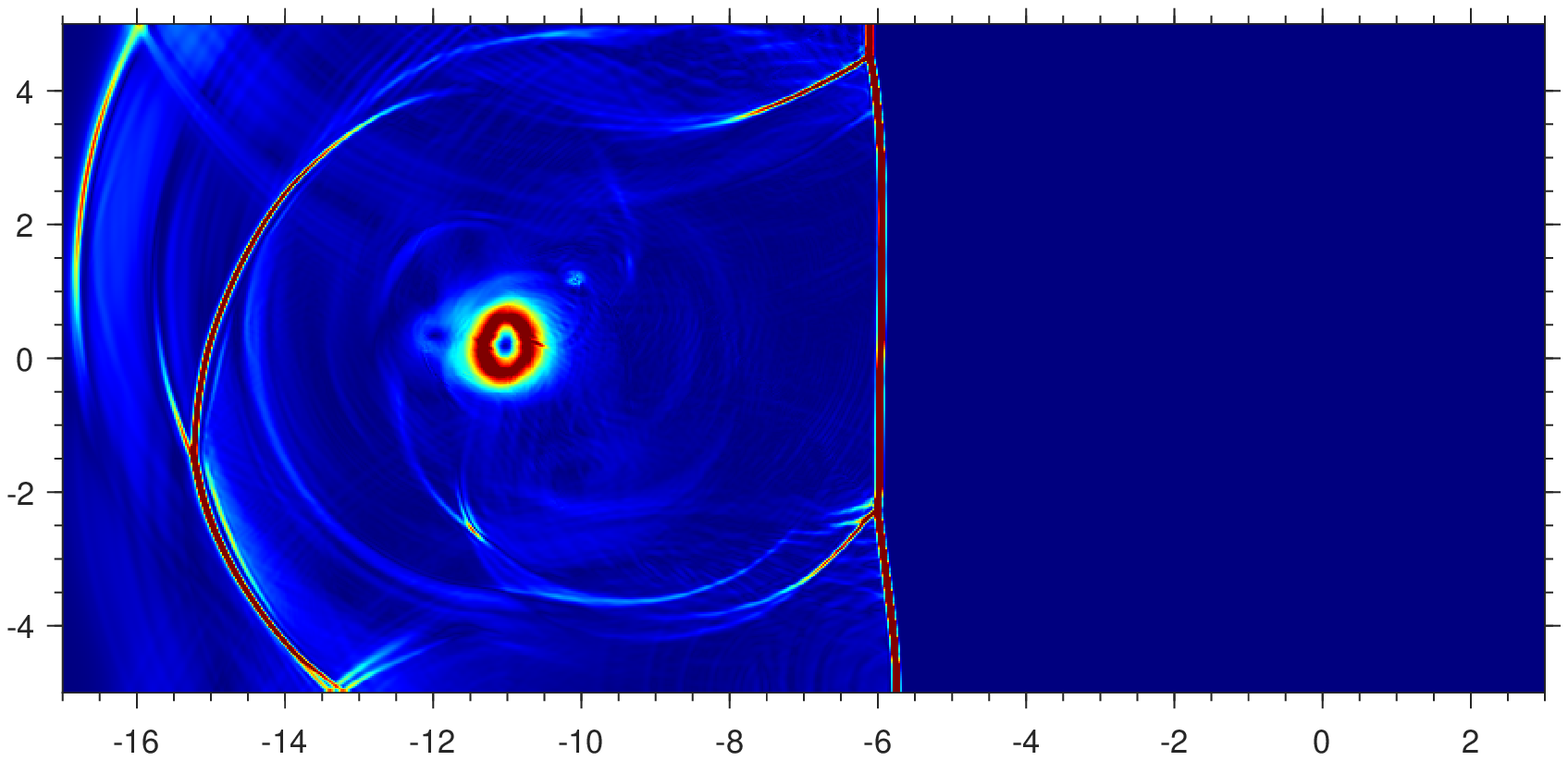}\label{fig:shockvortex2p19}
	}
	
%	\subfloat[PCP limited cells at $t$=0]{
%		\includegraphics[width=0.45\textwidth]{rhd results//2Dshockvortex2_ini_troublecell.eps}\label{fig:shockvortex2_troublecell0}
%	}
%	\subfloat[PCP limited cells at $t$=19]{
%		\includegraphics[width=0.45\textwidth]{rhd results//2Dshockvortex2_troublecell.eps}\label{fig:shockvortex2_troublecell19}
%	}

		\captionsetup{font=small}
		\caption{Example \ref{test:shockvortex interaction}\label{fig:shockvortex2} with the demanding vortex: the schlieren images of $\log_{10}(1+\left | \nabla  \rho\right |)$ from 0 to 1 and schlieren images of $\left | \nabla  p\right |$ from 0 to 20. } %The bottom two pitcures show the PCP limited cells at $t=0$ and $t=19$. 
\end{figure}

\end{exa}

\begin{exa} [Axisymmetric relativistic jets]\label{test:jet} 
	This test	simulate a relativistic jet by solving the RHD equations in cylindrical coordinates (see Section \ref{section:cylindrical coordinates} for details). Relativistic jet flows have been extensively investigated by many researchers  \cite{marti1997morphology, 2015High, zhang2006ram, qin2016bound, chen2022physical}. The computational domain consists of a 2D cylindrical box with dimensions of $(0\leq r\leq15,~0\leq z\leq45)$, which is discretized into $375 \times 1125$ uniform cells.
	The initial conditions in the computational domain are given by 
	\begin{equation*}
		\bm{Q}(r,z,0)=(1,0,0,1.70303\times 10^{-4})^\top.
	\end{equation*}
	The relativistic jet beam has a velocity $v_b=0.99$, density $\rho_b=0.01$, and pressure $p_b=1.70303\times 10^{-4}$. The jet is injected through the inlet part $(r\leq 1)$ of the low-$z$  boundary. At the symmetric axis $r=0$, we use reflection conditions, and at the other parts of the boundaries, we impose outflow boundary conditions. It is worth noting that simulating such jets successfully is challenging because they contain ultra-relativistic regions, strong relativistic shock waves, shear flows, and interface instabilities. 
	
	We present in Figure \ref{fig:jet} the schlieren images of the density logarithm $\ln \rho$ 
	 along with the PCP limited cells, obtained using our PCP HWENO scheme at $t=60, 80, 100$. The images clearly demonstrate the formation of a bow shock by the moving jet. The discontinuity between the jet material and the initial static material gives rise to Kelvin--Helmholtz instabilities.  These results agree well with those reported in previous studies \cite{marti1997morphology, 2015High, zhang2006ram, chen2022physical}. 
	 One can see that the PCP limiter is necessary in this challenging test, while only a small portion of cells are limited during the simulations.

	\begin{figure}[!htb]
		\centering
		\subfloat[$\ln\rho$ at $t=60$]{
			\includegraphics[width=0.30\textwidth]{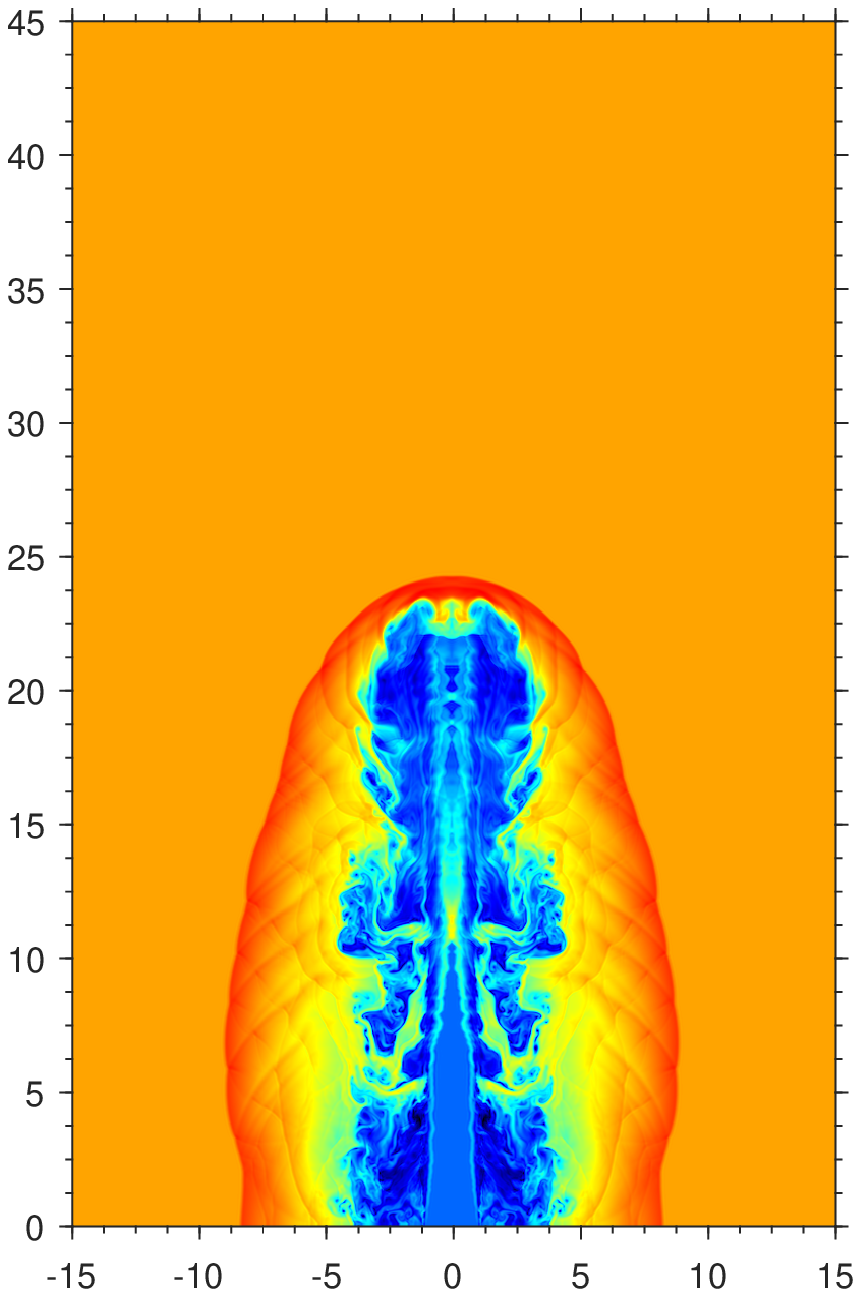}\label{fig:jet_rho60}
		}
		\subfloat[$\ln\rho$ at $t=80$]{
			\includegraphics[width=0.30\textwidth]{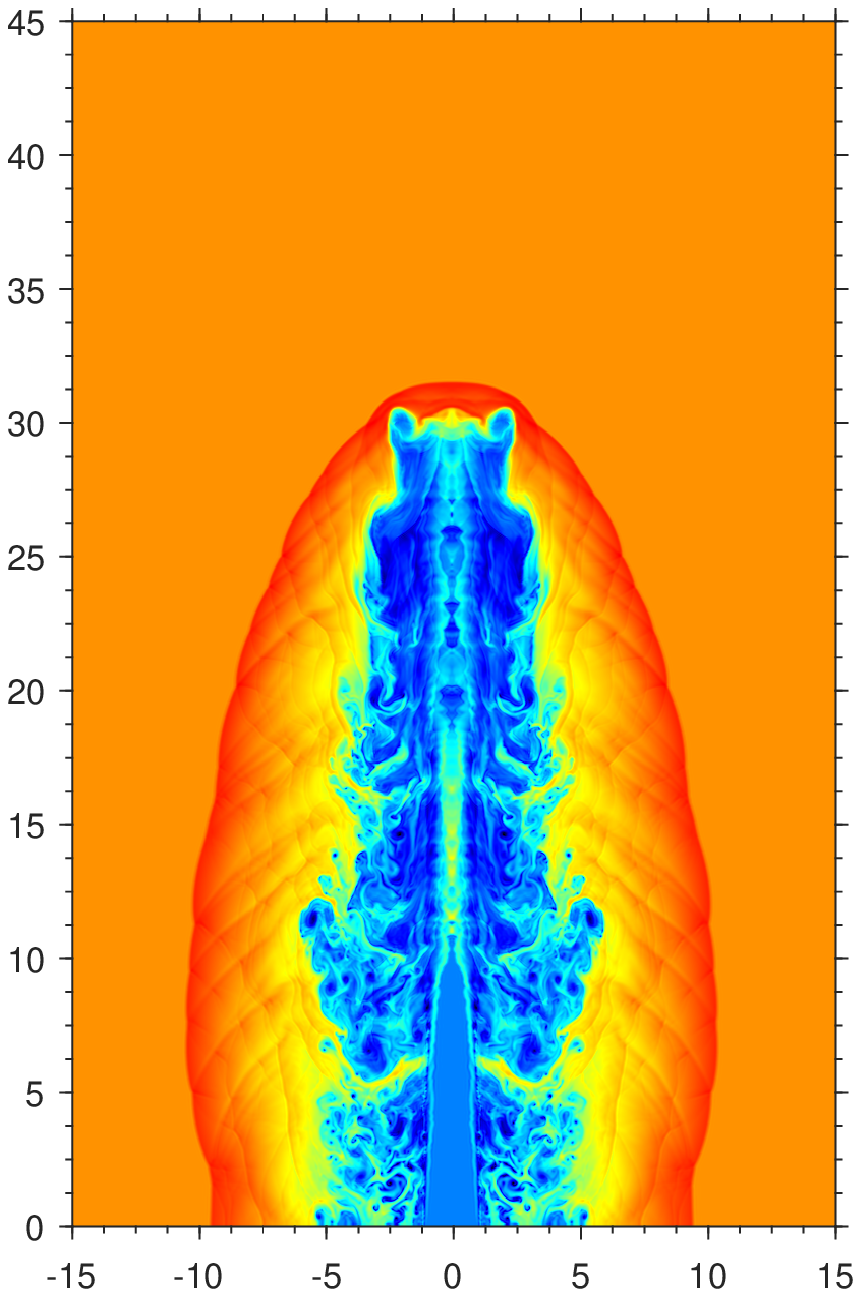}\label{fig:jet_rho80}
		}
		\subfloat[$\ln\rho$ at $t=100$]{
			\includegraphics[width=0.30\textwidth]{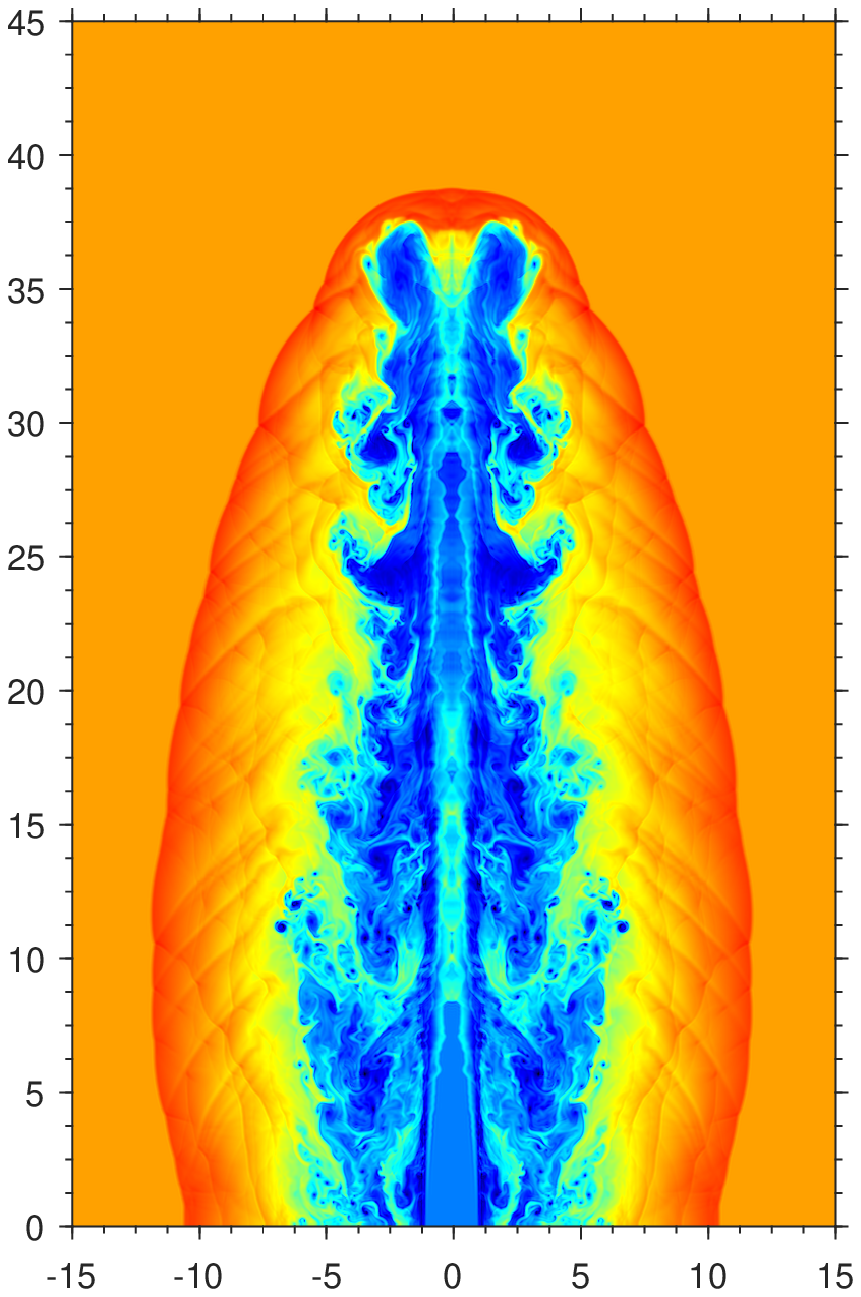}\label{fig:jet_rho100}
		}
	
		\subfloat[PCP limited cells at $t=60$]{
			\includegraphics[width=0.30\textwidth]{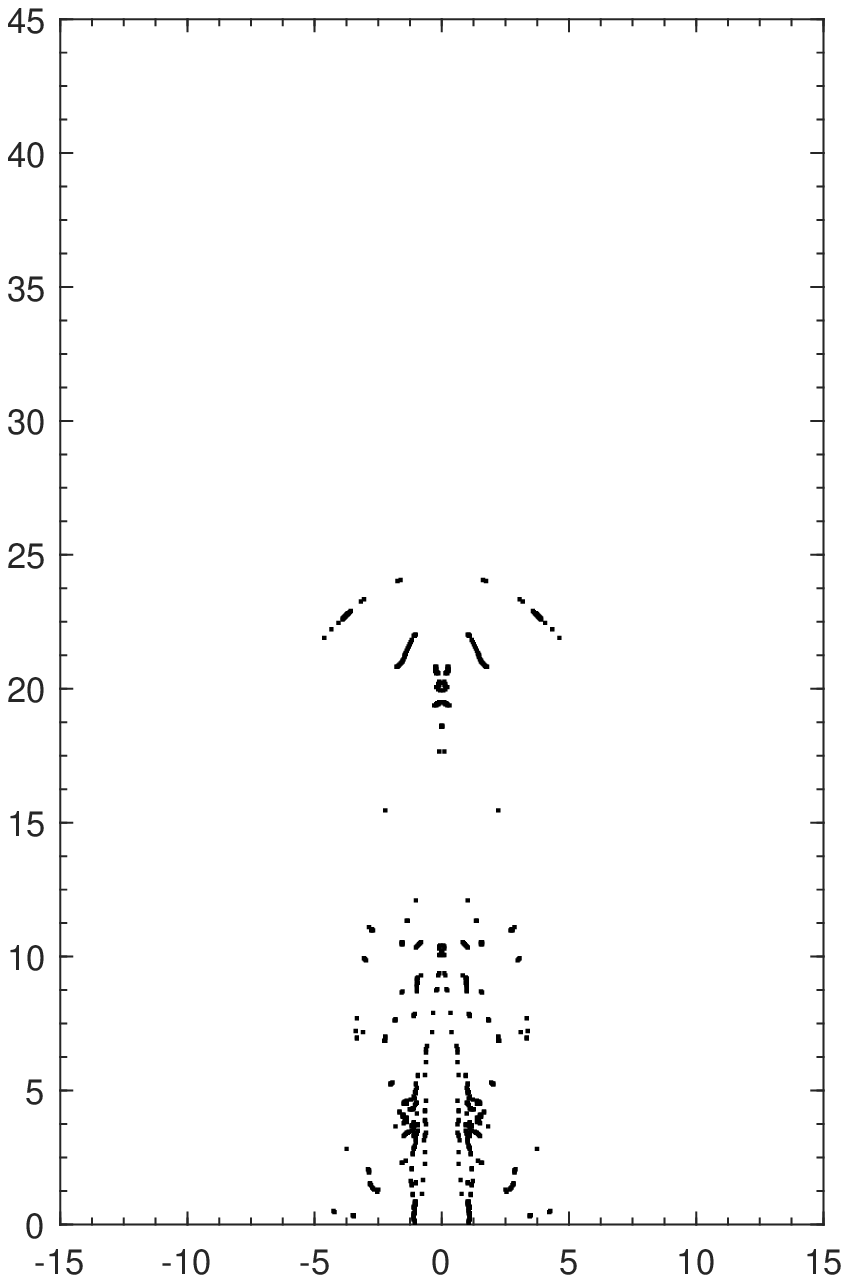}\label{fig:jet_tc60}
		}
		\subfloat[PCP limited cells at $t=80$]{
			\includegraphics[width=0.30\textwidth]{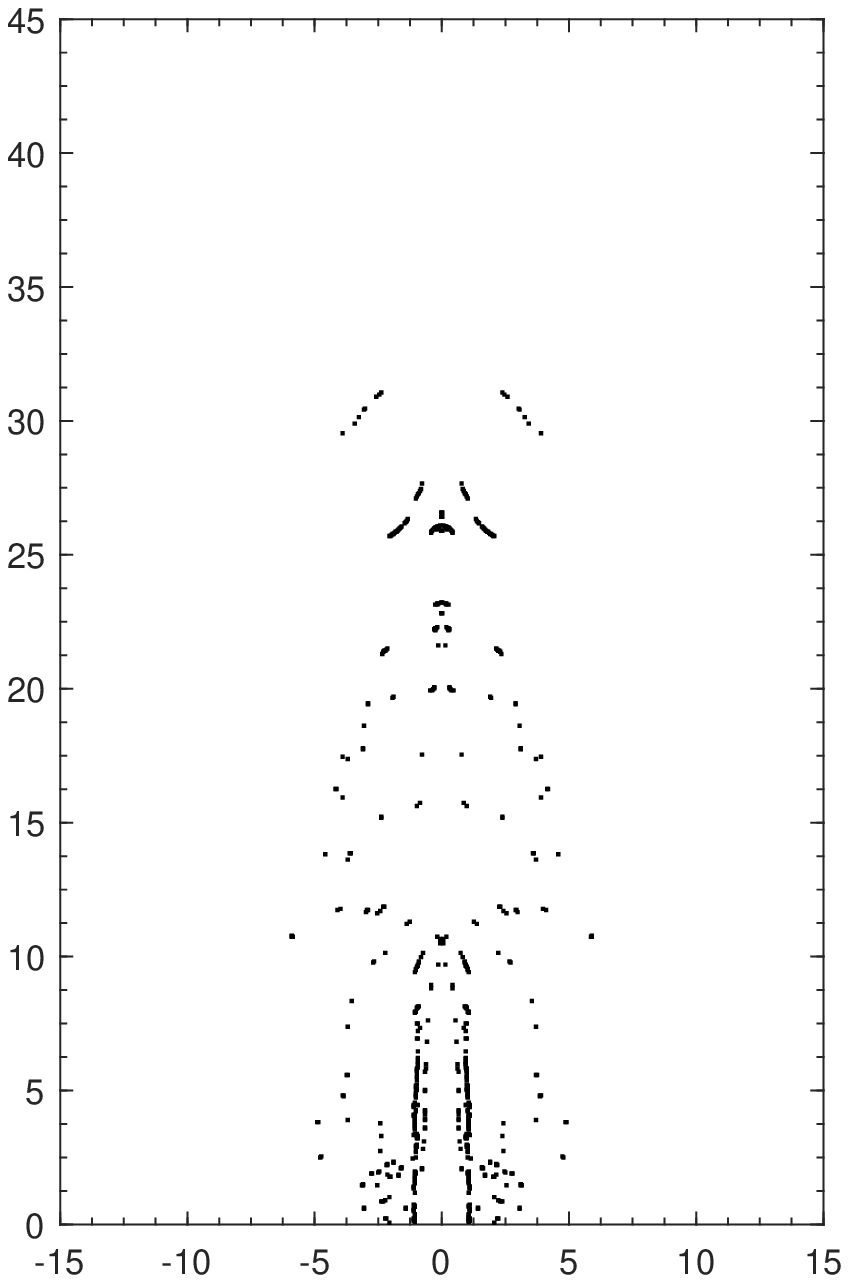}\label{fig:jet_tc80}
		}
		\subfloat[PCP limited cells at $t=100$]{
			\includegraphics[width=0.30\textwidth]{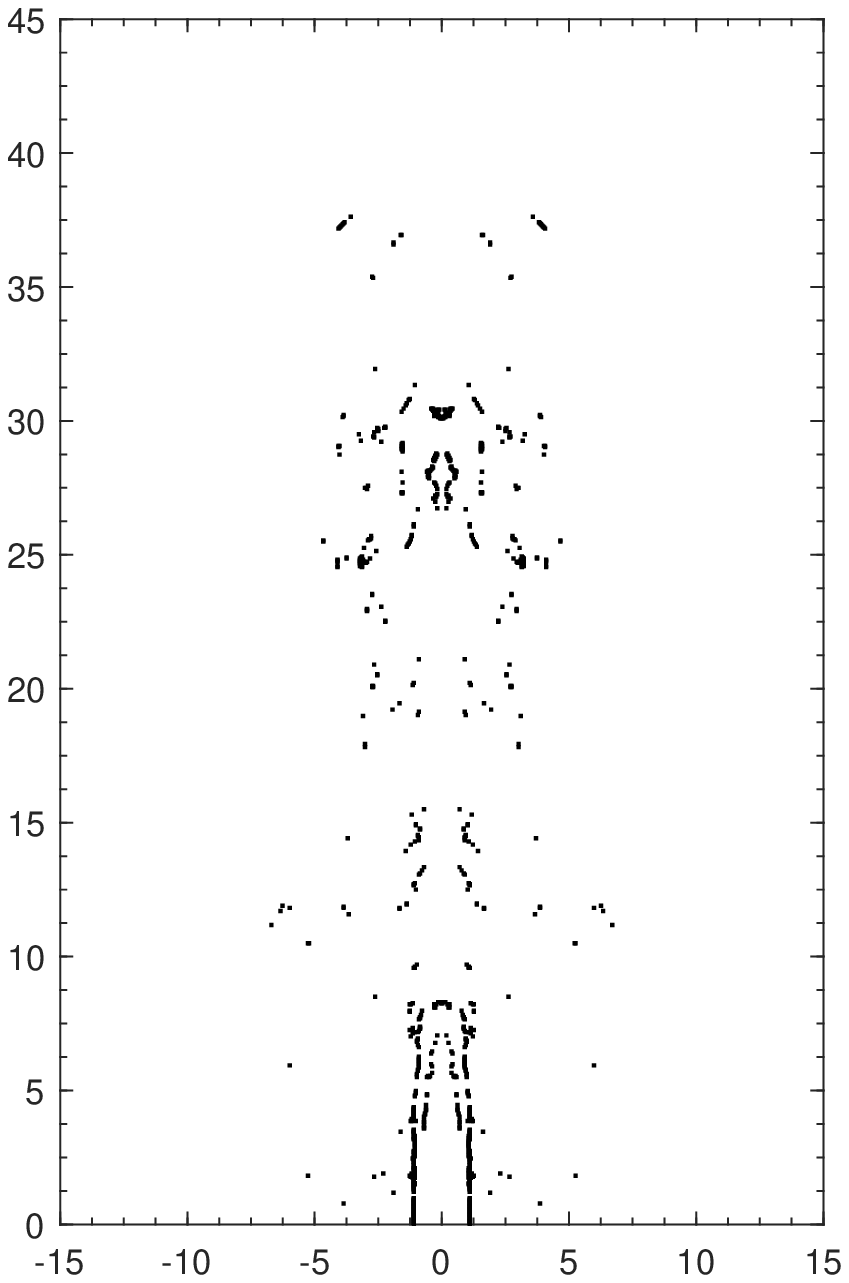}\label{fig:jet_tc100}
		}
		\captionsetup{font=small}
		\caption{Example \ref{test:jet}: The schlieren images of the density logarithm $\ln \rho$ and the PCP limited cells.}\label{fig:jet}
	\end{figure}

\end{exa}

\section{Conclusion}\label{sec:conclusion}

Designing genuinely PCP schemes is a challenging task, as relativistic effects make it difficult to reformulate primitive variables explicitly using conservative variables. 
This paper proposed three efficient NR methods for robustly recovering primitive variables from conservative variables, with applications to developing PCP finite volume HWENO schemes for relativistic hydrodynamics. 
Our rigorous analysis demonstrated that these NR methods are always convergent and PCP, meaning that they preserve the physical constraints throughout the NR iterations. The discovery of these robust NR methods and their convergence analysis are very nontrivial.  
The presented PCP HWENO schemes were built on the NR methods, high-order HWENO reconstruction, a PCP limiter, and strong-stability-preserving time discretization. We rigorously demonstrated the PCP property of our schemes using convex decomposition techniques. In addition, we proposed the characteristic decomposition approach with rescaled eigenvectors and scale-invariant nonlinear weights to improve the performance of the HWENO schemes in simulating large-scale RHD problems. 
Several challenging numerical examples were provided to evaluate the robustness, accuracy, and high resolution of our PCP HWENO schemes and to demonstrate the efficiency of the proposed NR methods.

\appendix

\section{2D Linear reconstruction operator $\bm{M}_L$}\label{app1}

For convenience, we number the night cells around the cell $I_{i,j}$, as shown in Figure \ref{FIG:2Dneighbour}. We denote $\xi(x)=\frac{x-x_i}{\Delta x}$, $\eta(y)=\frac{y-y_j}{\Delta y}$. Let us reconstruct a quartic polynomial $P_0(x,y)=\sum\limits_{s=0}^{4-r}\sum\limits_{r=0}^{4}a_{s,r}^0\xi(x)^s\eta(y)^r$ satisfying
\begin{eqnarray}\label{2Dlinear rec equ}
	\left\{
	\begin{aligned}
		&\frac{1}{\Delta x\Delta y}\int_{I_{i,j}^k}P_0(x,y)dxdy=u_k,\quad k=1,\dots,9,\\
		&\frac{1}{\Delta x\Delta y}\int_{I_{i,j}^5}P_0(x,y)\frac{x-x_i}{\Delta x}dxdy=v_5,\\
		&\frac{1}{\Delta x\Delta y}\int_{I_{i,j}^5}P_0(x,y)\frac{y-y_j}{\Delta y}dxdy=w_5
	\end{aligned}
	\right.
\end{eqnarray}
and minimizing
\begin{equation}\label{2Dlinear rec mini}
	\sqrt{\!\sum_{k=2,4,6,8}\!\left[\left(\frac{1}{\Delta x\Delta y}\int_{I_{i,j}^k}P_0(x,y)\frac{x-x_i}{\Delta x}dxdy\!-\!v_k\right)^2\!+\!\left(\frac{1}{\Delta x\Delta y}\int_{I_{i,j}^k}P_0(x,y)\frac{y-y_j}{\Delta x}dxdy\!-\!w_k\right)^2\right]},
\end{equation}
where $u_k$, $v_k$, $w_k$ are any given real numbers. The conditions \eqref{2Dlinear rec equ}--\eqref{2Dlinear rec mini} form a constrained least squares problem for the unknowns $\left\{a_{s,r}^0\right\}_{s,r=0}^4$. Solving this problem with the nullspace method (see \cite[Section 5.1.3]{bjorck1996numerical} for details) gives the expressions of $\{a_{s,r}^0\}$, which are the linear combinations of
\begin{equation*}
	\begin{bmatrix}
		u_7&u_8&u_9\\
		u_4&u_5&u_6\\
		u_1&u_2&u_3
	\end{bmatrix},~
	\begin{bmatrix}
		&v_8&\\
		v_4&v_5&v_6\\
		&v_2&
	\end{bmatrix},~\text{and }
	\begin{bmatrix}
		&w_8&\\
		w_4&w_5&w_6\\
		&w_2&
	\end{bmatrix}.
\end{equation*}
In order to save space, we here omit the specific expressions of $a_{s,r}^0$.
Define the operator
\begin{equation*}
	M_L([u_1,\dots,u_9],[v_2~v_4~v_5~v_6~v_8],[w_2~w_4~w_5~w_6~w_8],\xi,\eta):=P_0(x(\xi),y(\eta))=\sum\limits_{s=0}^{4-r}\sum\limits_{r=0}^{4}a_{s,r}^0\xi^s\eta^r
\end{equation*}
which is a mapping from $\mathbb{R}^{1\times9}\times\mathbb{R}^{1\times5}\times\mathbb{R}^{1\times5}\times\mathbb{R}\times\mathbb{R}$ to $\mathbb{R}$. 
Using this operator, it is convenient to compute the value of $P_0(x,y)$ at $(x_{i+\xi},y_{j+\eta})$ with
\begin{equation*}
	P_0(x_{i+\xi},y_{j+\eta})=M_L([u_1,\dots,u_9],[v_2~v_4~v_5~v_6~v_8],[w_2~w_4~w_5~w_6~w_8],\xi,\eta).
\end{equation*}

The operator $M_L$ represents the reconstruction mapping for the scalar equation. In order to extend the reconstruction to the 2D RHD equations, we generalize it to vector cases component-wisely as follows
\begin{equation*}
	\begin{aligned}
		&\bm{M}_L\left(\left[\bm{U}_1,\dots,\bm{U}_9\right],\left[\bm{V}_2~\bm{V}_4~\bm{V}_5~\bm{V}_6~\bm{V}_8\right],\left[\bm{W}_2~\bm{W}_4~\bm{W}_5~\bm{W}_6~\bm{W}_8\right],\xi,\eta\right)\\
		:=&\begin{pmatrix}
			M_L\left(\left[{U}_1^{(1)},\dots,{U}_9^{(1)}\right],\left[{V}_2^{(1)}~{V}_4^{(1)}~{V}_5^{(1)}~{V}_6^{(1)}~{V}_8^{(1)}\right],\left[{W}_2^{(1)}~{W}_4^{(1)}~{W}_5^{(1)}~{W}_6^{(1)}~{W}_8^{(1)}\right],\xi,\eta\right)
			\\
			M_L\left(\left[{U}_1^{(2)},\dots,{U}_9^{(2)}\right],\left[{V}_2^{(2)}~{V}_4^{(2)}~{V}_5^{(2)}~{V}_6^{(2)}~{V}_8^{(2)}\right],\left[{W}_2^{(2)}~{W}_4^{(2)}~{W}_5^{(2)}~{W}_6^{(2)}~{W}_8^{(2)}\right],\xi,\eta\right)
			\\
			M_L\left(\left[{U}_1^{(3)},\dots,{U}_9^{(3)}\right],\left[{V}_2^{(3)}~{V}_4^{(3)}~{V}_5^{(3)}~{V}_6^{(3)}~{V}_8^{(3)}\right],\left[{W}_2^{(3)}~{W}_4^{(3)}~{W}_5^{(3)}~{W}_6^{(3)}~{W}_8^{(3)}\right],\xi,\eta\right)
			\\
			M_L\left(\left[{U}_1^{(4)},\dots,{U}_9^{(4)}\right],\left[{V}_2^{(4)}~{V}_4^{(4)}~{V}_5^{(4)}~{V}_6^{(4)}~{V}_8^{(4)}\right],\left[{W}_2^{(4)}~{W}_4^{(4)}~{W}_5^{(4)}~{W}_6^{(4)}~{W}_8^{(4)}\right],\xi,\eta\right)
		\end{pmatrix},
	\end{aligned}
\end{equation*}
%\begin{equation*}
%	\bm{M}_L\left(
%	\left[
%	\begin{aligned}
	%		&\bm{u}_1\\
	%		&\bm{u}_2\\
	%		&\bm{u}_3\\
	%		&\bm{u}_4
	%	\end{aligned}
%	\right],
%	\left[
%	\begin{aligned}
	%		&\bm{v}_1\\
	%		&\bm{v}_2\\
	%		&\bm{v}_3\\
	%		&\bm{v}_4
	%	\end{aligned}
%	\right],
%	\left[
%	\begin{aligned}
	%		&\bm{w}_1\\
	%		&\bm{w}_2\\
	%		&\bm{w}_3\\
	%		&\bm{w}_4
	%	\end{aligned}
%	\right]
%	,\xi,\eta\right)
%	:=\left[
%	\begin{aligned}
	%		&M_L(\bm{u}_1,\bm{v}_1,\bm{w}_1,\xi,\eta)\\
	%		&M_L(\bm{u}_2,\bm{v}_2,\bm{w}_2,\xi,\eta)\\
	%		&M_L(\bm{u}_3,\bm{v}_3,\bm{w}_3,\xi,\eta)\\
	%		&M_L(\bm{u}_4,\bm{v}_4,\bm{w}_4,\xi,\eta)
	%	\end{aligned}
%	\right],
%\end{equation*}
where $\bm{M}_L$ is the reconstruction operator from $\mathbb{R}^{4\times9}\times\mathbb{R}^{4\times5}\times\mathbb{R}^{4\times5}\times\mathbb{R}\times\mathbb{R}$ to $\mathbb{R}^{4\times1}$. It is worth pointing out that $\bm{M}_L(\cdot,\cdot,\cdot,\xi,\eta)$ is a linear mapping for fixed $\xi$ and $\eta$,

\section{2D HWENO reconstruction operator $\bm{M}_H$}\label{app2}

Reconstruct four quadratic polynomials $P_n(x,y):=\sum\limits_{s=0}^{2-r}\sum\limits_{r=0}^{2}a_{s,r}^r\xi(x)^s\eta(y)^r$ $(r=1,2,3,4)$ satisfying 
\begin{eqnarray}\label{2Dlowlinear rec equ}
	\left\{
	\begin{aligned}
		&\frac{1}{\Delta x\Delta y}\int_{I_{i,j}^k}P_1(x,y)dxdy=u_k,\quad k=1,2,4,5,\\
		&\frac{1}{\Delta x\Delta y}\int_{I_{i,j}^k}P_2(x,y)dxdy=u_k,\quad k=2,3,5,6,\\
		&\frac{1}{\Delta x\Delta y}\int_{I_{i,j}^k}P_3(x,y)dxdy=u_k,\quad k=4,5,7,8,\\
		&\frac{1}{\Delta x\Delta y}\int_{I_{i,j}^k}P_4(x,y)dxdy=u_k,\quad k=5,6,8,9,\\
		&\frac{1}{\Delta x\Delta y}\int_{I_{i,j}^5}P_n(x,y)\frac{x-x_i}{\Delta x}dxdy=v_5,\quad n=1,2,3,4,\\
		&\frac{1}{\Delta x\Delta y}\int_{I_{i,j}^5}P_n(x,y)\frac{y-y_j}{\Delta y}dxdy=w_5,\quad n=1,2,3,4.
	\end{aligned}
	\right.
\end{eqnarray}
Similarly, we can obtain the expressions of $a_{s,r}^n$ ($n=1,2,3,4$, $s,r=1,2$), which are linear combinations of 
\begin{equation*}
	\begin{bmatrix}
		u_7&u_8&u_9\\
		u_4&u_5&u_6\\
		u_1&u_2&u_3
	\end{bmatrix},~
	v_5,~\text{and }
	w_5,
\end{equation*}

Next, in order to measure the smoothness of the polynomial $P_n\left ( x,y\right )$ in the cell $I_{i,j}$, we calculate the smooth indicators, with the same definition as in \cite{zhao2020hermite},
\begin{equation*}\label{smooth indicator2D}
	\beta_n=\sum_{|l|=1}^{r}|I_{i,j}|^{|l|-1}\int_{I_i,j}\left (\frac{\partial^{|l|}}{\partial x^{l_1}\partial y^{l_2}}P_n(x,y)\right)^2{\mathrm{d}x\mathrm{d}y},~n=0,\dots,4,
\end{equation*}
where $r$ is the degree of the polynomials $P_n(x,y)$. Then the HWENO reconstruction polynomial is defined by
\begin{equation}
	P_{H}\left(x,y\right)=\omega_0 \left(\frac{1}{\gamma_0}P_0\left(x,y\right)-\sum_{n=1}^4\frac{\gamma_n}{\gamma_0}P_n\left(x,y\right)\right)+\sum_{n=1}^4\omega_nP_n\left(x,y\right), \nonumber
\end{equation}
\noindent where the nonlinear weights
\begin{equation}\label{nonlinear weights 2_2D}
	\omega_n=\frac{\bar{\omega}_n}{\sum_{k=0}^4\bar{\omega}_k}\quad \text{with} \quad \bar{\omega}_n=\gamma_n\left(1+\frac{\tau^2}{\left(\beta_n\right)^2+\epsilon}\right), n=0,\dots,4,
\end{equation}
and $\tau:=\left(\frac{\sum_{n=0}^{4}|\beta_0-\beta_n|}{4}\right)$. Similar to the 1D case, these nonlinear weights possess the ``scaling-invariant'' property. 

Define the operator
\begin{equation*}
	M_H([u_1,\dots,u_9],[v_2~v_4~v_5~v_6~v_8],[w_2~w_4~w_5~w_6~w_8],\xi,\eta):=P_H(x(\xi),y(\eta)),
\end{equation*} 
which is a mapping from $\mathbb{R}^{1\times9}\times\mathbb{R}^{1\times5}\times\mathbb{R}^{1\times5}\times\mathbb{R}\times\mathbb{R}$ to $\mathbb{R}$. It is easy to compute the value of $P_H(x,y)$ at $(x_{i+\xi},y_{j+\eta})$ with
\begin{equation*}
	P_H(x_{i+\xi},y_{j+\eta})=M_H([u_1,\dots,u_9],[v_2~v_4~v_5~v_6~v_8],[w_2~w_4~w_5~w_6~w_8],\xi,\eta).
\end{equation*}
We can generalize the scalar HWENO reconstruction operator $M_H$ to the vector cases in a component by component manner:
\begin{equation*}
	\begin{aligned}
		&\bm{M}_H\left(\left[\bm{U}_1,\dots,\bm{U}_9\right],\left[\bm{V}_2~\bm{V}_4~\bm{V}_5~\bm{V}_6~\bm{V}_8\right],\left[\bm{W}_2~\bm{W}_4~\bm{W}_5~\bm{W}_6~\bm{W}_8\right],\xi,\eta\right)\\
		:=&\begin{pmatrix}
			M_H\left(\left[{U}_1^{(1)},\dots,{U}_9^{(1)}\right],\left[{V}_2^{(1)}~{V}_4^{(1)}~{V}_5^{(1)}~{V}_6^{(1)}~{V}_8^{(1)}\right],\left[{W}_2^{(1)}~{W}_4^{(1)}~{W}_5^{(1)}~{W}_6^{(1)}~{W}_8^{(1)}\right],\xi,\eta\right)
			\\
			M_H\left(\left[{U}_1^{(2)},\dots,{U}_9^{(2)}\right],\left[{V}_2^{(2)}~{V}_4^{(2)}~{V}_5^{(2)}~{V}_6^{(2)}~{V}_8^{(2)}\right],\left[{W}_2^{(2)}~{W}_4^{(2)}~{W}_5^{(2)}~{W}_6^{(2)}~{W}_8^{(2)}\right],\xi,\eta\right)
			\\
			M_H\left(\left[{U}_1^{(3)},\dots,{U}_9^{(3)}\right],\left[{V}_2^{(3)}~{V}_4^{(3)}~{V}_5^{(3)}~{V}_6^{(3)}~{V}_8^{(3)}\right],\left[{W}_2^{(3)}~{W}_4^{(3)}~{W}_5^{(3)}~{W}_6^{(3)}~{W}_8^{(3)}\right],\xi,\eta\right)
			\\
			M_H\left(\left[{U}_1^{(4)},\dots,{U}_9^{(4)}\right],\left[{V}_2^{(4)}~{V}_4^{(4)}~{V}_5^{(4)}~{V}_6^{(4)}~{V}_8^{(4)}\right],\left[{W}_2^{(4)}~{W}_4^{(4)}~{W}_5^{(4)}~{W}_6^{(4)}~{W}_8^{(4)}\right],\xi,\eta\right)
		\end{pmatrix},
	\end{aligned}
\end{equation*}
%\begin{equation*}
%	\bm{M}_H\left(
%	\left[
%	\begin{aligned}
	%		&\bm{u}_1\\
	%		&\bm{u}_2\\
	%		&\bm{u}_3\\
	%		&\bm{u}_4
	%	\end{aligned}
%	\right],
%	\left[
%	\begin{aligned}
	%		&\bm{v}_1\\
	%		&\bm{v}_2\\
	%		&\bm{v}_3\\
	%		&\bm{v}_4
	%	\end{aligned}
%	\right],
%	\left[
%	\begin{aligned}
	%		&\bm{w}_1\\
	%		&\bm{w}_2\\
	%		&\bm{w}_3\\
	%		&\bm{w}_4
	%	\end{aligned}
%	\right]
%	,\xi,\eta\right)
%	:=\left[
%	\begin{aligned}
	%		&M_H(\bm{u}_1,\bm{v}_1,\bm{w}_1,\xi,\eta)\\
	%		&M_H(\bm{u}_2,\bm{v}_2,\bm{w}_2,\xi,\eta)\\
	%		&M_H(\bm{u}_3,\bm{v}_3,\bm{w}_3,\xi,\eta)\\
	%		&M_H(\bm{u}_4,\bm{v}_4,\bm{w}_4,\xi,\eta)
	%	\end{aligned}
%	\right],
%\end{equation*}
where ${U}_{i,j}^{(\ell)}$ is the $\ell$th component of ${\bm U}_{i,j}$, ${V}_{i,j}^{(\ell)}$ is the $\ell$th component of ${\bm V}_{i,j}$, ${W}_{i,j}^{(\ell)}$ is the $\ell$th component of ${\bm W}_{i,j}$. Different from $\bm{M}_L$, the operator  
$\bm{M}_H(\cdot,\dots,\cdot,\xi,\eta)$ is a nonlinear mapping for fixed $\xi$ and $\eta$.

%\appendix
%\appendixpage
%\addappheadtotoc
%\input{append}

\bibliographystyle{abbrv}
\bibliography{ref}

\begin{thebibliography}{10}

\bibitem{balsara2016subluminal}
D.~S. Balsara and J.~Kim.
\newblock {A subluminal relativistic magnetohydrodynamics scheme with ADER-WENO
  predictor and multidimensional Riemann solver-based corrector}.
\newblock {\em Journal of Computational Physics}, 312:357--384, 2016.

\bibitem{bjorck1996numerical}
{\AA}.~Björck.
\newblock {\em Numerical methods for least squares problems}.
\newblock Society for Industrial and Applied Mathematics, 1996.

\bibitem{2011High}
M.~Castro, B.~Costa, and W.~S. Don.
\newblock {High order weighted essentially non-oscillatory WENO-Z schemes for
  hyperbolic conservation laws}.
\newblock {\em Journal of Computational Physics}, 230(5):1766--1792, 2011.

\bibitem{chen2022physical}
Y.~Chen and K.~Wu.
\newblock {A physical-constraint-preserving finite volume WENO method for
  special relativistic hydrodynamics on unstructured meshes}.
\newblock {\em Journal of Computational Physics}, 466:111398, 2022.

\bibitem{costa2007multi}
B.~Costa and W.~S. Don.
\newblock {Multi-domain hybrid spectral-WENO methods for hyperbolic
  conservation laws}.
\newblock {\em Journal of Computational Physics}, 224(2):970--991, 2007.

\bibitem{dolezal1995relativistic}
A.~Dolezal and S.~Wong.
\newblock Relativistic hydrodynamics and essentially non-oscillatory shock
  capturing schemes.
\newblock {\em Journal of Computational Physics}, 120(2):266--277, 1995.

\bibitem{duan2019high}
J.~Duan and H.~Tang.
\newblock {High-Order accurate entropy stable finite difference schemes for
  one- and two-dimensional special relativistic hydrodynamics}.
\newblock {\em Advances in Applied Mathematics and Mechanics}, 12(1):1--29,
  2019.

\bibitem{dunaway1972some}
D.~K. Dunaway and B.~L. Turlington.
\newblock Some major modifications to a new method for solving ill-conditioned
  polynomial equations.
\newblock In {\em Proceedings of the ACM Annual Conference-Volume 2}, pages
  636--643, 1972.

\bibitem{eulderink1994general}
F.~Eulderink and G.~Mellema.
\newblock {{General relativistic hydrodynamics with a Roe solver}}.
\newblock {\em Astron. Astrophys. Suppl. Ser.}, 110:587, 1995.

\bibitem{flocke2015algorithm}
N.~Flocke.
\newblock {Algorithm 954: An accurate and efficient cubic and quartic equation
  solver for physical applications}.
\newblock {\em ACM Transactions on Mathematical Software (TOMS)}, 41(4):1--24,
  2015.

\bibitem{font2008numerical}
J.~A. Font.
\newblock Numerical hydrodynamics and magnetohydrodynamics in general
  relativity.
\newblock {\em Living Reviews in Relativity}, 11(1):1--131, 2008.

\bibitem{harten1983high}
A.~Harten.
\newblock High resolution schemes for hyperbolic conservation laws.
\newblock {\em Journal of Computational Physics}, 49:357--393, 1983.

\bibitem{harten1987preliminary}
A.~Harten.
\newblock {Preliminary results on the extension of ENO schemes to
  two-dimensional problems}.
\newblock In {\em Nonlinear Hyperbolic Problems}, volume 1270, pages 23--40,
  1987.

\bibitem{harten1987uniformly}
A.~Harten, B.~Engquist, S.~Osher, and S.~R. Chakravarthy.
\newblock {Uniformly high order accurate essentially non-oscillatory schemes,
  III}.
\newblock {\em Journal of Computational Physics}, 71(2):231--303, 1987.

\bibitem{he2012adaptive}
P.~He and H.~Tang.
\newblock An adaptive moving mesh method for two-dimensional relativistic
  hydrodynamics.
\newblock {\em Communications in Computational Physics}, 11(1):114--146, 2012.

\bibitem{1999Weighted}
C.~Hu and C.-W. Shu.
\newblock {Weighted essentially non-oscillatory schemes on triangular meshes}.
\newblock {\em Journal of Computational Physics}, 150(1):97--127, 1999.

\bibitem{hu2013positivity}
X.~Y. Hu, N.~A. Adams, and C.-W. Shu.
\newblock {Positivity-preserving method for high-order conservative schemes
  solving compressible Euler equations}.
\newblock {\em Journal of Computational Physics}, 242:169--180, 2013.

\bibitem{jiang1996efficient}
G.-S. Jiang and C.-W. Shu.
\newblock {Efficient implementation of weighted ENO schemes}.
\newblock {\em Journal of Computational Physics}, 126(1):202--228, 1996.

\bibitem{kidder2017spectre}
L.~E. Kidder, S.~E. Field, F.~Foucart, E.~Schnetter, S.~A. Teukolsky, A.~Bohn,
  N.~Deppe, P.~Diener, F.~H{\'e}bert, J.~Lippuner, et~al.
\newblock {SpECTRE: A task-based discontinuous Galerkin code for relativistic
  astrophysics}.
\newblock {\em Journal of Computational Physics}, 335:84--114, 2017.

\bibitem{krivodonova2004shock}
L.~Krivodonova, J.~Xin, J.-F. Remacle, N.~Chevaugeon, and J.~E. Flaherty.
\newblock {Shock detection and limiting with discontinuous Galerkin methods for
  hyperbolic conservation laws}.
\newblock {\em Applied Numerical Mathematics}, 48(3-4):323--338, 2004.

\bibitem{levy1999central}
D.~Levy, G.~Puppo, and G.~Russo.
\newblock {Central WENO schemes for hyperbolic systems of conservation laws}.
\newblock {\em ESAIM: Mathematical Modelling and Numerical Analysis},
  33(3):547--571, 1999.

\bibitem{li2010hybrid}
G.~Li and J.~Qiu.
\newblock Hybrid weighted essentially non-oscillatory schemes with different
  indicators.
\newblock {\em Journal of Computational Physics}, 229(21):8105--8129, 2010.

\bibitem{ling2019physical}
D.~Ling, J.~Duan, and H.~Tang.
\newblock {Physical-constraints-preserving Lagrangian finite volume schemes for
  one-and two-dimensional special relativistic hydrodynamics}.
\newblock {\em Journal of Computational Physics}, 396:507--543, 2019.

\bibitem{liu1994weighted}
X.-D. Liu, S.~Osher, and T.~Chan.
\newblock Weighted essentially non-oscillatory schemes.
\newblock {\em Journal of Computational Physics}, 115(1):200--212, 1994.

\bibitem{marti2003numerical}
J.~M. Mart{\'\i} and E.~M{\"u}ller.
\newblock Numerical hydrodynamics in special relativity.
\newblock {\em Living Reviews in Relativity}, 6(1):1--100, 2003.

\bibitem{marti2015grid}
J.~M. Mart{\'\i} and E.~M{\"u}ller.
\newblock Grid-based methods in relativistic hydrodynamics and
  magnetohydrodynamics.
\newblock {\em Living Reviews in Computational Astrophysics}, 1(1):1--182,
  2015.

\bibitem{marti1997morphology}
J.~M. Mart{\'\i}, E.~M{\"u}ller, J.~Font, J.~M.~Z. Ib{\'a}{\~n}ez, and
  A.~Marquina.
\newblock {Morphology and dynamics of relativistic jets}.
\newblock {\em The Astrophysical Journal}, 479(1):151, 1997.

\bibitem{mignone2005hllc}
A.~Mignone and G.~Bodo.
\newblock {An HLLC Riemann solver for relativistic flows—I. Hydrodynamics}.
\newblock {\em Monthly Notices of the Royal Astronomical Society},
  364(1):126--136, 2005.

\bibitem{pao1981numerical}
S.~Pao and M.~Salas.
\newblock A numerical study of two-dimensional shock vortex interaction.
\newblock In {\em 14th Fluid and Plasma Dynamics Conference}, page 1205, 1981.

\bibitem{qin2016bound}
T.~Qin, C.-W. Shu, and Y.~Yang.
\newblock {Bound-preserving discontinuous Galerkin methods for relativistic
  hydrodynamics}.
\newblock {\em Journal of Computational Physics}, 315:323--347, 2016.

\bibitem{qiu2004hermite}
J.~Qiu and C.-W. Shu.
\newblock {Hermite WENO schemes and their application as limiters for
  Runge--Kutta discontinuous Galerkin method: one-dimensional case}.
\newblock {\em Journal of Computational Physics}, 193(1):115--135, 2004.

\bibitem{qiu2005hermite}
J.~Qiu and C.-W. Shu.
\newblock {Hermite WENO schemes and their application as limiters for
  Runge–Kutta discontinuous Galerkin method II: two dimensional case}.
\newblock {\em Computers $\&$ Fluids}, 34(6):642--663, 2005.

\bibitem{radice2011discontinuous}
D.~Radice and L.~Rezzolla.
\newblock {Discontinuous Galerkin methods for general-relativistic
  hydrodynamics: Formulation and application to spherically symmetric
  spacetimes}.
\newblock {\em Physical Review D}, 84(2):024010, 2011.

\bibitem{radice2012thc}
D.~Radice and L.~Rezzolla.
\newblock {THC: a new high-order finite-difference high-resolution
  shock-capturing code for special-relativistic hydrodynamics}.
\newblock {\em Astronomy $\&$ Astrophysics}, 547:A26, 2012.

\bibitem{radice2014high}
D.~Radice, L.~Rezzolla, and F.~Galeazzi.
\newblock High-order fully general-relativistic hydrodynamics: new approaches
  and tests.
\newblock {\em Classical and Quantum Gravity}, 31(7):075012, 2014.

\bibitem{riccardi2008primitive}
G.~Riccardi and D.~Durante.
\newblock {Primitive variable recovering in special relativistic hydrodynamics
  allowing ultra-relativistic flows}.
\newblock In {\em International Mathematical Forum}, volume~42, pages
  2081--2111, 2008.

\bibitem{ryu2006equation}
D.~Ryu, I.~Chattopadhyay, and E.~Choi.
\newblock Equation of state in numerical relativistic hydrodynamics.
\newblock {\em The Astrophysical Journal Supplement Series}, 166(1):410, 2006.

\bibitem{shi2002technique}
J.~Shi, C.~Hu, and C.-W. Shu.
\newblock {A technique of treating negative weights in WENO schemes}.
\newblock {\em Journal of Computational Physics}, 175(1):108--127, 2002.

\bibitem{shu2016bound}
C.-W. Shu.
\newblock {Bound-preserving high-order schemes for hyperbolic equations: survey
  and recent developments}.
\newblock In {\em XVI International Conference on Hyperbolic Problems: Theory,
  Numerics, Applications}, pages 591--603, 2016.

\bibitem{tchekhovskoy2007wham}
A.~Tchekhovskoy, J.~C. McKinney, and R.~Narayan.
\newblock {WHAM: a WENO-based general relativistic numerical scheme--I.
  Hydrodynamics}.
\newblock {\em Monthly Notices of the Royal Astronomical Society},
  379(2):469--497, 2007.

\bibitem{teukolsky2016formulation}
S.~A. Teukolsky.
\newblock {Formulation of discontinuous Galerkin methods for relativistic
  astrophysics}.
\newblock {\em Journal of Computational Physics}, 312:333--356, 2016.

\bibitem{wilson2007relativistic}
J.~R. Wilson and G.~J. Mathews.
\newblock {\em Relativistic numerical hydrodynamics}.
\newblock Cambridge University Press, 2003.

\bibitem{wu2017design}
K.~Wu.
\newblock Design of provably physical-constraint-preserving methods for general
  relativistic hydrodynamics.
\newblock {\em Physical Review D}, 95(10):103001, 2017.

\bibitem{wu2018positivity}
K.~Wu.
\newblock Positivity-preserving analysis of numerical schemes for ideal
  magnetohydrodynamics.
\newblock {\em SIAM Journal on Numerical Analysis}, 56(4):2124--2147, 2018.

\bibitem{wu2021minimum}
K.~Wu.
\newblock Minimum principle on specific entropy and high-order accurate
  invariant-region-preserving numerical methods for relativistic hydrodynamics.
\newblock {\em SIAM Journal on Scientific Computing}, 43(6):B1164--B1197, 2021.

\bibitem{wu2021provably}
K.~Wu and C.-W. Shu.
\newblock {Provably physical-constraint-preserving discontinuous Galerkin
  methods for multidimensional relativistic MHD equations}.
\newblock {\em Numerische Mathematik}, 148(3):699--741, 2021.

\bibitem{WuShu2021GQL}
K.~Wu and C.-W. Shu.
\newblock Geometric quasilinearization framework for analysis and design of
  bound-preserving schemes.
\newblock {\em SIAM Review}, in press, 2022.

\bibitem{2015High}
K.~Wu and H.~Tang.
\newblock {High-order accurate physical-constraints-preserving finite
  difference WENO schemes for special relativistic hydrodynamics}.
\newblock {\em Journal of Computational Physics}, 298:539--564, 2015.

\bibitem{wu2016physical}
K.~Wu and H.~Tang.
\newblock {Physical-constraint-preserving central discontinuous Galerkin
  methods for special relativistic hydrodynamics with a general equation of
  state}.
\newblock {\em The Astrophysical Journal Supplement Series}, 228(1):3, 2016.

\bibitem{wu2017admissible}
K.~Wu and H.~Tang.
\newblock Admissible states and physical-constraints-preserving schemes for
  relativistic magnetohydrodynamic equations.
\newblock {\em Mathematical Models and Methods in Applied Sciences},
  27(10):1871--1928, 2017.

\bibitem{xiong2016parametrized}
T.~Xiong, J.-M. Qiu, and Z.~Xu.
\newblock {Parametrized positivity preserving flux limiters for the high order
  finite difference WENO scheme solving compressible Euler equations}.
\newblock {\em Journal of Scientific Computing}, 67(3):1066--1088, 2016.

\bibitem{xu2014parametrized}
Z.~Xu.
\newblock Parametrized maximum principle preserving flux limiters for high
  order schemes solving hyperbolic conservation laws: one-dimensional scalar
  problem.
\newblock {\em Mathematics of Computation}, 83(289):2213--2238, 2014.

\bibitem{xu2017bound}
Z.~Xu and X.~Zhang.
\newblock {Bound-preserving high-order schemes}.
\newblock In {\em Handbook of Numerical Analysis}, volume~18, chapter~4, pages
  81--102. Elsevier, 2017.

\bibitem{zhang2006ram}
W.~Zhang and A.~I. MacFadyen.
\newblock {RAM: a relativistic adaptive mesh refinement hydrodynamics code}.
\newblock {\em The Astrophysical Journal Supplement Series}, 164(1):255, 2006.

\bibitem{zhang2010maximum}
X.~Zhang and C.-W. Shu.
\newblock On maximum-principle-satisfying high order schemes for scalar
  conservation laws.
\newblock {\em Journal of Computational Physics}, 229(9):3091--3120, 2010.

\bibitem{zhang2010positivity}
X.~Zhang and C.-W. Shu.
\newblock {On positivity-preserving high order discontinuous Galerkin schemes
  for compressible Euler equations on rectangular meshes}.
\newblock {\em Journal of Computational Physics}, 229(23):8918--8934, 2010.

\bibitem{zhao2013runge}
J.~Zhao and H.~Tang.
\newblock {Runge--Kutta discontinuous Galerkin methods with WENO limiter for
  the special relativistic hydrodynamics}.
\newblock {\em Journal of Computational Physics}, 242:138--168, 2013.

\bibitem{zhao2020hybrid}
Z.~Zhao, Y.~Chen, and J.~Qiu.
\newblock {A hybrid Hermite WENO scheme for hyperbolic conservation laws}.
\newblock {\em Journal of Computational Physics}, 405:109175, 2020.

\bibitem{zhao2020hermite}
Z.~Zhao and J.~Qiu.
\newblock {A Hermite WENO scheme with artificial linear weights for hyperbolic
  conservation laws}.
\newblock {\em Journal of Computational Physics}, 417:109583, 2020.

\bibitem{zhao2019new}
Z.~Zhao, J.~Zhu, Y.~Chen, and J.~Qiu.
\newblock {A new hybrid WENO scheme for hyperbolic conservation laws}.
\newblock {\em Computers $\&$ Fluids}, 179:422--436, 2019.

\bibitem{zhu2016new}
J.~Zhu and J.~Qiu.
\newblock {A new fifth order finite difference WENO scheme for solving
  hyperbolic conservation laws}.
\newblock {\em Journal of Computational Physics}, 318:110--121, 2016.

\end{thebibliography}
\end{document}